\documentclass[11pt, leqno]{article}

\usepackage{verbatim}
\usepackage{hyperref}
\usepackage{amsmath}
\usepackage{enumerate, amsthm}
\usepackage{mathrsfs}
\usepackage{rotating}
\usepackage{xcolor}
\usepackage[margin=1.1 in]{geometry}
\usepackage[margin={1.5cm, 1.5cm},font=small, labelsep=endash]{caption}
\usepackage{mathtools}
\usepackage{tikz}
\usetikzlibrary{patterns}
\usetikzlibrary{arrows}
\usepackage{tikz-dimline}
\usepackage{ifthen}
\DeclarePairedDelimiter\floor{\lfloor}{\rfloor}

\usepackage{bm}

      \usepackage{amssymb}
      \usepackage{amsmath}
      \usepackage{centernot}

      \numberwithin{equation}{section}
      \theoremstyle{plain}
      \newtheorem{theorem}{Theorem}[section]

            \newtheorem{thm}[theorem]{Theorem}
      \newtheorem{lemma}[theorem]{Lemma}
      \newtheorem{lem}[theorem]{Lemma}
      \newtheorem{corollary}[theorem]{Corollary}
      \newtheorem{proposition}[theorem]{Proposition}
      \newtheorem{prop}[theorem]{Proposition}

            \newtheorem{claim}[theorem]{Claim}
      \theoremstyle{definition}

      \theoremstyle{remark}
      \newtheorem{remark}[theorem]{Remark}

\renewcommand{\P}{\mathbb P}
\newcommand{\R}{\mathbb R}
\newcommand{\E}{\mathbb E}
\newcommand{\Z}{\mathbb Z}
\newcommand{\N}{\mathbb N}
\newcommand{\mb}{\mathbf}
\renewcommand{\L}{\Lambda}

\newcommand{\lr}[4]{#3\xleftrightarrow[#1]{#2} #4}
     \newcommand{\nlr}[4]{#3\mathrel{\mathop{\centernot\longleftrightarrow}_{#1}^{#2}} #4}

\usepackage{constants}
\usepackage{titletoc}
\usepackage{stackrel}
\usepackage{etex}
\usepackage{enumitem}
\usepackage{enumerate}
\usepackage{vwcol}
\usepackage{stackengine}

\newcommand\couprad{10} 
\newcommand\rangeofdep{40}
\newcommand\pivrange{\the\numexpr\couprad + \rangeofdep\relax}

\usepackage{hyperref}
\hypersetup{linktocpage,
  colorlinks   = true,
  urlcolor     = blue,
  linkcolor    = blue,
  citecolor   = black
}

\newconstantfamily{c}{symbol=c}
\makeatother

\usepackage{titlesec}
\titleformat{\subsection}[runin]{\normalfont\bfseries}{\thesubsection.}{.5em}{}[.]\titlespacing{\subsection}{0pt}{2ex plus .1ex minus .2ex}{.8em}
\titleformat{\subsubsection}[runin]{\normalfont\bfseries}{\thesubsubsection.}{.5em}{}[.]
\titlespacing{\subsubsection}{0pt}{2ex plus .1ex minus .2ex}{.8em}

\usepackage{multicol}
\usepackage{parcolumns}

\usepackage{float}

\title{{\textbf{\normalsize{A CHARACTERIZATION OF STRONG PERCOLATION  VIA DISCONNECTION}}}}
\date{}
\begin{document}
	\thispagestyle{empty}
	\maketitle
	\vspace{0.1cm}
	\begin{center}
		\vspace{-1.7cm}
		Hugo Duminil-Copin$^{1,2}$, Subhajit Goswami$^3$, Pierre-Fran\c{c}ois Rodriguez$^4$,\\[1em] Franco Severo$^{5}$ and Augusto Teixeira$^6$

	\end{center}
	\vspace{0.1cm}
	\begin{abstract}
We consider a percolation model, the vacant set $\mathcal{V}^u$ of random interlacements on $\Z^d$, $d \geq 3$, in the regime of parameters $u>0$ in which it is strongly percolative. By definition, such values of $u$ pinpoint a robust subset of the super-critical phase, with strong quantitative controls on large local clusters. In the present work, we give a new charaterization of this regime in terms of a single property, monotone in $u$, involving a disconnection estimate for $\mathcal{V}^u$. A key aspect is to exhibit a gluing property for large local clusters from this information alone, and a major challenge in this undertaking is the fact that the conditional law of $\mathcal{V}^u$ exhibits degeneracies. As one of the main novelties of this work, the gluing technique we develop to merge large clusters accounts for such effects. In particular, our methods do not rely on the widely assumed finite-energy property, which the set $\mathcal{V}^u$ does not possess. The charaterization we derive plays a decisive role in the proof of a lasting conjecture regarding the coincidence of various critical parameters naturally associated to $\mathcal{V}^u$ in the companion article \cite{RI-I}.
\end{abstract}

	\vspace{1.5cm}
	
	\begin{flushleft}
		\thispagestyle{empty}
		\vspace{0.1cm}
		{\footnotesize
			\noindent\rule{6cm}{0.35pt} \hfill {\normalsize August 2023} \\[2em]
			
			\begin{multicols}{2}
				\small
				$^1$Institut des Hautes \'Etudes Scientifiques
				\\   35, route de Chartres \\ 91440 -- Bures-sur-Yvette, France.\\ \url{duminil@ihes.fr}\\[2em]
				
				$^2$Universit\'e de Gen\`eve\\
				Section de Math\'ematiques\\
				2-4 rue du Li\`evre \\
				1211 Gen\`eve 4, Switzerland.\\
				\url{hugo.duminil@unige.ch} \\[2em]
				
				$^3$School of Mathematics\\
				Tata Institute of Fundamental Research\\
				1, Homi Bhabha Road\\
				Colaba, Mumbai 400005, India. \\ \url{goswami@math.tifr.res.in} \columnbreak
				
				\hfill$^4$Imperial College London\\
				\hfill Department of Mathematics\\
				\hfill London SW7 2AZ \\
				\hfill United Kingdom.\\
				\hfill \url{p.rodriguez@imperial.ac.uk}  \\[2em]

				\hfill$^5$ETH Zurich\\
				\hfill Department of Mathematics\\
				\hfill R\"amistrasse 101\\
				\hfill 8092 Zurich, Switzerland.\\
				\hfill \url{franco.severo@math.ethz.ch}\\ [2em]
				
				\hfill $^6$Instituto de Matem\'atica Pura e Aplicada\\
				\hfill  Estrada dona Castorina, 110\\
				\hfill  22460-320, Rio de Janeiro - RJ, Brazil.\\
				\hfill  \url{augusto@impa.br}

			\end{multicols}
		}
	\end{flushleft}

\newpage

\setcounter{page}{1}

\section{Introduction}
\label{Sec:intro}
The study of the super-critical phase of percolation models, i.e.~the regime of parameters in which an infinite cluster exists, typically exhibits a subset, possibly strict, in which the model is \textit{strongly percolative}. We will soon give a precise meaning to this -- intuitively, strong percolation describes a robust percolative phase with good quantitative control on large local clusters. We consider this regime for a benchmark case of interest, the vacant set of random interlacements, one notable difficulty being that the model lacks `ellipticity' (cf.~for instance \eqref{eq:no-FE} below). For this model, a non-trivial strongly percolative regime is so far known to exist on $\Z^d$ for all $d \geq 3$ by already involved perturbative arguments, see \cite{Tei11} for $d \geq 5$ and \cite{MR3269990} for all $d \geq 3$, and our understanding of various features of vacant clusters in this regime has witnessed considerable progress over the last decade \cite{https://doi.org/10.48550/arxiv.2105.12110, zbMATH07483480, zbMATH07227743, zbMATH07226362, zbMATH06257634, MR3602841, zbMATH07114721, MR3650417, MR3568036, MR3390739}. 

As much as being strongly percolative is an insightful notion, absence of strong percolation yields very limited information. In the present work we address this imbalance by exhibiting an a-priori much weaker, but as will turn out equivalent, property involving only monotone information in the form of a suitable disconnection upper bound, uniform over scales. This is by no means obvious, one striking reason being that any reasonable notion of strong percolation, which comprises both  `existence-' and `uniqueness'-type characteristics (cf.~\eqref{e:strong-perco} below), is usually far from being a monotone property. 

The characterization of strong percolation we obtain is of independent interest. 
In a sense, it defines a `symmetric' analogue to the critical parameter $u_{**}$ introduced and extensively studied in \cite{MR2680403,MR2891880,MR2744881,PopTeix}, which exhibits a corresponding phase in the sub-critical regime, in which connectivity functions are well-behaved (i.e.~exhibit rapid decay) as soon as a suitable uniform connection upper bound holds.
The resulting more balanced view towards criticality is in line with the heuristic picture by which the system ought to be oblivious to the side from which the critical point is approached. As one important application, the conjectured sharpness of the phase transition for $\mathcal{V}^u$ follows by combining the characterization we obtain in the present work with the results of the companion article \cite{RI-I}. Our arguments imply that a regime of parameters in which  connection and disconnection  both occur with sizeable probability over all scales cannot be an extended interval.

\subsection{Main result}\label{subsec:MAIN} Let $\mathcal{V}^u $ denote the vacant set of random interlacements at level $u>0$ on $\Z^d$, $d \geq 3$, introduced in \cite{MR2680403}; see  Section~\ref{sec:prelims} for details. The random set $\mathcal{V}^u$ is decreasing in $u$. It undergoes a percolation phase transition across a threshold $u_*=u_*(d) \in (0,\infty)$, as follows: for all $u> u_*$, the connected components (clusters) of $\mathcal{V}^u$ are finite almost surely, whereas for $u<u_*$, there exists a unique infinite cluster with probability one; see \cite{MR2680403,MR2512613,10.1214/ECP.v20-3734,MR2498684}. With $B_r=([-r,r] \cap \Z )^d$ and for $u,v>0$, consider the events 
\begin{equation}
\label{e:strong-perco}
\begin{split}
\text{Exist}(r,u)&=\left\{\begin{array}{c}  \text{$\mathcal V^u \cap B_r$ contains a cluster with }\\ \text{($\ell^\infty$-)diameter at least  $\frac r5$}\end{array}\right\}, \\[0.3cm]
\text{Unique}(r,u,v)&= \left\{\begin{array}{c}\text{any two clusters in $ \mathcal V^u \cap B_{r}$ having diameter at}\\ \text{least $\frac r{10}$ are connected to each other in $ \mathcal V^v \cap B_{2r}$ } \end{array}\right\}.
\end{split}
\end{equation}
Events as in \eqref{e:strong-perco} have appeared in the percolation literature, see for instance \cite{AntPis96}, and also \cite{Tei11,MR3390739} in the context of $\mathcal{V}^u$. When present with high enough probability, these events lend themselves to powerful renormalization arguments, as witnessed in the above (long) list of references, which all crucially exploit this feature. 

We now introduce a (simpler) disconnection event. For $U,V \subset \Z^d$, we denote by $\{\displaystyle{\lr{}{ \mathcal{V}^u}{U}{V}}\}$  the connection event that a cluster of $\mathcal{V}^u$ intersects both $U$ and $V$ and replace $\longleftrightarrow$ by $\nlr{}{}{}{}$ to denote its complement, the corresponding disconnection event, by which we mean that no cluster of $\mathcal{V}^u$ intersects both $U$ and $V$ simultaneously. For a parameter $\gamma_M > 1$, we define the length scale
\begin{equation}
	\label{eq:def_M}
	M(r)=\exp\big\{(\log r)^{\gamma_M}\big\}, 
\end{equation}
which grows super-polynomially in $r$ and will play a central role in this article. In the sequel, $c,c',C$ etc.~refer to generic positive constants (i.e.~in $(0,\infty))$ that can change from place to place. Numbered constants are fixed upon first appearance within the text. All constants may implicitly depend on the dimension $d$. Their dependence on any other quantity will be made explicit. Following is our main result.

\begin{theorem} \label{T:MAIN}
For all $ d \geq 3$, there exists $\alpha=\alpha(d) \in (0,1)$ such that the following holds.\\[-0.5em] 

\noindent For all $u>0$ and $\gamma_M \geq \Cl{c:gamma_M}$, the following are equivalent:
\begin{itemize}
\item[i)]  for all $v \in (0,u)$, with $M=M(r)$ as in \eqref{eq:def_M}, 
\begin{equation}
\liminf_{r} \,  ( M / r)^{d} \,\P[\nlr{}{ \mathcal{V}^v}{B_{r}}{\partial B_{M}}] \leq \alpha; \label{eq:disc-cond}
\end{equation}
\item[ii)] for all $v,v'$ with $0 < v < v' < u$ and $C=C(v,v', \gamma_M) \in (0,\infty)$,
\begin{equation}
 \label{eq:barh1}
\P[\textnormal{Exist}(r,v')]\ge 1-C{ e}^{-r^{\Cl[c]{c:decay}}}\text{ and  }\P\left[\textnormal{Unique}(r,v',v)\right]\geq 1- C{ e}^{-r^{\Cr{c:decay}}}, \text{ for $r\ge1$}.
\end{equation}
\end{itemize}
\end{theorem}

The fact that $ii)$ implies $i)$ is a straightforward matter. The gist of Theorem~\ref{T:MAIN} is thus the implication $i) \Longrightarrow ii)$. Before discussing the difficulties with this in due detail (see \S\ref{subsec-pf-outline}) let us relate Theorem~\ref{T:MAIN} to existing results. Employing the language from the beginning of this introduction, we say that $\mathcal V =(\mathcal V^u)_{u>0}$ \textit{strongly percolates} at levels $v',v$ if \eqref{eq:barh1} holds for some constant $C=C(v,v') \in (0,\infty)$, and define 
\begin{equation}
\label{eq:baru3}
\bar{u}= \bar{u}(d) = \sup\big\{s >0:~ \mathcal V \textrm{ strongly percolates at level } v',v ~\textrm{for all } 0< v< v' < s \big\}.
\end{equation}
By \cite[Theorem 1.1]{MR3269990}, see also \cite{Tei11} for $d \ge 5$, one knows that $\bar{u} $ is non-trivial, i.e.~$\bar{u} >0$ for all $d \ge 3$. The (critical) value $\bar u (\leq u_*)$ pins down a subset of the percolative phase that is very robust, meaning that one has strong quantitative control on large local clusters (in the sense of \eqref{e:strong-perco} and \eqref{eq:barh1}). In analogy with \eqref{eq:baru3}, one naturally introduces, with $\alpha$ as supplied by Theorem~\ref{T:MAIN}, the parameter
\begin{equation}
	\label{eq:tildeu}
	\tilde{u}= \tilde{u}(d)= \sup \{ v > 0 : \eqref{eq:disc-cond} \text{ holds} \}.
\end{equation}
The threshold $\tilde{u}$ is of somewhat similar flavor as the definition of the critical parameter for Bernoulli percolation in \cite{DumTas15}, which can be viewed as refining the analogue of \eqref{eq:tildeu} (incorporating in particular a key `exploratory' feature).
With \eqref{eq:baru3} and \eqref{eq:tildeu}, the statement of Theorem~\ref{T:MAIN} now has the following immediate and succinct consequence.
\begin{corollary} For all $d \geq { 3}$ and all $\gamma_M \geq \Cr{c:gamma_M}$,  
\begin{equation}
\label{C:supercritical}
\text{$\bar{u}(d) =  \tilde{u}(d)$.}
\end{equation}
\end{corollary}
For completeness, let us mention that various reinforcements of being `strongly percolative' have been in circulation. The notion we deal with here has been fruitfully exploited to give strong answers to various problems relating to disconnection and the formation of droplet(s) in the supercritical regime \cite{https://doi.org/10.48550/arxiv.2105.12110, zbMATH07483480, zbMATH07227743, zbMATH07226362, zbMATH06257634, MR3602841, zbMATH07114721}. For other questions, see e.g.~\cite{MR3650417, MR3568036, MR3390739}, see also \cite{GriMar90, DCGR20, gosrodsev2021radius, https://doi.org/10.48550/arxiv.2107.06326} in various other contexts, it is of interest to remove the sprinkling, i.e.~to require $v'=v$ in \eqref{eq:baru3}, see e.g.~\cite[(1.3)]{MR3269990}. This is very close in spirit to the notion of ``well-behavedness'' of the supercritical phase which has appeared in the literature, see \cite{GriMar90,zbMATH02164755} and references therein. It is plausible, but presently open, that the sprinkling inherent to $\bar u$ in \eqref{eq:baru3} can be removed. We hope to return to this elsewhere~\cite{GRS23+}.

For certain applications, one may even wish to require a small `unfavorable' sprinkling, i.e.~to demand that \eqref{eq:baru3}  hold for all $u< s$ and some $v \in (u,s)$, see~\cite[(2.16)]{MR2838338}. Some regularity of the constant $C$ appearing in \eqref{eq:barh1} in its arguments $v,v'$ has also been propitiously used, see \cite[(2)-(3)]{zbMATH07395560}. In a related fashion, as follows upon inspection of our proof, one can in fact choose the constant $\Cr{c:decay}$ appearing in \eqref{eq:barh1} uniformly in $d$ (for instance $\Cr{c:decay}=\frac14$ works for the conclusions of Theorem~\ref{T:MAIN} to hold). One may then naturally wonder what the optimal decay for the events in \eqref{e:strong-perco} as well as whether $\bar u$ can be substituted for any of these stronger notions; cf.~\cite{GRS23+}.

\subsection{Proof outline}\label{subsec-pf-outline}  
We now discuss the proof of Theorem~\ref{T:MAIN}, and highlight some of the key issues in proving the implication $i) \Longrightarrow ii)$.
For $v$ such that \eqref{eq:disc-cond} holds, one has abundance of large clusters inside $B_{M}$, in that every translate of $B_r$ inside $B_M$ (with $M=M(r)$) is connected to $\partial B_M$ with high probability. The key is to argue that a certain gluing property holds, by which these large clusters all communicate after sprinkling, with probability tending to one as $r \to \infty$. From this, \eqref{eq:barh1} is then deduced via renormalisation.

For the purposes of this introduction, let us assume for simplicity that we only have two disjoint clusters in $\mathcal{V}^v$ crossing the annulus $A_M= B_M \setminus B_{M/2}$ whose $r$-neighborhoods cover all of $A_M$. This simplified setup is good enough to illustrate the main steps of the argument, as well as the difficulties encountered along the way. The reduction to this case from the general one, which includes many ambient (i.e.~large and $r$-dense) clusters, is inspired by an argument of Benjamini-Tassion \cite{MR3634283} in the context of Bernoulli percolation in a perturbative regime.
Within the simplified setup with two clusters only, one can exhibit (cf.~Lemma~\ref{L:goodpath_supercrit}) many disjoint contact zones between the two clusters, i.e.~boxes $\Lambda_k$ of side length $2r$ each intersecting a large chunk (macroscopic at scale $r$) of the two clusters. The key point is to exhibit a not too degenerate lower bound on the probability that the two clusters be connected locally inside $\Lambda_k$ after sprinkling to $\mathcal{V}^{u-\delta}$ for $\delta > 0$, and to exhibit this cost \textit{multiplicatively} in $k$, i.e.~generate some decoupling.
This will be achieved via a delicate bridging technique, which implements a surgery argument to construct a path at an affordable cost. 

To put things into perspective, let us start by recalling a line of argument from \cite{DCGRS20}, where a similar problem was faced in the context of the Gaussian free field. As it turns out, this is a much simpler problem, and the whole surgery argument developed in \cite{DCGRS20}, which was already intricate, fails to work here, but it helps to highlight the main issues. Roughly speaking, in \cite{DCGRS20}, one could afford to ask \textit{a-priori} for the boxes $\Lambda_k$ to have a certain `renormalized goodness' property, and to then condition on this goodness along with the two clusters before performing the surgery, which employed a technical device called bridging lemma; see \cite[Lemma 3.6]{DCGRS20}. In a nutshell, the goodness ensured the presence of a so-called good bridge, which facilitated the (quenched!) construction of a path in a cost-efficient way. Importantly, all this conditional information (i.e.~goodness+clusters) still \textit{left} randomness in spite of long-range correlations, a highly non-trivial feature, which is owed to a certain amount of `ellipticity' inherent to the free field. This left-over randomness was carefully exhibited using a decomposition of the field over scales, and thankfully `enough' randomness remained to perform the surgery. We will not further detail the specifics of this here; it will anyways be useless for us.

This approach is completely doomed for $\mathcal{V}^u$ because such heavy conditioning may in fact completely freeze the configuration, i.e.~remove all randomness. 
This is due to the fact that $\mathcal{V}^u$ exhibits strong degeneracies: for instance, for any finite set $K \ni 0$ (e.g.~$K= B_L$ for arbitrary large $L \geq 1$), one has that
\begin{equation}
\label{eq:no-FE}
\P[0 \notin \mathcal{V}^u \, | \, x \in \mathcal{V}^u \text{ for all } x \in \partial K]=0,
\end{equation}
which is an indication of `non-ellipticity'. In particular, \eqref{eq:no-FE} violates the commonly assumed finite-energy property, see e.g.~\cite[Def.~12.1]{zbMATH05728593} or \cite[Def.~3.2]{zbMATH01496108}, where this property is called insertion (and deletion) tolerance (which the free field satisfies). 
This feature poses severe restrictions on any attempt to condition on part of the configuration. Moreover, as explained in \cite{RI-I, RI-III}, the set $\mathcal{V}^u$ does not admit a natural intrinsic decomposition over scales. The best available conditional decoupling results \cite{PopTeix, CaioSerguei2018}, see also Proposition~\ref{prop:cond_decoup} below, essentially require a buffer zone of size at least $r^c$ when revealing the configuration in $B_r$, and even then leave very little control on the underlying conditional laws compared to the Gaussian free field for example, where one can exploit very explicit decomposition and monotonicity properties of conditional distributions. Still in the context of (smooth) Gaussian fields, we further refer to \cite{AHL_2022__5__987_0,severo2022uniqueness}, which implement a certain shifting technique to deal with degeneracy issues owing to analyticity, in the presence of short-range correlations.

\bigskip
\begin{figure}[h!]
  \centering 
  \includegraphics[scale=0.99]{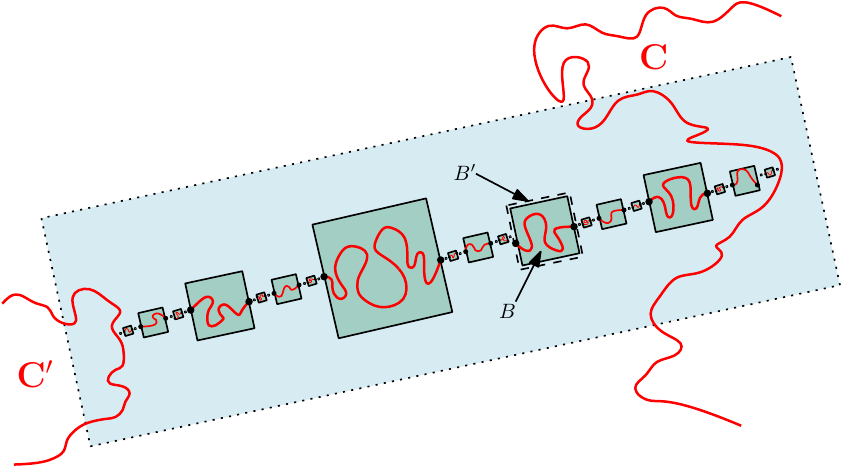}
  \caption{\textbf{Gluing} two vacant clusters $\mathbf{C}$ and $\mathbf{C}'$ using a fractal bridge (green). Connections in $\mathcal{V}^u$ (red) are re-constructed within each box $B$ of the bridge at preferential (polynomial) cost. The bottom scale is treated differently (not depicted). For each box $B=B(x,s)$ in the picture, the concentric larger box $B'=B(x,s+s^{\xi})$ for well-chosen (i.e.~not too small) $\xi \in (0,1)$ does not intersect the other green boxes (safety gaps). This feature is owed to very limited information on the conditional behavior of the occupation field of $\mathcal{V}^u$.}
  \label{F:bridge}
\end{figure}

The features outlined above warrant a completely novel approach to performing surgery, and the bridging technique we 
devise here, which is invented from scratch, needs to be mindful of strong effects such as \eqref{eq:no-FE}. The features outlined in the previous paragraph, especially the restrictive conditional decoupling property, preclude the possibility to ask for any kind of  `goodness property' and to condition on a good configuration a-priori. Rather, our approach is \textit{dynamic} in that we ask for good features (which may or may not occur) along the way, as we explore the region in which the connection is to be constructed. We explain this in more detail in the next paragraph. Interestingly, the overall failure rate of the procedure is ultimately measured in terms a large deviation event for a certain count of bad boxes; this is loosely reminiscent of an exploration argument by Aizenman-Kesten-Newman \cite{AizKesNew87}, see also \cite{GanGriRus88,cerf2015}, used in the course of proving uniqueness of the infinite cluster for Bernoulli percolation on $\Z^d$, which is key in addressing a related question: namely, that of bounding the probability of so-called `two-arm' events involving the presence of two disjoint large nearby clusters.

The `exploration' we perform starts with boxes at large scales, far away from the two clusters, and the `arrow of time' points towards smaller scales, i.e.~progressively refining the resolution. When a given box turns out to be `good,' we can reconstruct a piece of path at an affordable price. The surgery proceeds in this way and 
(re-)constructs `almost'-connections between the two clusters (whose geometry can be wild) in a hierarchical 
fashion, leaving polynomial gaps at each scale to generate decoupling; see Figure~\ref{F:bridge}. Each scale thereby contributes to the `finite-energy' cost of gluing, but a separate tool is needed when reaching the bottom scale. It involves a little device exhibiting a weak finite energy property, which essentially allows to vacate a box at a sprinkled level with \emph{not too low} conditional(!)~probability on a local event having high probability. When the entire `exploration' is successful, which happens overall with not too degenerate probability, the resulting path connecting the two clusters is a fractal curve involving all scales at once; see Figure~\ref{F:bridge}.

\bigskip
The methods we develop here are robust, and as such, provide a template that paves the way towards a better understanding of the super-critical regime in other dependent models of interest, including `non-elliptic' ones, e.g.~violating finite-energy property; matters relating to the `well-behavedness' of the supercritical phase have so far witnessed comparatively little progress, and results are restricted to specific models \cite{GriMar90,zbMATH02164755,DCGR20, DCGRS20, gosrodsev2021radius,  AHL_2022__5__987_0,  https://doi.org/10.48550/arxiv.2107.06326} and references~therein. To wit, the gluing technique we develop here yields a more robust proof of \cite[Proposition 4.1]{DCGRS20}, which avoids the use of a very specific decomposition of the field, and relies overall on a much less precise understanding of the conditional behavior of the occupation field.
Let us also mention that exactly the same gluing technique is employed in the companion article \cite{RI-I} in a more elaborate context, involving certain (inhomogenous, finite-range) models $\mathcal{V}_{k,L}^u$  approximating $\mathcal{V}^u$ (see \cite[Section 4]{RI-I}). Whereas ellipticity is not an issue for these models, conditional decoupling in a form as needed to perform the `exploration' does not come for free, and is achieved through additional coupling arguments, see \cite[Section 7]{RI-I} involving intermediate models in which `time runs for free' inside regions of interest, thus facilitating a comparison with $\mathcal{V}^u$.

\subsection{Organization} Section~\ref{sec:prelims} sets up the notation and gathers a few preliminary results, starting with a useful topological ingredient. It then collects two important inputs about random interlacements, a connectivity estimate and a conditional decoupling property. Sections~\ref{sec:sfe} and~\ref{sec:bridging} each contain a self-contained ingredient for the proof. Section~\ref{sec:sfe} exhibits a sprinkled finite energy property, which is interesting in its own right. Sections~\ref{sec:prelims} and \ref{sec:sfe} contain all model-specific inputs. Section~\ref{sec:bridging} comprises the deterministic bridge construction underlying our later surgery argument.

The proof of Theorem~\ref{T:MAIN} starts in Section~\ref{sec:supercritical}. This short section reduces the result to the key `gluing property' mentioned above, see Lemma~\ref{lem:uniqreduct} (cf.~also Proposition~\ref{P:glue-boost} for an enhancement). This reduction step follows closely the setup of \cite{DCGRS20} (itself adapted from \cite{MR3634283}), to which it frequently refers. All external inputs from \cite{DCGRS20} are isolated in that section.

Finally, Sections \ref{Sec:surgery} and \ref{Sec:2lemmas} are devoted to the proof of the (one-step) gluing lemma (Lemma~\ref{lem:uniqreduct}). They contain the delicate surgery argument delineated above, which brings into play the various ingredients from Sections~\ref{sec:prelims}-\ref{sec:bridging}, and represent the core of this article. The proof of
Lemma~\ref{lem:uniqreduct} is given in full in Section~\ref{Sec:surgery}, but the argument is discharged from two intermediate results, Lemmas~\ref{L:bound_bad} and~\ref{cor:ample_connection}, which control certain (key) counts of good and bad boxes used in the proof. These two lemmas are proved separately in Section~\ref{Sec:2lemmas} for the sake of readability.

\section{Preliminaries}

\label{sec:prelims}

In this section, we gather a few preliminary results. In \S\ref{subsec:crossing_interface}, we collect a topological result, see Lemma~\ref{L:goodpath_supercrit}, which will be used in Section~\ref{sec:supercritical} to exhibit the boxes $\Lambda_k$ in which connections are attempted. Its proof uses a result from Deuschel-Pisztora \cite{MR1384041}. In \S\ref{subsec:RI}, we first introduce a small amount of notation concerning random walks and random interlacements. We then gather two ingredients. The first is a useful connectivity estimate for $\mathcal{V}^u$, see Lemma~\ref{lem:twopointsbound}, which will be employed when reconstructing a path. The second, which is the content of Proposition~\ref{prop:cond_decoup}, is an adaptation of a conditional decoupling result of Alves-Popov \cite{CaioSerguei2018}, tailored to our needs.

\subsection{Connectivity of interfaces}\label{subsec:crossing_interface}
 We consider the lattice $\Z^d$, $d \geq3$, endowed with the usual nearest-neighbor graph structure and denote by $|\cdot|_2$ and $|\cdot|$ the  $\ell^{2}$- and $\ell^{\infty}$-norms on $\Z^d$. We write $x\sim y$ to denote neighbors, i.e.~when $|x-y|_2=1$ for $x,y \in \Z^d$. For $U \subset \Z^d$, the set $U^c =\Z^d 
\setminus U$ denotes its complement (in $\Z^d$), the set $\partial U = \{x \in U: \exists y\notin U \text{ s.t. } y \sim x\}$
is the interior vertex boundary of $U$ and $\partial_{\text{out}} U = \partial (U^c)$ its outer vertex boundary. We let $\overline{U}=U \cup \partial_{\text{out}} U$, and $U \subset \subset \Z^d$ means that $U\subset \Z^d$ has finite cardinality. We use the notations $B_r(x)=B(x,r)$ interchangeably to denote balls with radius $r> 0$ around $x \in \Z^d$ with respect to the $\ell^{\infty}$-norm and abbreviate $B_r=B_r(0)$. 
We write $d(\cdot, \cdot)$ to refer to the $\ell^{\infty}$-distance between subsets of $\Z^d$ and abbreviate $d(x,U)= d(\{x\},U)$ for $x \in \Z^d$ and $U \subset \Z^d$. 
A $*$-path is a finite or infinite sequence $(x_k)\subset \Z^d$ such that $|x_{k+1}-x_k| = 
1$. A path is defined similarly with $|\cdot|_2$ replacing $|\cdot|$. A crossing from $U$ to $V$ is a path whose range intersects both $U,V \subset \Z^d$.
 \begin{lem} \label{L:goodpath_supercrit} $(m > n \geq 1).$
	Let $U, V \subset \Z^d$ be such that $U \cup V = B_m \setminus B_{n-1}$ and both $U$ and $V$ contain a crossing from 
	$\partial B_n$ to $ \partial B_m$. Then there exists a $*$-path 
	$\pi \equiv (\pi(i))_{1 \le i \le |\pi|}$ intersecting both $\partial B_n$ and $ \partial B_m$ with
	\begin{equation}
		\label{eq:cond_path}
		d(\pi(i), U) \vee d(\pi(i), V)  \le 1 \text{\,\,for all $1 \le i \le |\pi|$.}
	\end{equation}
\end{lem}
\begin{proof}
Let $\mathcal C$ be any fixed connected component of $U$ intersecting both $\partial B_m$ and $\partial B_n$, which exists by hypothesis, and let $(\mathcal C')_1, \ldots, (\mathcal C')_k$ denote 
the connected components of $B_m \setminus \mathcal C$.

Now let $\pi_V \subset V$ be a (finite, nearest-neighbor) path connecting $\partial B_n$ and $\partial 
B_m$, which also exists by the hypothesis of the lemma. To fix ideas, we assume $\pi_V$ has a 
starting point in $\partial B_n$, terminates when first visiting $\partial B_m$ and does not intersect 
$\partial B_n \cup \partial B_m$ in between. By considering its successive entrance points in $\mathcal 
C$ and exit points from $\mathcal C$, the path $\pi_V$ is decomposed into disjoint, non-empty 
sub-paths $\pi_V^1, \ldots, \pi_V^\ell$ lying alternately in $\mathcal C$ and some $(\mathcal C')_j$ 
with $j \in \{1, \dots,k\}$. For concreteness, we assume henceforth that the starting point $\pi_V(0)$ of 
$\pi_V$ lies in $\mathcal{C}$; the other case is treated in a likewise manner. Hence, $\pi_{V}^{2i} 
\subset (\mathcal C')_{j_{2i}}$ for all $1 \leq i \leq \lfloor \ell/2 \rfloor$ and some $j_{2i} \in \{1,\dots,k\}$. 
	
	In the sequel we write $\partial_{B_m} S= \{ x \in B_m \setminus S : \exists y \in S : y \sim x \}$ to denote the relative boundary of a set $S \subset B_m$ in $B_m$. By \cite[Lemma~2.1 -- (ii)]{MR1384041}, one knows that the $\partial_{B_m} 
	(\mathcal C')_{j}$'s are all $*$-connected and hence the union $\pi'$ of the sets $\pi_V^1, \partial_{B_m} 
	(\mathcal C')_{j_2}, \pi_V^3, \partial_{B_m} (\mathcal C')_{j_4}, \ldots$ is $*$-connected as well. 
	Furthermore, as we now explain,
	\begin{equation}\label{eq:appC-1}\text{$\pi'$ intersects both $\partial B_n$ and $\partial B_m$.}
	\end{equation}
	The case of $\partial B_n$ is clear since $\pi'$ contains $\pi_V^1$ and hence $\pi_V(0)$. The `last' set entering the union forming $\pi'$ is either i) $\pi_V^{\ell}$ or ii) the relative boundary of a component $(\mathcal C')_j$ containing $\pi_V^{\ell}$. But since $\pi_V$ is a crossing, $\pi_V^{\ell}$ intersects $\partial B_m$. This immediately yields \eqref{eq:appC-1} in case i). In case ii) one has that $ (\mathcal C')_j$ intersects $\partial B_m$ (since it contains $\pi_V^{\ell}$, which does). We claim that this implies that $\partial_{B_m} (\mathcal C')_j$ then necessarily intersects $\partial B_m$. For, if $(\mathcal C')_j \cap \partial B_m \ne  \emptyset$ but 
	$\partial_{B_m}(\mathcal C')_j \cap \partial B_m =\emptyset$, then since $\partial B_m$ is a connected subset of $B_m$, we 
	get that $\partial B_m \subset (\mathcal C')_j$. However, this is not possible because $\mathcal 
	C \subset (B_m \setminus  (\mathcal C')_j)$ intersects $\partial B_m$. The same argument applies to $\partial B_n$ as well, which is relevant in case $\pi_V(0) \notin \mathcal{C}$. 
	
	Since $\pi'$ is $*$-connected and on account of \eqref{eq:appC-1}, we can therefore extract a $*$-connected crossing $\pi$ of $B_m \setminus B_{n-1}$ from $\pi'$. We claim that the crossing $\pi$ satisfies \eqref{eq:cond_path}. This is owed to the following two facts:
	\begin{itemize}
		\item[i)] $\pi_V^{2i-1} \subset 
		(\mathcal{C} \cap V) \subset U \cap V$, for all $i$,
		\item[ii)] $\partial_{B_m} (\mathcal C')_j \subset \partial_{{\rm out}} \mathcal C$, for all $j$ 
	\end{itemize} 
	(recall for item ii) that $(\mathcal C')_j $ are the connected components of $B_m \setminus \mathcal{C}$). Now ii) implies that $d(\partial_{B_m} (\mathcal C')_{j},\, U)  \leq d(\partial_{B_m} (\mathcal C')_{j},\, \mathcal C ) = 1$. Together with i), this immediately yields \eqref{eq:cond_path}.
\end{proof}

\subsection{Random walks and random interlacements}
\label{subsec:RI}

We write $P_x$ for the canonical law of the symmetric simple random walk on $\Z^d$ with starting point $x \in \Z^d$ and  $X = (X_n)_{n 
\geq 0}$ the corresponding discrete-time canonical process, whence $X_0=x$ under $P_x$. The measure $P_x$ is defined on the space $W^+$ endowed with its canonical $\sigma$-algebra $\mathcal{W}^+$ generated by the evaluation maps $X_n$, where $W^+$ refers the set of nearest-neighbor transient $\Z^d$-valued trajectories $w=(w_i)_{i\in \mathbb{N}}$ (transience means that $w^{-1}(\{x\})$, $x  \in \mathbb{Z}^d$, has finite cardinality). For $K \subset \Z^d$, we introduce the entrance time $H_K = \inf \{n \geq 0: X_n \in K \}$ in $K$, the exit time $T_K = H_{\Z^d \setminus K}$ from $K$ and the hitting time of $K$, defined as $\widetilde{H}_K = \inf \{ n \geq 1: X_n \in K\}$. 
 We further introduce
\begin{equation}\label{eq:equilib_K} 
e_K(x)=P_x[\widetilde{H}_K=\infty]1\{x \in K\},
\end{equation}
the equilibrium measure of $K \subset\subset \Z^d$, which is supported on $\partial K$. Its total mass
$\text{cap}(K) = \sum_x e_{K}(x)$ is the capacity of $K$.

The interlacement point process is defined on its canonical space $(\Omega, \mathcal{A}, \P)$, under which $\mathbb{P}$ is the probability measure governing a Poisson point process on the space $W^* \times \mathbb{R}_+$ with intensity measure $\nu(\mathrm{d}w^*) \mathrm{d}u$, where $\mathrm{d}u$ denotes the Lebesgue measure on $\R_+$ and $(W^*, \nu)$ is a 
 $\sigma$-finite measure space defined as follows. Let $W$ denote the set of doubly-infinite, nearest-neighbor transient trajectories in $\mathbb{Z}^d$, defined in a similar fashion as $W^+$, endowed with its canonical $\sigma$-algebra $\mathcal{W}$. The corresponding canonical shifts are denoted by $\theta_n: W\to W$, $n \in \Z$, with $\theta_n (w)(\cdot)= w(n+\cdot)$ and the canonical coordinates by $(X_n)_{n\in \Z}$. The shifts $\theta_n$, $n \geq 0$, also act on $W^+$. The space $W^*$ is the set of trajectories in $W$ modulo time shift, i.e.~$W^* = W/\sim$, where $w \sim w'$ if $w = \theta_n (w')$ for some $n\in \mathbb{Z}$. Let $\pi^*: W\to W^*$ denote the corresponding canonical projection. The $\sigma$-algebra $\mathcal{W}$ projects to $\mathcal{W}^*$, the canonical $\sigma$-algebra on $W^*$. We write $W^*_K \subset W^*$ for the trajectories visiting~$K \subset \mathbb{Z}^d$. The space $(W^*, \mathcal{W}^*)$ carries a natural measure $\nu(\mathrm{d}w^*)$, where \begin{equation}\label{e:RI-intensity}
\begin{split}
&1_{W^*_K} \nu_{\infty} = \pi^* \circ Q_K, \text{ for all $K \subset \subset \Z^d$,}
\end{split}
\end{equation}
and $Q_K$ refers to the finite measure on $W$ with
\begin{equation}\label{eq:Q_K}
Q_K[(X_{-n})_{n \geq 0} \in {A}, \, X_0 =x , \, (X_{n})_{n \geq 0} \in {A}' ]= P_x[ {A} \, | \, \widetilde{H}_K =\infty]e_K(x)P_x[ {A}'],
\end{equation}
 for all $x \in \Z^d$ and ${A}, {A}' \in \mathcal{W}^+$, with $e_K$ as in \eqref{eq:equilib_K}. The fact that $\nu$ given by \eqref{e:RI-intensity}-\eqref{eq:Q_K} gives rise to a (unique) well-defined measure follows from \cite[Theorem 1.1]{MR2680403}.
 
Given a sample $\omega  \in \Omega $ under $\P$, one defines the interlacement set
\begin{equation}
\label{e:def-I-u}
\mathcal{I}^u= \mathcal{I}^u(\omega)=\bigcup_{(w^*,v) \in \omega, \, v \leq u} \text{range}(w^*),
\end{equation}
where, with a slight abuse of notation, in writing $(w^*,v) \in \omega$ we tacitly identify the point measure $\omega$ with its support, a collection of points in $W^* \times \mathbb{R}_+$. The corresponding vacant set is given by $\mathcal{V}^u= \Z^d \setminus \mathcal{I}^u$. The set $\mathcal{V}^u$ is thus decreasing in $u$, and the parameter $u$ governs the number of trajectories entering the picture. We denote by $\ell^u= (\ell_x^u)_{x\in \Z^d}$ the field of (discrete) occupation times under $\P$, defined as
\begin{equation}
\label{e:ell-u}
\ell_x^u(\omega) = \sum_{(w,v) \in \omega} \sum_{n \in \Z} 1\{ w(n)=x, \, v \leq u\},
\end{equation}
for $x \in \Z^d$, so that, in view of \eqref{e:def-I-u} and \eqref{e:ell-u}, one has that $\mathcal{I}^u=\{ x\in \Z^d: \ell_x^u > 0\}$. Moreover, in view of \eqref{e:RI-intensity}, \eqref{eq:Q_K} and \eqref{e:def-I-u}, and recalling $\text{cap}(\cdot)$ from below \eqref{eq:equilib_K}, one has that
\begin{equation}
\label{eq:I_u}
\P[\mathcal{V}^u \supset K] = \exp\{ - u \text{cap}(K) \},
\end{equation}
which characterizes the law of $\mathcal{V}^u $.

Next, we derive some a-priori connectivity lower bound for the vacant set $\mathcal{V}^u$ under suitable assumptions on $u$.
\begin{lemma}\label{lem:twopointsbound}
	If $u>0$ and $\kappa>0$ are such that
\begin{equation}
\label{eq:cond-u-twopointsbound}
\inf_r  \P\big[\lr{}{\mathcal V^u}{B_r}{\partial B_{2r}}\big] \geq \kappa \, (\, > 0),
\end{equation}
then for every $r\geq 1$ and $x, y\in B_r$,
	\begin{align}
	&\label{eq:twopointsboundweaker}
	\P\big[\lr{}{\mathcal V^u \cap B_{2r}}{x}{y}\big]\geq c(\kappa)r^{-C},\\
	&\label{eq:twopointsbound}
	\P\big[\lr{}{\mathcal V^u \cap B_{r}}{x}{y}\big]\geq e^{-\Cl{ctwopoints}(\kappa)(\log r)^2}.
\end{align}
\end{lemma}

\begin{proof} The proof of \eqref{eq:twopointsboundweaker} proceeds exactly as that of \cite[Lemma 3.4]{DCGRS20}: the argument only relies on the FKG-inequality, which holds in the present context, see \cite[Theorem 3.1]{MR2525105}, and the invariance of the law of $\mathcal{V}^u$ under lattice symmetries. The assumption \eqref{eq:cond-u-twopointsbound} replaces the condition on $h$ appearing in \cite{DCGRS20}.

 We now show \eqref{eq:twopointsbound}. By the FKG-inequality, we may assume that $y=0$ and that $r\geq 10$. Let $L_k \stackrel{{\rm def}.}{=} 2^{k}$, for $k \geq 0$. Still by the FKG-inequality, we may suppose that 
	$B(x,4) = B(x, 4L_0) \subset B_r$. Let $k_0 = \max\{ k: B(x, 4L_k) \subset B_r\}$ and $k_1 \geq k_0$ be such that 
$L_{k_1+2}\leq r < L_{k_1+3}$. 

We define a sequence of vertices $y_k$ for $ k \in [[k_0, k_1]] \stackrel{{\rm def}.}{=} [k_0 , k_1] \cap \Z$ inductively as follows. 
Let $y_{k_0}= x$ and note that $B(y_{k_0}, 4L_{k_0}) \subset B_r$ by choice of $k_0$. 
Assuming $y_{k_0},\dots, y_{\ell}$ have been defined for some $ k_0 \leq \ell < k_1$ and that 
$B(y_{\ell},4L_{\ell}) \subset B_r$, we define $y_{\ell+1}$ in the following way: first we 
choose an intermediate point $y_{\ell + 1/2}$ deterministically on $\partial B(y_{\ell}, 
L_{\ell +1})$ such that $B(y_{\ell+1/2}, 4L_{\ell}+ L_{\ell +1}) \subset B_r$ (for instance, the point on $\partial B(y_{\ell}, 
L_{\ell +1})$ minimizing the $\ell^1$-distance to the origin will work). 
Then we repeat this and pick $y_{\ell+1}$ on $\partial B(y_{\ell+1/2}, L_{\ell +1})$ such that 
$B(y_{\ell+1}, 4L_{\ell}+ 2L_{\ell +1})= B(y_{\ell+1}, 4L_{\ell+1}) \subset B_r$. 

The sequence of points thereby constructed has the following property: for all $\ell \in [[k_0 , k_1]]$ and $k \in \{ \ell,\ell+1/2 \}$,
\begin{equation}
\label{eq:connectLB1}
y_{k+\frac12} \in B(y_k, L_{\lfloor k \rfloor +1} )  \text{ and } B_k \stackrel{\text{def.}}{=} B(y_k, 2L_{\lfloor k \rfloor +1} ) \subset B_r.
\end{equation}
Thus, we obtain that 
\begin{multline*}
	\P\big[\lr{}{\mathcal V^u \cap B_{r}}{x}{y_{k_1}}\big] \\
	\stackrel{\eqref{eq:connectLB1}}{\geq} \prod_{k_0\leq \ell < k_1} \prod_{k \in \{ 
		\ell,\ell+\frac12 \}}\P\big[\lr{}{\mathcal V^u \cap B_k}{y_k}{y_{k+\frac12}}\big] \stackrel{\eqref{eq:twopointsboundweaker}}{\geq}c(\kappa )\prod_{0 \leq k \leq k_1} L_k^{-C} 
	\geq c(\kappa )e^{-c'(\kappa)(\log r)^2}, 
	\end{multline*}noting in the last step that $k_1 \leq C \log r$ by definition. Since $B(y_{k_1}, cr) \subset B_r$, it easily follows by suitably covering $B_r$ with a constant number of boxes of radius, say, $cr/10$, using \eqref{eq:twopointsboundweaker} and the FKG-inequality that $\P\big[\lr{}{\mathcal V^u \cap B_{r}}{0}{y_{k_1}}\big] \geq c(\kappa)r^{-c}$, and \eqref{eq:twopointsbound} follows.
\end{proof}

We conclude this section with a certain decoupling 
estimate, see Proposition~\ref{prop:cond_decoup} below. To this effect, we start by setting up a decomposition of trajectories into excursions. This framework will also be useful in the next section.
 We assume henceforth that for \textit{any} realization $\omega =\sum_{i \geq 0} \delta_{(w_i^*, u_i)} \in \Omega $, 
the labels $u_i$, $i \geq 0$, are pairwise distinct, that $\omega(W_A^*\times[0,u])<\infty$ for all $u 
\geq 0$ and that $\omega(W_A^*\times \R_+)=\infty$, which is no loss of generality since these sets 
have full $\P$-measure.

We will use the following excursion decomposition. Let $A ,U$ be finite subsets of $\mathbb{Z}^d$ 
with $\emptyset \neq A \subset U$. The (doubly) infinite 
transient trajectories, i.e.~elements of $W$ or $W^+$, see around \eqref{eq:equilib_K} for notation, are split into excursions between $A$ and $\partial_{{\rm out}} U$ by introducing the successive return and departure times 
between these sets: $D_0 =0$ and 
\begin{equation*} 
  \begin{aligned}
    & R_k = D_{k-1} + H_A \circ \theta_{D_{k-1}} , \qquad &  D_k = R_k + T_U \circ \theta_{R_k},
  \end{aligned}
\end{equation*}
 for $k \geq 1$, where all of $D_k, R_j, D_j$, $j > k$ are understood to be $=\infty$ whenever $R_k=\infty$ for some $k \geq 0$.
We denote by $W_{A, \partial_{{\rm out}} U}^+$ the set of all excursions between $A$ and $\partial_{{\rm out}} U$, i.e.~all finite nearest neighbor trajectories starting in $A$, ending in $\partial_{{\rm out}}U$ and not exiting $U$ in between. 
Given $\omega=\sum_{i \geq 0} \delta_{(w_i^*, u_i)}$, we order all the excursions from $A$ to $\partial_{{\rm out}} U$, first by increasing value of $\{u_i: w_i^* \in W_A^*\}$, 
then by order of appearance within a given trajectory $w_i^* \in W_A^*$. This yields a sequence of 
$W^+_{A, \partial_{{\rm out}} U}$-valued random variables under $\P$, encoding the successive 
excursions:
\begin{equation}
\label{eq:RI_Z}
\big( Z^{A, U}_n(\omega) \big)_{n \geq 1} \stackrel{\text{def.}}{=} \big( w_0[R_1, D_1], \dots, 
w_0[R_{N^{A, U}}, D_{N^{A, U}}],\, w_1[R_1, D_1], \dots \big),
\end{equation}
where $N^{A,U}= N^{A,U}(w_0^*)$ is the total number of excursions from $A$ to $\partial_{{\rm out}} 
U$ in $w_0^*$, i.e. $N^{A, U}(w_0^*) = \sup\{j : D_j(w_0) < \infty\}$ and $w_0$ is any point in the 
equivalence class $w_0^*$. We will omit the superscripts $A,U$ whenever no risk of confusion arises.

We now associate to $\omega$ a finite 
multiset $\mathcal{C}_u= \mathcal{C}_u^{A, U}(\omega)$, for $u \geq 0$, obtained by collecting the pairs 
of start- and endpoints of any excursion between $A$ and $\partial_{{\rm out}} U$ in the support of $\omega$ with 
label at most $u$, and forgetting their labels. That is, $\mathcal{C}_u$ comprises all pairs $(x,y) \in \partial A \times \partial_{{\rm out}} U$  such that $x=w^*(R_k)$, $y= w^*(D_k)$ for some $1\leq k \leq N(w^*)$ and  $w^* \in 
W^*$ such that $(w^\ast, v) \in \text{supp}(\omega)$ for some $v \leq u$ (note that this gives rise to a multiset 
since pairs can appear repeatedly). The random variable $\mathcal{C}_u$ takes values in the measure space $(\Omega_{\mathcal{C}}, \mathcal{A}_{\mathcal{C}})$ (depending implicitly on $A$ and $U$). One can simply take $\mathcal{A}_{\mathcal{C}}=2^{\Omega_{\mathcal{C}}}$ since 
$\Omega_{\mathcal{C}}$ is countable.

Let $B \subset A$. The following decoupling result will apply conditionally on the 
endpoints of the successive excursions appearing in \eqref{eq:RI_Z} (thereby typically 
decoupling a set $B$ well inside $A$, for $A \subset U \subset \subset \Z^d$). For any pair of points $(x, y) \in \mathcal C_{u}$ associated to a labeled trajectory in the support of $\omega$ \textit{visiting $B$}, one considers the induced sub-trajectory, starting from the time it first visits $B$, until the last time the trajectory is in $A$ prior to $T_U$. Let $\mathcal 
D_u$ denote the collection of all such sub-trajectories and $(\Omega_{\mathcal D}, 2^{\Omega_{\mathcal D}})$ denote the measure space underlying it ($\Omega_{\mathcal D}$ is 
countable just like $\Omega_{\mathcal C}$). We assume that $\Omega_{\mathcal{D}}$ carries a cemetery state $\Delta \in \Omega_{\mathcal{D}}$ corresponding to pairs $(x,y) \in \mathcal{C}_u$ whose associated trajectory doesn't visit $B$. Hence, one has that\begin{equation}\label{eq:IuDu}
\mathcal I^u \cap B = \mathcal I(\mathcal D_u) \stackrel{{\rm def}.}{=} \bigcup_{w \in \mathcal D_u} {\rm 
range}(w) \cap B,
\end{equation}
(with the convention $\text{range}(\Delta)= \emptyset$), i.e.~the interlacement set inside $B$ is a function of $\mathcal D_u$.

\begin{prop}\label{prop:cond_decoup}
There exist $\Cl[c]{c:box_gap} \in (0, \frac 12)$ and $\Cl{C:box_gap} > 0$ with the following properties:

\smallskip
\noindent \begin{itemize}
\item[i)]	 
For all $r \geq C$ there exist sets $A, U$ with $ B=B_{r} \subset A \subset U \subset B_{r + r^{1 - \Cr{c:box_gap}}}$ 
such that, for every $0 < u$ 
and 
$\delta \in (0, 1)$, one can find 
$\mathcal{A} = \mathcal{A}(B) \in \mathcal{A}_{\mathcal{C}}$ with
\begin{equation}\label{eq:good_event_excursion_bnd}
\P\big[ \mathcal C_{u}  \in \mathcal{\mathcal A} \big]\geq 1 - e^{-\Cr{c:box_gap}  u \delta^2 
	r^{\Cr{c:box_gap}}}		
\end{equation}
and for all fixed 
$\zeta \in \mathcal A$, there is a coupling $\hat{\mathbb Q}=\hat{\mathbb Q}_{\zeta}$ 
of 
three 
$\Omega_{\mathcal D}$-valued random variables $(\hat{\mathcal D}_{u(1-\delta)}, \hat{\mathcal D}, \hat{\mathcal D}_{u(1+\delta)})$ such that $(\hat{\mathcal D}_v : v \in \{u(1 \pm \delta)\}) \stackrel{{\rm law}}{=} (\mathcal D_{v}: v \in \{u(1 \pm \delta)\})$ and 
$\hat{\mathcal D}$ having the law of 
$\mathcal D_{u}$ under 
$\P[\, \cdot \, |  \mathcal C_{u} = \zeta]$, and
\begin{equation}\label{eq:excursion_sandwich}
\hat{\mathbb Q}\big[ \hat{\mathcal D}_{u(1 - \delta)}  \subset \hat{\mathcal D} \subset \hat{\mathcal 
D}_{u(1+ \delta)} \big] \geq 1 - \Cr{C:box_gap} e^{-\Cr{c:box_gap}  u\delta^2r^{\Cr{c:box_gap}}}.
\end{equation}
\item[ii)]  With $B(=B_r)$, $A$ and $U$ as in item i), letting $\mathcal V^u_B= \mathcal V^u \cap B$ and defining 
\begin{equation}\label{def:GBudelta}
\Xi_{B}^{u, \delta} = \bigg\{ \zeta \in \Omega_{\mathcal{C}} : \begin{array}{l} \E\big[ f(\mathcal V_B^u ) \, \big |  \, \mathcal C_u = \zeta \big] \geq \E\big[ f(\mathcal V_B^{u(1+\delta)}) \big] - 
\Cr{C:box_gap}e^{-\Cr{c:box_gap} u {\delta}^2 r^{\Cr{c:box_gap}}} \\[0.3em] \text{for all increasing functions $f : \{0, 1\}^{B} 
		\mapsto [0, 1]$}\end{array}\bigg\},	
\end{equation}
one has for every $r \ge C$, $u , \delta > 0$, letting $\mathcal G_B^{u, \delta} = \{\mathcal C_{u} \in \Xi_{B}^{u, \delta}\}$, that
\begin{equation}\label{eq:corcond_decouple}
	\P[\mathcal G_{B}^{u, \delta}] \ge 1 - e^{-\Cr{c:box_gap} u {\delta}^2r^{\Cr{c:box_gap}}}.
\end{equation}
\end{itemize}
\end{prop}

\begin{proof}
As we now explain, $i)$ essentially follows from \cite[Propositions~4.2]{CaioSerguei2018} with some modifications.
We first define the relevant sets $A$ and $U$. For $r \geq 1$, we let $s=r^{\frac1{b}}$ with (say) $b= \frac12(1 + \frac{4d-4}{3d-2})$. The choice of $b$ corresponds to a valid choice of the quantity $b_{\square}$ in \cite[(1.7)]{CaioSerguei2018}. Note that $b>1$ for all $d \geq 3$. 
Now, for a given set $B=B_r$, introduce the rounded boxes $A= \bigcup_{x \in B_{r+s}} B^2(x,s)$, where $B^2(x,s)$ refers to the $\ell^2$-ball in $\Z^d$ of radius $s$ around $x$. The set $U$ is defined similarly as $A$, with the union ranging over all $x \in B_{r+2s}$ instead. This gives $A\subset U \subset B_{r + r^{1 - \Cr{c:box_gap}}}$ upon choosing $ \Cr{c:box_gap} >0$ small enough.
The ref.~\cite{CaioSerguei2018} involves sets $A_1$-$A_3$ and $V$, and one sets $A_2=A_3 = U^c$ and $V= \partial A$. The set $A_1$ corresponds to the region in which the coupling operates (one considers excursions upon hitting $A_1$ until their last visit to $V$ prior to hitting $\partial A_2$, see \cite[(3.7-8)]{CaioSerguei2018}) and one sets  $A_1= \bigcup_{x \in B_{r}} B^2(x,s)$, so that $A_1 \supset B$. 

With these choices one applies 
\cite[Propositions~4.2]{CaioSerguei2018}, which yields a coupling $\mathcal{P}$ such that (in the notation of \cite{CaioSerguei2018}, see in particular (4.2) therein), %
\begin{equation}\label{e:caio-result}
\E\Big[ \mathcal{P}\big[ G^{\Sigma}_{u(1-\delta)}\leq G^{\zeta} \leq G^{\Sigma}_{u(1+\delta)} \big] \big\vert_{\bar\zeta= \bar{\mathcal{C}}_u} \Big] \geq 1 - \Cr{C:box_gap}e^{-\Cl[c]{c:caio-annealed}  u\delta^2r^{\Cr{c:box_gap}}};
\end{equation}
here, the outer expectation is with respect to $\P$ and acts on 
$\bar{\mathcal{C}}_u$ alone, $(G^{\Sigma}_u)_{u \geq 0}$ refers to the soft local time of 
the process $(\mathcal D_u)_{u \ge 0}$ and $G^{\bar\zeta}$ 
to that of 
the process $\mathcal D_u$ at level $u$ 
whose {\em clothesline process} 
$\bar{\mathcal{C}}_u$ is conditioned to equal $\bar{\zeta}$. 
As opposed to $\mathcal{C}_u$, the process $\bar{\mathcal{C}}_u$ keeps track of the order of occurrence 
of points similarly as in \eqref{eq:RI_Z}. The event in \eqref{e:caio-result} in turn readily implies the chain of inclusions
appearing in \eqref{eq:excursion_sandwich}, with the correct marginal laws for the 
$i)$, one defines $\hat{\mathbb{Q}}$ as the induced joint law of the 
three sets in question under $\mathcal{P}$. Observe that the law 
$\hat{\mathbb{Q}} =  \hat{\mathbb{Q}}_{\zeta}$ is indeed a function of 
$\zeta$ alone (with hopefully obvious notation, $\zeta$ denotes the multi-set associated to $\bar{\zeta}$): 
for, reconstructing $\mathcal{V}_B^u$ under $\P[\cdot | \bar{\mathcal{C}}_u=\bar\zeta]$ does not require 
knowing the order of appearance of elements in $\bar\zeta$. One then sets
\begin{equation}
\label{e:A_explicit}
\mathcal{A} =  \big\{ \zeta : \, \hat{\mathbb{Q}}_{\zeta}\big[ \hat{\mathcal D}_{u(1 - \delta)}  \subset 
\hat{\mathcal D} \subset \hat{\mathcal D}_{u(1+ \delta)} \big] \geq 1 - 
\Cr{C:box_gap}e^{-\frac{\Cr{c:caio-annealed} }{2} u{\delta}^2r^{\Cr{c:box_gap}}}\big\}.
\end{equation}
With \eqref{e:A_explicit}, \eqref{eq:excursion_sandwich} is immediate, and \eqref{eq:good_event_excursion_bnd} follows by \eqref{e:caio-result}, upon noticing that the left-hand side of \eqref{e:caio-result} is bounded from above by 
$$\E\Big[ \hat{\mathbb{Q}}_{\zeta}\big[ \hat{\mathcal D}_{u(1 - \delta)}  \subset \hat{\mathcal D} \subset 
\hat{\mathcal D}_{u(1+ \delta)} \big] \big\vert_{\zeta= {\mathcal{C}}_u} \Big]$$
by monotonicity. Distinguishing in the previous display whether $\mathcal{A}$ or $\mathcal{A}^c$ occur, 
and using the upper bound implied by \eqref{e:A_explicit} in the latter case then readily gives the inequality 
$x+(1 - \Cr{C:box_gap}e^{-\frac{\alpha}{2}})(1-x) \geq 1 - \Cr{C:box_gap}e^{-\alpha}$, with $x = \P[ (\mathcal C_{u}, \mathcal C_{u'})  \in 
\mathcal{\mathcal A} ]$ and $\alpha = \Cr{c:caio-annealed}  u{\delta}^2r^{\Cr{c:box_gap}}$, from which \eqref{eq:good_event_excursion_bnd} follows. Item $ii)$ is a straightforward consequence of $i)$ and \eqref{eq:IuDu}. 
\end{proof}

For later reference, we conclude with the following observations. 
\begin{remark}\label{R:Caio}
\begin{enumerate}[label*=\arabic*)]
 \item \label{rem:caio_monot} (Monotonicity in \eqref{def:GBudelta}).
By inclusion, the multisets $\zeta$ carry a natural partial order and $\Xi_{B}^{u, \delta} $ is decreasing with respect to this partial order. Indeed if $\zeta' \supset \zeta$ contains more (pairs of) points, then $\E[ f(\mathcal V_B^u ) \,  |  \, \mathcal C_u = \zeta' ] \leq \E[ f(\mathcal V_B^u ) \,  |  \, \mathcal C_u = \zeta ]$, for $\zeta'$ requires constructing additional (independent) random walk bridges (having the correct marginal law) to connect the additional points present in $\zeta'$, which decreases $\mathcal V_B^u$.

\item \label{eq:monotonicity_Du} (Monotonicity with respect to $\mathcal D_u$). For any $B' = B(x', r')$ such that $B(x', r' + {r'}^{1 - \Cr{c:box_gap}}) \subset B$, with $A'$ and $U'$ referring to the sets $A$ and $U$ for the box $B'$, cf.~Proposition~\ref{prop:cond_decoup},$i)$,
\begin{equation}\label{eq:monotone_Cu}
	\text{$\mathcal C_u^{A', U'}$ is a monotonically increasing function of $\mathcal D_u$ with respect to inclusion.}
\end{equation}

\item \label{rem:caio_multiple}
(Multiple $u$'s in \eqref{eq:excursion_sandwich}). Let $u_1, \ldots, u_k > 0; \, \delta_1, \ldots, \delta_k \in (0, 1)$, and $A, U$ and $r$ be as in item~i) of 
Proposition~\ref{prop:cond_decoup}. Then there exists an event $\mathcal A \in \mathcal A_{\mathcal{C}}^{\otimes k}$ with 
\begin{equation}\label{eq:good_event_mul_excursion_bnd}
\P\big[ (\mathcal C_{u_1},\ldots, \mathcal C_{u_k})   \in \mathcal{\mathcal A} \big]\geq 1 - e^{-\Cr{c:box_gap} u {\delta}^2 r^{\Cr{c:box_gap}}},
\end{equation}
where $u \stackrel{{\rm def}.}{=} \min\{ u_1, \ldots, u_k\}$ and $\delta \stackrel{{\rm def}.}{=} \min\{ \delta_1, 
\ldots, \delta_k\}$, and for all fixed $(\zeta_1, \ldots, \zeta_k) \in \mathcal A$, there is a coupling $\hat{\mathbb Q} = \hat{\mathbb Q}_{\zeta_1, \ldots, \zeta_k}$ between the family of $\{0, 1\}^{\Z^d}$-valued random 
variables $(\hat{\mathcal D}_{v} : v \in \{u_i(1 \pm \delta_i): 1 \le i \le k\}) \stackrel{{\rm law}}{=} (\mathcal D_{v}  : v \in \{u_i(1 \pm \delta_i) : 1 \le i \le k\})$ and $(\hat{\mathcal D}_1, \ldots, \hat{\mathcal D}_k)$ having 
the law of $(\mathcal D_{u_1}, \ldots, \mathcal D_{u_k})$ under $\P[\, \cdot \, |  (\mathcal C_{u_1}, \ldots, \mathcal C_{u_k}) = (\zeta_1, \ldots, \zeta_k)]$, and
\begin{equation}\label{eq:mul_excursion_sandwich}
\hat{\mathbb Q}\big[ \hat{\mathcal D}_{u_i(1 - \delta_i)}  \subset \hat{\mathcal D}_i  
\subset \hat{\mathcal D}_{u_i(1+ \delta_i)}\ \text{ for all } 1\leq i \leq k \big] \geq 1 - k\,\Cr{C:box_gap} e^{-\Cr{c:box_gap}  
u{\delta}^2r^{\Cr{c:box_gap}}}.
\end{equation}
This follows from a minor modification to the argument used in the proof of 
Proposition~\ref{prop:cond_decoup}. Indeed, (a slight extension of) \cite[Propositions~4.2]{CaioSerguei2018} also gives, for all $1 \le i \le k$ (cf.~\eqref{e:caio-result}),
\begin{equation}\label{eq:mul_caio-result}
\E\Big[ \mathcal{P}\big[ G^{\Sigma}_{u_i(1-\delta_i)}\leq G^{\bar\zeta_i} \leq G^{\Sigma}_{u_i(1+\delta_i)} \big] \big\vert_{\bar\zeta_1 = \bar{\mathcal{C}}_{u_1}, \ldots, \bar\zeta_k = \bar{\mathcal{C}}_{u_k}} \Big] \geq 1 - \Cr{C:box_gap}e^{-\Cl[c]{c:caio-annealed}  u\delta^2r^{\Cr{c:box_gap}}},
\end{equation}
where $\E$ now acts on $( \bar{\mathcal{C}}_{u_1}, \ldots, \bar{\mathcal{C}}_{u_k})$. The remainder then 
follows in the same manner as before, applying a union bound over $k$ to deduce 
\eqref{eq:mul_excursion_sandwich} from \eqref{eq:mul_caio-result}. 
\end{enumerate}
\end{remark}

\section{Sprinkled finite energy property}

\label{sec:sfe}

We now derive a separate ingredient for our proof of Theorem~\ref{T:MAIN}, which we call sprinkled finite energy. Roughly speaking, the event $\widetilde{F}_B$ introduced in Proposition~\ref{lem:finite_energy} below is designed with the following property in mind: $\widetilde{F}_B$ enables us to open up the box $B$ in $\mathcal V^{u - \delta}$ with not too degenerate probability conditionally on carefully chosen information, including $\mathcal I^u$ as well as all starting and endpoints of excursions within a larger box $\widehat B$, starting from its boundary. Note in particular that this entails a `buffer' zone $\widehat{B} \setminus B$, which is non-negotiable. We will eventually use this tool at the bottom scale in the upcoming bridge construction, in order to `plug' its remaining holes (see the beginning of Section~\ref{sec:bridging}, where holes will be precisely defined).

Stating the sprinkled finite energy property precisely requires a minimal amount of preparation. 
For $B$ a box and $u \ge 0$, let $\widetilde{\mathcal 
C}_u^{B} = \widetilde{\mathcal C}_u^{B}(\omega)$ denote the (finite) sequence containing the pairs of start- and endpoints of the 
successive excursions between $B$ and $\partial_{{\rm out}} B$ in the support of $\omega$ with label at 
most $u$, in order of appearance and with their associated labels; recall the allied notion $\mathcal C_{u}^{B, B}$ 
introduced below \eqref{eq:RI_Z}, where in contrast both the labels and the order of appearance were forgotten.
For any $w^{*} \in W^{*}$, we denote by $\phi_{B}^{-}(w^*)$ the 
sequence of 
segments (sub-paths) of $w^\ast$ obtained when removing the \textit{interiors} of all 
excursions in $w^\ast$ between $B$ and $\partial_{{\rm out}} B$, i.e.~with the exception of their start- and endpoints. The segments are arranged according to order of appearance within $w^*$. Now let
\begin{equation}\label{def:omega-BPhiBu}
	\omega_{B}^{-}= \sum_{(w^\ast, v) \in \omega} \delta_{(\phi_{B}^{-}(w^\ast), v)}.
\end{equation}
In particular, $\widetilde{\mathcal C}^B_u$ is measurable relative to $\omega_{B}^{-}$.

\begin{prop}[Sprinkled finite energy]\label{lem:finite_energy}
For any $u \ge \delta > 0$, $x \in \Z^d$, $r_0,r \geq 1$, $B=B(x,r)$ and $\widehat B=B(x, r + 7{r_0})$, there exists an event $\widetilde{F}_B=\widetilde{F}_B^{u,\delta, r_0}$ having the following properties:
\begin{align}
&\label{eq:mathscrEBmeasurable}
\text{$\widetilde{F}_B$ is measurable relative to $(\widetilde{\mathcal C}^{\widehat B}_u,  \mathcal I^u \cap \widehat B)$},\\
&\label{eq:finite_energy}
\P\big[B\subset \mathcal V^{u - \delta} \, \big| \, \sigma ( \omega_{\widehat{B}}^{-},\, \mathcal 
I^u \cap \widehat B )\,\big]1_{\widetilde{F}_B}\ge e^{-C(r \vee {r_0})^{2d}},\\
&\label{e:F_B-proba} \P[\widetilde F_B^c] \leq C(r \vee r_0)^d e^{-c (u \wedge \delta) r_0^c}.
\end{align}
\end{prop}

The proof of Proposition~\ref{lem:finite_energy} is given below. 
A box $B=B(x,r)$ will later be called \textit{finite-energy good (with parameters $(u,\delta, r_0)$)} if an event $\widetilde{F}_B$ with the properties postulated by Proposition~\ref{lem:finite_energy} occurs; for concreteness, one can take the explicit event \eqref{def:mathscrE-concrete} constructed in the proof. In practice (see for instance Section~\ref{subsec:proof_bound_bad}), it can at times be useful to know that $\widetilde{F}_B$ is implied by another event, still satisfying \eqref{e:F_B-proba} but with `worse' measurability properties than \eqref{eq:mathscrEBmeasurable}, however with advantageous monotonicity features in terms of $u$, lending themselves to arguments involving sprinkling; see Remark~\ref{R:FE-implication-event} for more on this.

\begin{proof}[Proof of Proposition~\ref{lem:finite_energy}]
We start by defining the event $\widetilde{F}_B$. Recall the sequence $\widetilde{\mathcal 
C}_u^{B}$ of labeled start- and endpoints of the 
successive excursions between $B$ and $\partial_{{\rm out}} B$ by trajectories in the support of the interlacement process $\omega$ with label at most $u$. For $r, r_0 \geq 1$ and $x \in \Z^d$, let $B=B(x,r)$ and $\widehat{B}=B(x, r + 7{r_0})$ as in the statement of Proposition~\ref{lem:finite_energy}. Now consider the event
\begin{equation}\label{def:mathscrE-concrete}
\begin{split}
\widetilde{F}_B = \widetilde{F}_B^{u, \delta , r_0} \stackrel{{\rm def}.}{=}  \widetilde{F}_{B}'  \cap \widetilde{F}_B'' (\mathcal{I}^u \cap \widehat{B})
\end{split}
\end{equation}
where $u \ge \delta>0$ ,
$$
\widetilde{F}_B'= \big\{|\widetilde{\mathcal C}^{\widehat B}_u|  \le {r_0}\,|\partial_{{\rm out}} \widehat B|\big\}
$$
and for $I \subset \widehat{B}$,
\begin{equation*}
\begin{split}
 \widetilde{F}_{B}''(I) \stackrel{{\rm def}.}{=} \left\{\begin{array}{c}\text{$\widetilde{\mathcal C}^{\widehat B}_u= \widetilde{\mathcal{C}}$ for some $\widetilde{\mathcal{C}}$ such that, for any $\xi=(x,y,v) \in \widetilde{\mathcal{C}}$, there is an excursion}\\ \text{$w=w(\xi)$ starting in $x$ and ending in $y$, with $|w| \le (20(r \vee {r_0}))^d$, such that}\\ 
\text{$\bigcup_{\xi \in  \widetilde{\mathcal{C}}}\, {\rm range}(w(\xi)) =I$ and $\bigcup_{\xi =(x,y,v) \in  \widetilde{\mathcal{C}},\, v \le u - \delta}\, {\rm range}(w (\xi)) \cap B = \emptyset$}\end{array}\right\};
\end{split}
\end{equation*}
here with hopefully obvious notation, $w$ refers to an excursion between $\widehat{B}$ and $\partial_{\text{out}} \widehat{B}$ and $|w|$ denotes its (time-)length. We note that $(w(\xi))_{\xi \in \widetilde{\mathcal C}^{\widehat B}_u}$ naturally constitutes a sequence whose order is inherited from $\widetilde{\mathcal C}^{\widehat B}_u$ (which is arranged according to increasing label and order of appearance within a trajectory, cf.~\eqref{eq:RI_Z} for a similar procedure).

Plainly, \eqref{def:mathscrE-concrete} implies \eqref{eq:mathscrEBmeasurable}. We now show \eqref{eq:finite_energy}. As noted below \eqref{def:omega-BPhiBu}, $\widetilde{\mathcal 
C}_u^{B}$ is measurable with respect to the truncated process $\omega_{B}^{-}$, obtained from $\omega$ by removing these excursions except for their start- and endpoints (which correspond to elements in $\widetilde{\mathcal 
C}_u^{B}$ whenever the underlying trajectory has label at most $u$). In particular, this implies that both $\widetilde{F}_B'$ and $\widetilde{F}_B''(I)$ for fixed $I$ are $\sigma(\omega_{B}^{-})$-measurable. Let $\widetilde{ Z}^B_u= \widetilde{ Z}^B_u(\omega)$ denote the (finite) sequence of successive excursions between $B$ 
and $\partial_{{\rm out}} B$ along with their associated labels. One now observes that for any $I \subset \widehat{B}$ such that $\P[\mathcal{I}^u \cap \widehat{B} = I, \widetilde{F}_{B}'  \cap \widetilde{F}_B'' (I)] > 0$, as we now explain,
\begin{multline}
\label{eq:fe-compu}
\P\big[B\subset \mathcal V^{u - \delta}, \,  \mathcal 
I^u \cap \widehat B =I, \, \widetilde{F}_B  \big| \, \sigma ( \omega_{\widehat{B}}^{-} )\,\big] \\ = \P\big[\text{range}\big(\widetilde{ Z}^{\widehat{B}}_{u-\delta} \cap B \big) = \emptyset, \,  \text{range}\big(\widetilde{ Z}^{\widehat{B}}_{u} \big) =I \,  \big| \, \sigma ( \omega_{\widehat{B}}^{-} )\,\big]1_{\{\widetilde{F}_B' , \widetilde{F}_{B}''(I)\}}\\
\stackrel{(*)}{\geq} \P\big[ \widetilde{ Z}^{\widehat{B}}_u= (w(\xi))_{\xi \in \widetilde{\mathcal C}^{\widehat B}_u} \big| \, \sigma ( \omega_{\widehat{B}}^{-} )\,\big]1_{\{\widetilde{F}_B' , \widetilde{F}_{B}''(I)\}} \stackrel{(**)}{\geq} e^{-C(r \vee {r_0})^{2d}} 1_{\{\widetilde{F}_B, \,  \mathcal 
I^u \cap \widehat B =I\}},
\end{multline}
from which \eqref{eq:mathscrEBmeasurable} readily follows upon integrating on any $\sigma ( 
\omega_{\widehat{B}^-} )$-measurable event. The inequality $(*)$ is an inclusion of events, which follows by the defining properties of the event $\widetilde{F}_B''(I)$; in plain words, if the excursions $ \widetilde{ Z}^{\widehat{B}}_u$ match precisely the sequence $(w(\xi))_{\xi \in \widetilde{\mathcal C}^{\widehat B}_u}$, which is deterministic upon conditioning on $ \omega_{\widehat{B}}^{-}$ and whose existence is guaranteed on the event $\widetilde{F}_{B}''(I)$, then both $B\subset \mathcal V^{u - \delta}$ and $ \mathcal I^u \cap \widehat B =I$ occur. To obtain $(**)$, one simply notes that under $\P[\, \cdot \, | \sigma ( \omega_{\widehat{B}}^{-} )]$, the excursions constituting $\widetilde{ Z}^{\widehat{B}}_u$ are independent and each distributed as lazy random walk bridge conditioned to stay inside $\widehat{B}$ until reaching its endpoint. The probability that such a bridge follows a fixed path $w$ is bounded from below by $(4d)^{-|w|}$. The event $\widetilde{F}_B'$ ensures that there are at most $Cr_0(r_0+r)^{d-1}$ different bridges to consider, each of which follows a path of length at most $C(r\vee r_0)^d$ due to $\widetilde{F}_B''(I)$, and \eqref{eq:fe-compu} follows.

It remains to show \eqref{e:F_B-proba}. We seize the opportunity to show slightly more, namely that $\widetilde{F}_B$ in \eqref{def:mathscrE-concrete} is implied by another event $\widehat{F}_B$ satisfying \eqref{e:F_B-proba} with explicit monotonicity properties; see also Remark~\ref{R:FE-implication-event} below. For $B= B(x,r)$ as above and integer  ${r_0}\ge 1$, we let 
\begin{equation}\label{def:ABg}
\begin{split}
&A(B, {r_0}) =  B(x, r + 5{r_0}) \setminus B(x, r + 3{r_0}),\\
&\widetilde{A}(B, {r_0}) =  B(x, r + 6{r_0}) 
\setminus B(x, r + 2{r_0}).
\end{split}
\end{equation}
We now introduce for $u_1 \ge u_2 \ge u_3 > \delta_2 > \delta_1$ positive numbers, the
event  $\widehat{{F}}_{B}(u_1, u_2, u_3, \delta_1, \delta_2, r_0)$ under $\P$ as the 
intersection of the following three events (keeping the dependence on the underlying parameters implicit):
	\begin{equation}
	\begin{split}\label{eq:fin_energy_good}
				&\widehat{{F}}_{B}^1
\stackrel{{\rm def.}}{=} \bigcap_{x, y \in \mathcal I^{u_2 - \delta_1} \, \cap \, A(B, {r_0})}\big\{  \lr{}{\mathcal I^{u_2} \, \cap\,\widetilde{A}(B, {r_0})}{x}{y}  \big\}, \\ 				
&\widehat{{F}}_{B}^2 \stackrel{{\rm def.}}{=} \big\{ (\mathcal I^{u_3 - \delta_1} \setminus\mathcal I^{u_3 - \delta_2})\cap B^{r_0}\ne \emptyset\big\}, \mbox{ and}\\[0.3em]
&\widehat{{F}}_{B}^3
\stackrel{{\rm def.}}{=} \bigcap_{x \in B^{8{r_0}}} \{\ell_x^{u_1} \le {r_0}\},
		\end{split}
\end{equation}
where $B^s = \{x \in \Z^d : d(z , B) \le s\}$ and $(\ell_x^u: x \in \Z^d, u > 0)$ denote the occupation times of the interlacement, see \eqref{e:ell-u}. We now claim that one has the inclusion
\begin{equation}\label{eq:fin_en_inclusion}
\widehat{F}_{B} \stackrel{\text{def.}}{=} \widehat{F}_{B}(u, u, u, \delta/2, \delta, r_0) \subset \widetilde{F}_B
\end{equation}
with $\widetilde{F}_B$ as defined in \eqref{def:mathscrE-concrete}. Once \eqref{eq:fin_en_inclusion} is shown, \eqref{e:F_B-proba} follows using \cite[Theorem 5.1]{DPR22} to bound $\P[(\widehat{{F}}_{B}^1)^c]$, combining the formula \eqref{eq:I_u}, the fact that $\mathcal{I}^u \setminus \mathcal{I}^v$ has the same law as $\mathcal{I}^{u-v}$ for $u>v$ and the bound $\text{cap}(B^{r_0}) \geq c(r \vee r_0)^{d-2}$ to deal with $\P[(\widehat{{F}}_{B}^2)^c]$, and applying a union bound together with a straightforward large-deviation estimate to bound $\P[(\widehat{{F}}_{B}^3)^c]$, observing that $\ell_0^u$ follows a compound Poisson distribution.

We now turn to the proof of \eqref{eq:fin_en_inclusion}. Recalling that $\widehat{B}=B(x, r + 7{r_0})$, first notice that $$\widehat{F}_{B} 
\subset \widehat{{F}}_{B}^3 \subset \big\{|\widetilde{\mathcal C}^{\widehat B}_u|  \le {r_0}\,|\partial_{{\rm out}} \widehat B|\big\} = \widetilde{F}_B'$$ 
and thus we only need to show the inclusion 
\begin{equation}\label{eq:FGGBinclusion}
\widehat{F}_{B} \subset  \widetilde{F}_B'' (\mathcal{I}^u \cap \widehat{B}),
\end{equation}
which rests on a purely combinatorial argument. To this end for any given 
$\xi \in \widetilde{\mathcal C}^{\widehat B}_u$ on the event $\widehat{F}_{B} $, we propose a way to choose an 
excursion 
$w=w(\xi)$ between $\widehat{B}$ and $\partial_{{\rm out}} \widehat{B}$ by considering three mutually exclusive and exhaustive cases. We then proceed to check that the excursions $w(\xi)$ have the properties required for $\widetilde{F}_B'' (\mathcal{I}^u \cap \widehat{B})$ to occur. Let $W(\xi)$ denote the excursion in $\widetilde{ Z}_u^{\widehat{B}}$ corresponding to $\xi =(x,y,v) (\in \widetilde{\mathcal C}^{\widehat B}_u)$.
\begin{enumerate}[label = {\em Case~\arabic*.}]
\item 
${\rm range}(W(\xi)) \cap B^{{r_0}} = \emptyset$ or $v \in (u - \frac{\delta}2, u]$. In this case, we choose 
$w$ to be the concatenation of $x$,  $\tilde w$ and $y$ where $\tilde w$ is a minimum length traversal of some spanning tree 
of 
the connected set ${\rm range}(W(\xi))$. Clearly, $$|w| \le 2\,(\text{number of edges in $\widehat B$}) + 2 \le (20 (r \vee {r_0}))^d.$$
\item ${\rm range}(W(\xi)) \cap B^{{r_0}} \ne \emptyset$ and $v \le u - \delta$. In this case there is a path 
$\gamma \subset \mathcal I^u \cap (\widehat B \setminus B)$ connecting 
$x$ and $y$. For, otherwise, any path connecting $x$ and $y$ must intersect $B$. In particular, this is true of the 
excursion $W(\xi)$, which is in $\widetilde{ Z}^B_{u - \delta}$. However, since $\widehat{F}_{B} \subset \widehat{F}_{B}^1$ occurs, it follows from the local 
connectivity implied by the latter, see~\eqref{eq:fin_energy_good}, that there is a path $\gamma \subset \mathcal I^u \cap (\widehat B 
\setminus B)$ connecting 
$x$ and $y$, a contradiction. We now define 
$w$ in a similar way as in the previous case with 
$W(\xi)$ replaced by a path 
$\gamma \subset \mathcal I^u \cap (\widehat B \setminus B)$ joining $x$ and $y$  (which we just showed exists).  In particular, 
$|w| \le (20 (r \vee {r_0}))^d$ as above.

\item 
${\rm range}(W(\xi)) \cap B^{{r_0}} \ne \emptyset$ and $v \in  (u-\delta, u - \frac{\delta}{2}]$. As in 
Case~2, using the local connectivity ensured by $\widehat{F}_{B}^1$, we can find a connected set 
$C(\xi) \subset \mathcal I^u \cap \widehat B$ such that 
$${\rm range}(W(\xi)) \, \cup \, \bigcup_{\substack{\xi' = (v',x',y'):\, v' \le u-\delta,\\  W(\xi') \cap B^{{r_0}} \ne \emptyset}} {\rm range}(W(\xi')) \subset C(\xi).$$
Now define 
$w=w(\xi)$ similarly as in Case~2 (or Case~1) with 
$C(\xi)$ substituting for 
${\rm range}(W(\xi))$. The required bound on $|w|$ still holds.
\end{enumerate}
By our treatment of Case~2, it immediately follows that
\begin{equation*}
\bigcup_{\xi =(x,y,v): \, v \le u - \delta}\, {\rm range}(w (\xi)) \cap B = \emptyset.\end{equation*}
On the other hand, as $\widehat{F}_{B} \subset \widehat{{F}}_{B}^2$, see \eqref{eq:fin_energy_good}, there is at least one 
$\xi \in \widetilde{\mathcal C}^{\widehat B}_u$ falling under Case~3. Since the only excursions for which 
${\rm range}(W(\xi)) \not\subset {\rm range}(w(\xi))$ are those falling under Case~2, we deduce from our treatment of Case~3 that
$$\bigcup_{\xi}\,\, {\rm range}(w(\xi)) = \bigcup_{\xi}\,\, {\rm range}(W(\xi)) = \mathcal I^{u} \cap \widehat B.$$
In view of the definition of $\widetilde{F}_B'' (\mathcal{I}^u \cap \widehat{B})$, the last two displays together with the bound on $|w|$ implied in all three cases yield \eqref{eq:FGGBinclusion}. \qedhere
\end{proof}

\begin{remark}\label{R:FE-implication-event}
We record for later reference that the event $\widetilde{F}_B=\widetilde{F}_B^{u,\delta, r_0}$ entering Proposition~\ref{lem:finite_energy} 
(and later our bridging construction in defining a finite-energy good box, see e.g.~\eqref{def:mathscrE} and 
\eqref{e:K_b-tilde} below), satisfies $\widetilde{F}_B \supset \widehat{F}_{B}$, see \eqref{eq:fin_en_inclusion}, with 
$\widehat{F}_B$ defined as the intersection of the three events appearing in \eqref{eq:fin_energy_good} 
(with $u_1=u_2=u_3=u$, $\delta_1= \frac{\delta}{2}$, $\delta_2=\delta$) and
\begin{equation}
\label{e:F_B-proba-boosted} \P[\widehat{F}_B^c] \leq C(r \vee r_0)^d e^{-c (u \wedge \delta) r_0^c}.
\end{equation}
Whereas the event $\widetilde{F}_B$ is advantageous (notably due to \eqref{eq:mathscrEBmeasurable}) to perform the surgery arguments presented below, the event $\widehat{F}_B$ is cut out for renormalization-type arguments, for which monotonicity in terms of the various parameters involved is crucial.
\end{remark}

\section{Hierarchical bridges}\label{sec:bridging}
In this section, we construct a geometric object which we call a \textit{bridge}. It is a fractal set 
comprising boxes at all scales with several desirable features, expressed as conditions 
\eqref{B1}-\eqref{B4} below. Bridges will be used in Section~\ref{sec:supercritical} 
as an efficient highway to build connections between 
clusters. Here, `efficiency' refers to the cost of building a connection between two given clusters, which are arbitrary, and possibly very irregular. This cost will later need to be optimized under polynomial lower bounds on connectivity such as those appearing in Lemma~\ref{lem:twopointsbound}. The multiple scales involved in the construction of the bridge (rather than just using one scale) reflect this feature, which is characteristic of critical geometry.  We note in passing that the same bridge construction is also crucially at play in the companion article \cite{RI-I}, see Remark~\ref{R:bridge-other-paper} at the end of this section for more on this.

The existence of bridges with the desired properties \eqref{B1}-\eqref{B4} is the content of Proposition~\ref{prop:general_bridge} below, 
which is the main result of this section. Its proof follows a deterministic construction and applies to \textit{any} 
given pair of clusters, which the bridge connects (see condition \eqref{B2}). Part of the construction is somewhat 
reminiscent of {Whitney}-type covering lemmas; see~e.g.~\cite[Chap.~I, \S3.2, p.~15]{MR1232192}. 
The inherently hierarchical nature of a bridge, which will become apparent in the proof, see also Figure~\ref{F:bridge}, is ultimately 
owed to a delicate and limited decoupling, cf.~Proposition~\ref{prop:cond_decoup} or Proposition~\ref{lem:finite_energy}, 
which warrant polynomial safety gaps.

\medskip

The bridge construction will occur inside tube regions (including boxes as a special case), defined as follows. Let $z \in \Z^d$, $ 1\leq j \leq d$ a coordinate direction and ${N, L \ge 0}$ be integers. 
The ($\ell^{\infty}$-){\em tube} of length $N + 2L$ and (cross-sectional) radius $L$ in the $j$-th coordinate direction is the set
\begin{equation}
\label{eq:def-tube}
T_{L, N}^j(z)= \bigcup_{0 \le n \le N}\, B(z + ne_j, L). 
\end{equation}

Let $\mathcal C$ and $\mathcal D$ be two disjoint subsets of $\Z^d$ ($d \ge 1$) that both intersect the tube 
$T = T_{L, N}^j(z)$ in \eqref{eq:def-tube}. A {\em bridge} associated to the septuple 
$(\mathcal{C} , \mathcal{D}, s, s', m, \xi, T)$, where $s, s', m \in (0,\infty)$ and $\xi \in (0, 1)$, is a 
collection of boxes $\mathbb B = \bigcup_{1\leq j \leq J}  \mathbb{B}_j$, for some integer $J\geq 1$ (the index $j$ should be thought of increasing the resolution, i.e.~boxes in $\mathbb{B}_j$ get smaller as $j$ grows), each  
 contained in $B(T, s)$ and having the following properties. Letting $\mathbb{H}= \mathbb{B}_J$ (the `holes'), each box in $\mathbb B_{-J} = \mathbb B \setminus \mathbb H= \bigcup_{1\leq j \bm {<} J} \mathbb{B}_j$ is a subset of $T$ 
 containing two marked points on its boundary. Moreover, the following hold.
 \begin{align}
 & \label{B1} \tag{B.1} \text{\parbox{14cm}{ {\em (Separation).} For $ 1\leq j < J$ and $B = B(x,r) \in \mathbb B_j$, the box $\widetilde{ B 
 	}\stackrel{\text{def.}}{=} B(x,r + \lceil r^{\xi} \rceil )$ is disjoint from all boxes $\overline{B'}$ (recall $\overline{K}$ denotes the closure of $K \subset \Z^d$, see \S\ref{subsec:crossing_interface} for notation), where $ \textstyle B' \in (\bigcup_{j' \leq j} \mathbb B_{j'}) 
 	\setminus \{B\}$, as well as from $ \overline{\mathcal{C}} \cup \overline{\mathcal{D}}$. Moreover, if $B(x, r)$ and $B(x', r')$ are two distinct boxes in $\mathbb H$, then $B(x, r + s') \cap B(x', r' + s')= \emptyset$.}}\\[0.2cm]
& \label{B2} \tag{B.2}\text{\parbox{14cm}{ {\em (Connectivity).} For $1 \leq j < J$ and $B \in \mathbb B_j$, if $\pi_B$ is 
 	any path connecting the two marked vertices of $B$, then the union of all such paths along with the boxes in $\mathbb H$ connects $\mathcal C$ and $\mathcal D$, i.e.~$$\mathcal{C} \cup \Big( \bigcup_{B \in \mathbb B \setminus \mathbb H} \pi_B \Big) \cup \Big( \bigcup_{B \in \mathbb H} B \Big) \cup \mathcal{D}$$ is a connected set (here and routinely below we identify a path $\pi_B$ with its range).}}
\end{align}
\begin{align}
& \label{B3} \tag{B.3} \text{\parbox{14cm}{ {\em (Size).} If $B= B(x,r)$ satisfies $ B \in \mathbb B \setminus \mathbb H$ then $ r \geq s'$ whereas if $ B \in \mathbb H$ then $r \leq s$.}}\\[0.2cm]
	&  \label{B4}  \tag{B.4} \text{\parbox{14cm}{ {\em (Complexity).} \, $|\mathbb B| \leq (N/L + 8d \log eL)(\log eL)^m$ and $J \le m\log \log e^2 L$.}}
 \end{align}

The following proposition is the main result of this section. We implicitly assume that 
$\mathcal C$, $\mathcal D$ and $T$ are related as stated above, i.e.~$\mathcal C$, $\mathcal D \subset \Z^d$, $\mathcal C \cap \mathcal{D}= \emptyset$ and $\mathcal C \cap T\neq \emptyset$, $\mathcal D \cap T\neq \emptyset$.

 \begin{proposition}
\label{prop:general_bridge}
For all $\xi \in (\frac12, 1)$, there exists $m = m(\xi) > 1$ such that for all $L  \ge 
2s \geq C(\xi)$ and with $s' = s^{1/4}/200$, there is a bridge $\mathbb{B}$ associated to $(\mathcal{C}, \mathcal{D}, s, s', m, \xi, T)$.
\end{proposition}

 \begin{proof}
By invariance under translations and lattice rotations, we may assume that $T=T_{L, 
N}^1(0)$. As we first explain, it is sufficient to work in the continuum, which is a matter of convenience.
 We thus consider $T$ as well as the boxes appearing in \eqref{B1}-\eqref{B4} as 
(closed) subsets of $\R^d$. With regards to giving sense to \eqref{B2}, we identify $\pi_B$ and $\mathcal{C}, \mathcal{D}$ with the connected subsets of $\R^d$ obtained by adding line segments between all neighboring pairs of points. With these conventions, \eqref{B1}-\eqref{B4} are naturally declared in $\mathbb{R}^d$. We will construct a bridge $\mathbb B$ satisfying the conclusions of Proposition~\ref{prop:general_bridge} in this continuous setup, but with
$s'= \frac{s^{1/4} }{100}$ instead (note that the conditions \eqref{B1} and \eqref{B3} become more stringent as $s'$ increases) and 
  $2\lceil r^{\xi} \rceil$ in place of $\lceil r^{\xi} \rceil$ as well as $\mathcal C \cup \mathcal D$ in place of $\overline{\mathcal C} \cup \overline{\mathcal D}$ in \eqref{B1}. The  discrete case then follows by taking lattice approximations of the 
corresponding boxes, thereby only increasing the
radius of the boxes in $\mathbb B$, in order for \eqref{B2} to continue to hold. Property 
\eqref{B4} is unaffected by this. The slightly stronger continuous result (with larger value of $s'$ and radius for $\widetilde{B}$) ensures that the lattice effects resulting from this approximation, which may cause the radii of any box to increase additively by a bounded amount, are duly accounted for, i.e.~the requirements \eqref{B1}, \eqref{B3} for the resulting discrete bridge hold whenever $s \geq C(\xi)$ (possibly replacing $s$ by $s+C'$ in the process of passing to the discrete framework).

\medskip
We now work within the above continuous setup and start with a reduction step (\eqref{eq:bridge-pf1} below).
We denote by $\partial_{{\, \rm L}} T$, resp.~$\partial_{{\, \rm R}} T$ the left, resp.~right face of $T$, i.e.~if $T$ is the (closed) continuous tube corresponding to $T_{L, 
N}^1(0)$ in \eqref{eq:def-tube}, then $\partial_{{\, \rm L}} T= \{-L\} \times [-L, L]^{d-1}$ and $ \partial_{{\, \rm R}} T=\{N +L \} \times [-L, L]^{d-1}$. The case of generic $T_{L, N}^j(z)$ is analogous. 
For any $S \subset \R^d$, we use $\mathring{S}$ and 
$\overline S$ for the topological interior and 
closure of $S$, respectively. We claim it is enough to show for $L \ge  s \ge C (\xi)$ that there exists a bridge
associated to $(\mathcal C, \mathcal D, s, s', m(\xi), \xi, 
 T)$ with $s' = \frac{{s}^{1/2}}{ 100}$ under the additional assumption that
 \begin{equation}\label{eq:bridge-pf1}
 \text{$(\mathcal C \cup \mathcal D)  \cap \mathring{T} = \emptyset$, 
$\mathcal C \cap \partial_{{\, \rm L}} T \ne \emptyset$ and $\mathcal D \cap \partial_{{\, \rm R}} T \neq \emptyset$.} 
\end{equation}
We first explain how to derive the general case from \eqref{eq:bridge-pf1}.

\begin{claim}[$L \geq 2s, \, s \geq C(\xi)$] \label{Claim:T'}
There exists a tube $T' = T_{L', N'}^{j'}(z') \subset T$ such that \eqref{eq:bridge-pf1} holds with $T'$ in place of $T$ and in addition,
\begin{equation}\label{eq:bridge-pf2}
\text{$L' \ge s \, (\ge C(\xi))$ or $N' + 2L' < 2 s$.}
\end{equation}
\end{claim} 

Once such at $T'$ is at our disposal, we conclude as follows. In case $L' \geq s$ we simply define $\mathbb{B}$ as the bridge associated to $(\mathcal C, \mathcal D, s, s', m, \xi, T')$, which exists by assumption as $T'$ satisfies \eqref{eq:bridge-pf1}. One then simply notes that  $\mathbb B$ thus satisfies \eqref{B1}-\eqref{B4} for $(\mathcal C, \mathcal D, s, s', m, \xi, T)$ and the claim of Proposition~\ref{prop:general_bridge} follows. On the 
 other hand if $L' < s$, then by \eqref{eq:bridge-pf2} we have that $N' + 2L' \le 2 s$, hence we can simply set $J = 1$ and  $ \mathbb B= \mathbb H = \{B(x, s)\}$ where $x$ is the center of the tube $T'$.

\begin{proof}[Proof of Claim~\ref{Claim:T'}]
Let $y_{\mathcal C} \in \mathcal C \cap T$ and $y_{\mathcal D} \in \mathcal D \cap T$ be such that $$|y_{\mathcal C}- y_{\mathcal D}| = d\big(\mathcal C \cap 
 	T, \mathcal D \cap T\big) \stackrel{\text{def.}}{=} D$$ 
(recall that $d(\cdot,\cdot)$ denotes the $\ell^{\infty}$-distance between sets). Thus, we have  that $y_{\mathcal D}^{j'} = y_{\mathcal C}^{j'} + D$ for 
 	some $ 1\leq j' \leq d$ (this defines $j'$ entering the definition of $T'$). We now distinguish two cases. If $j'=j$ (the direction $j$ refers to $T=T_{L, N}^j(z)$) and $D \geq 2L$, we define $T'$ as the tube with $L'=L$ whose
boundaries $\partial_{{\, \rm L}} T'$ and $\partial_{{\, \rm R}} T'$ contain $y_{\mathcal C} $ and $y_{\mathcal D} $, respectively. Note that this can always be accomodated (i.e.~$N' \geq0$) since $D \geq 2L$. Clearly, $T' \subset T$ and \eqref{eq:bridge-pf1} holds for $T=T'$ by construction, and the first condition in \eqref{eq:bridge-pf2} is in force (in fact $L'=L \geq 2s$).

The remaining case is that either i) $j' \neq j$ or ii) $j' = j$ and $D < 2L$. Since the cross-section of $T$ (orthogonal to $e_j$, thus corresponding to directions $j' \neq j$) is $2L$, one has $D \leq 2L$ regardless of whether i) or ii) occurs. 
One thus finds a box $B$ of radius $\frac D2$ contained in $T$ having $y_{\mathcal C}$ and $y_{\mathcal D}$ on opposite faces: for instance, $B$ can be obtained by considering the rectangular cuboid having $y_{\mathcal C}$ and $y_{\mathcal D}$ at opposite corners, extending its short directions (all except $j'$) to obtain a box of desired radius (but not necessarily in $T$) and rigidly shifting it one by one in all but the $j$'th direction to obtain $B \subset T$ (using that $D \leq 2L$). We set $T'=B$, whence $L' = \frac{D}{2}$ and $N' = 0$. Again, \eqref{eq:bridge-pf1} is plain and in case $L'  < s$, we have that $N'+2L' = 2L' < 2s$, whence \eqref{eq:bridge-pf2}.
\end{proof}

 In the remainder of the proof, we confine ourselves to the special case where $T(=T_{L, 
N}^1(0))$ satisfies \eqref{eq:bridge-pf1}. Thus, let $y_{\mathcal C} \in \mathcal C \cap \partial_{{\, \rm L}}  T$ and $y_{\mathcal D} \in \mathcal D \cap \partial_{{\, \rm R}} T$. For $n \in \Z$, consider the coarse-grained lattice ${\Lambda}_n 
 = L_n\,\Z^d$ with 
 \begin{equation}\label{eq:bridge-L_n}
 L_n = 2^{-n} L, \quad n \geq 0,
 \end{equation} 
 and let $B_{n,x} = x + [0, L_n)^d$ for any $x \in \Z^d$ 
 (not to be confused with $B(x, r)$). Observe that the families of (semi-closed) boxes $\mathbf B_n = 
 \{B_{n,x} : x \in {\Lambda}_n\}$  are naturally nested and each forms a tiling of $\R^d$.  We will 
 refer to a box in $\mathbf B_n$ as an {\em $n$-box}. We denote by $\mathbf B= \bigcup_{n \ge 0} \mathbf B_n$ the collection of all $n$-boxes as $n$ varies.  Two boxes in $ B, B' \in \mathbf B$ are called {\em adjacent} if $B \cap B' \notin\{ B, B'\}$ and they are `next' to each other, i.e.~$\overline{B} \cap \overline{B'}$ is a face of one of $\overline{B}$ and $\overline{B}'$. It readily follows that ,given any two boxes in $\mathbf B$, either one of them contains  the other, or they are adjacent, or they are disjoint but non-adjacent. 
A sequence $\gamma: \{1, \ldots, \ell\} 
 \to \mathbf B$ will be called {\em coarse path} if the boxes $\gamma(i)$ and $\gamma(i+1)$ are adjacent for all $i$ 
 and $\ell=\ell(\gamma)$ is called the length of $\gamma$. The coarse path $\gamma$ is {\em simple} if all its boxes 
 are disjoint. We will use the following result.
 \begin{lemma}\label{lem:coarse_exist}
For $L \geq s \geq C$, there exists a simple coarse path $\gamma $ such that, with $\ell=\ell(\gamma)$,	\begin{enumerate}
 		\item[i)] The box $\gamma(1)$ (resp.~$\gamma(\ell)$) 
 		is adjacent to a box of same length containing $y_{\mathcal C}$ (resp.~$y_{\mathcal D}$). 
 		\item[ii)] For all $1 \leq i \leq \ell $, $\gamma(i) \subset T$, $\gamma(i) \in \mathbf{B}_n$ for some 
 		$0 \le n \le \lceil \log_2 (16 L/ s) \rceil$ and if $|i - i'| \le 1$, then $\gamma(i') \in \mathbf{B}_m$ for some $m \in \{n -1, n, n+1\}$.
		\item[iii)] 
 		$\gamma(1)$ and $\gamma(\ell)$ are 
 		$\lceil \log_2 (16L/ s) \rceil$-boxes (so~their side lengths are at most 
 		$\frac s {16}$, see \eqref{eq:bridge-L_n}).
 		\item[iv)] $2 \le \ell= \ell(\gamma) \le \frac N L + 5 d \log_2 \frac{64L}{s}$. 
 		 	\end{enumerate}
 \end{lemma}
 
 The proof reminiscent of the bridge construction of \cite[Lemma 2.5]{DCGRS20}, but simpler.
 
 \begin{proof}[Proof of Lemma~\ref{lem:coarse_exist}]
 	We will construct $\gamma$ in a hierearchical manner through progressive refinements (recall that increasing $n$ corresponds to an increasing resolution in \eqref{eq:bridge-L_n}). The starting point is the `very' coarse path $\gamma_0$, defined as a shortest-length path with values in $\mathbf{B}_0$ connecting the unique $0$-boxes containing $y_{\mathcal C}$ 
 	and $y_{\mathcal D}$. Since $T= T_{L, N}^{1}(0)$ and on account of \eqref{eq:def-tube}, we can choose $\gamma_0$ in such a way so that all the boxes in $\gamma_0$ except, possibly, the 
 	initial and terminal ones $\gamma(1)$ and $\gamma(\ell(\gamma_0))$, lie inside $T$. It is clear from this construction that $3 \le \ell(\gamma_0) \le \frac N  L  +  4d$.  
	
	Now suppose that at the end of stage $k \ge 0$, we have obtained a simple coarse path $\gamma_k: \{0,\dots , \ell_k\} \to \bigcup_{0 \leq n \leq k} \mathbf{B}_n$, so $\ell_k = \ell(\gamma_k)$, having the following properties:
\begin{align}
&\label{eq:gammak-1} \text{$\gamma_k(1) \ni  y_{\mathcal C}$,  $\gamma_k(\ell_k) \ni  y_{\mathcal D}$ and $\gamma_k(i) \in \mathbf{B}_k$ for $i \in \{1,2,3, \ell_k-2, \ell_k-1, \ell_k\}$.}\\
&\label{eq:gammak-2} \text{$\gamma_k(i) \subset T$ for $1 < i < \ell_k$, and with $\gamma_k(i) \in \mathbf{B}_{n_i}$ one has $\textstyle  \frac{n_{i+1}}{n_i} \in \{ \frac12, 1, 2\}$ for $1 \leq i < \ell_k$.} \\
&\label{eq:gammak-3} \text{$\textstyle3 \le \ell_k \le \frac N L + 4(k+1)d$ and $4 \leq \ell_k$, $k \geq 1$.}
\end{align}  
It is plain that $\gamma_0$ satisfies \eqref{eq:gammak-1}-\eqref{eq:gammak-3} for $k=0$. We will momentarily construct $\gamma_{k + 1}$ inductively from $\gamma_k$ to deduce the existence of a $\gamma_k$ with the above features for all $k \ge 0$. Once the existence of $\gamma_k$ for all $k \ge 0$ is established, we simply define $\gamma$ to be the coarse path inside $T$ obtained from $\gamma_{\lceil \log_2 (16L/ s) \rceil}$ (recall that $L \ge s$) after removing the first and final boxes (containing $y_{\mathcal C}$ and $y_{\mathcal D}$, respectively, due to \eqref{eq:gammak-1}), should they lie outside $T$. It is clear from this construction and on account of \eqref{eq:gammak-2} that the resulting path $\gamma$ satisfies \textit{i)}. Item \textit{ii)} also follows from \eqref{eq:gammak-2} and the choice of depth $k= \lceil \log_2 (16L/ s) \rceil$, by which \eqref{eq:gammak-3} immediately yields the upper bound in \textit{iv)}. The lower bound $\ell \geq 2$ follows from the fact that $T$ has diameter at least $2L \geq 2s$ whereas $\gamma_k(1)$ and $\gamma_k(\ell)$ have side length at most $\frac s{16}$, 
whence $\ell(\gamma_{k}) \geq 4$ and thus $\ell(\gamma) \ge 2$. Finally \textit{iii)} follows directly from \eqref{eq:gammak-1}.
 	
 	\medskip
 	
 	We now prove the induction step. We construct $\gamma_{k+1}$ from $\gamma_k$ by retaining most of it 
 	while refining the boxes at its both ends.
 	Since $L_k$ divides $L$, see \eqref{eq:bridge-L_n}, it follows from the definition of the coarse-grained lattice $ 
 	\L_k$ and the tube $T$ that if a $k$-box $B_{k,x}$  intersects $\partial_{{\, \rm R}}T$ and 
 	has a neighboring $k$-box {inside} $T$, then this
 	neighbor is unique and in fact equals $B_{k, x - L_k e_1}$. The previous observation applies to $B_{k,x}=\gamma_k(\ell_k)$ by \eqref{eq:gammak-1} and since $y_{\mathcal D} \in \mathcal D \cap \partial_{{\, \rm R}} T$.
 Now two cases might occur based on the value of 
 	$x^1$, the first coordinate of the base point $x$ for the box $\gamma_k(\ell_k)$. 
	
	Firstly, we may have $x^1 < N+L-L_{k+1}$, in which 
 	case the left half of $B_{k,x}$, i.e.~the set $x + [0, L_{k+1}) \times [0, L_k)^{d-1}$ is contained in $T$. 
 	We then construct a simple coarse path consisting of three $(k+1)$-boxes $B_{{\rm R}, -2}, B_{{\rm 
 			R}, -1}$ and $B_{{\rm R}}$, each contained inside $B_{k,x}$, that connects $\partial_{{\, \rm L}}B_{k,x} = \partial_{{\, \rm R}} (\gamma_k(\ell_k-1)$) to 
 	$B_{{\rm R}} \subset B_{k,x}$, the unique $(k+1)$-box containing $y_{\mathcal C}$. This actually only requires two boxes $B_{{\rm 
 			R}, -1}$ and $B_{{\rm R}}$ but we add a third box $B_{{\rm R}, -2}$ neighboring both $B_{{\rm R}, -1}$ and $ \gamma_k(\ell_k-1)$ with a view towards \eqref{eq:gammak-1}. 
	
	 On the other hand, if $x^1 \ge N + L - 
 	L_{k+1}$, $B_{{\rm R}}$ defined as above is adjacent to $\partial_{{\, \rm L}}B_{k,x}= \partial_{{\, \rm R}}  
 	(\gamma_k(\ell_k-1))$ and hence we can choose a coarse path consisting of two $(k+1)$-boxes 
 	$B_{{\rm R}, -2}$ and $B_{{\rm R}, -1}$ inside $\gamma_k(\ell_k-1)$ joining the face of 
 	$\gamma_k(\ell_k-1)$ adjacent to $\gamma_k(\ell_k-2)$ and $B_{{\rm R}}$.  By the same reasoning, we find three $(k+1)$-boxes $B_{{\rm L}}$, $B_{{\rm L}, 1}$ and $B_{{\rm L}, 2}$ all contained in $\gamma_k(1)$, where $B_{{\rm L}}$ is the unique $(k+1)$-box containing $y_{\mathcal C}$. We now define the path 
 	$\gamma_{k+1}$ by considering several cases.
 	 	
		\medskip
 	{Case~(a): $B_{{\rm R}, -2}, B_{{\rm R}, -1} \subset \gamma_k(\ell_k)$.} In this 
 	case we set $\gamma_{k+1}(1) = B_{{\rm L}}, \gamma_{k+1}(2) = B_{{\rm L}, 1}$ and $\gamma_{k+1}(3) = B_{{\rm L}, 2}$; $\gamma_{k+1}(\ell_k + 4) = B_{{\rm R}}, \gamma_{k+1}(\ell_k + 3) = B_{{\rm R}, -1}$ and $\gamma_{k+1}(\ell_k + 2) = B_{{\rm R}, -2}$; and $ 	\gamma_{k+1}(i) = \gamma_{k}(i - 2)$ for all $4 \le i \le \ell_k + 1$.
 	
		\medskip
 	
 	{Case~(b): $B_{{\rm R}, -2}, B_{{\rm R}, -1} \subset \gamma_k(\ell_k-1)$, $\ell_k > 3$.} In 
 	this case we set $\gamma_{k+1}(1) = B_{{\rm L}}, \gamma_{k+1}(2) = B_{{\rm L}, 1}$ and $\gamma_{k+1}(3) 
 	= B_{{\rm L}, 2}$; $\gamma_{k+1}(\ell_k + 1) = B_{{\rm R}}, \gamma_{k+1}(\ell_k) = B_{{\rm R}, -1}$ and $\gamma_{k+1}(\ell_k - 1) = B_{{\rm R}, -2}$; and $ 	\gamma_{k+1}(i) = \gamma_{k}(i - 2)$ for all $4 \le i \le \ell_k$.
 	
		\medskip
 	
 	{Case~(c): $B_{{\rm R}, -2}, B_{{\rm R}, -1} \subset \gamma_k(\ell_k-1)$, $\ell_k = 3$.} In this case, $B_{{\rm R}, -2} \subset \gamma_k(1)$ and hence we can connect 
 	$\partial B_{{\rm L}, 2}$ 
 	with $B_{{\rm R}, -2}$ using a coarse path in $\gamma_k(1)$ of size at most $d + 1$. We then 
 	set $\gamma_{k+1}$ to be a simple coarse path in $\mathbf B_{k+1}$ formed by these $(k+1)$-boxes along with
	$B_{{\rm L}}, B_{{\rm L}, 1},B_{{\rm L}, 2}$ and $B_{{\rm R}, -2}, B_{{\rm R}, -1}, B_{{\rm R}}$. We stress that this includes the possibility of having overlaps among $B_{{\rm L}, 1},B_{{\rm L}, 2}$ and $B_{{\rm R}, -2}, B_{{\rm R}, -1}$, in which case the additional piece of path in $\gamma_k(1)$ is simply absent.
 	
		\medskip
 	
 	It follows readily from the above construction that the path $\gamma_{k+1}$ hereby defined satisfies properties
\eqref{eq:gammak-1} and \eqref{eq:gammak-2} with $k+1$ in place of $k$. As to \eqref{eq:gammak-3}, one sees plainly that $\ell_{k+1} = \ell_{k}+4$ in Case~(a), $\ell_{k+1} = \ell_{k}+1$ in Case~(b) and $\ell_{k+1} \leq d+1+6$ in Case~(c). The lower bound $\ell_{k+1} \geq 4$ follows from  $\ell_k \geq 3$ (which holds for all $k \geq 0$)  in Cases~(a) and~(b). For Case~(c) one observes that the `worst-case' scenario is $B_{{\rm L}, 1}=B_{{\rm R}, -2}$, $B_{{\rm L}, 2}=B_{{\rm R}, -1}$, in which case $\gamma_{k+1}$ is composed of the four boxes $B_{{\rm L}}, B_{{\rm L},1}, B_{{\rm L},1} $, whence $\ell_{k+1} \geq 4$.
 \end{proof}

 We are just one step away from proving Proposition~\ref{prop:general_bridge}. This step involves the proof 
 of the proposition in a very special case.
 
  \begin{lemma}[$\xi \in (\frac12, 1)$]\label{lem:bridging_super_specialized}
Let  $\mathcal C = \partial_{{\, \rm L}}T$ and $\mathcal D = \partial_{{\, \rm R}}T$. For some $m=m(\xi)$ and all $L \ge s \ge C(\xi)$, there is a bridge associated to $(\mathcal C, \mathcal D, s, s', m, \xi, T)$ with $s' = \frac{ \sqrt{s}}{ 100}$. Furthermore, all the boxes in the bridge can be chosen to lie inside $T$.
 \end{lemma}
 Assuming Lemma~\ref{lem:bridging_super_specialized}, we first finish the proof of Proposition~\ref{prop:general_bridge}.
 	Let $\gamma$ be the coarse path supplied by 
 	Lemma~\ref{lem:coarse_exist}. In the sequel, $B_{{\rm L}}$ and $B_{{\rm R}}$ refer to the boxes adjacent to $\gamma(1)$ and $\gamma(\ell)$ containing $y_{\mathcal C}$ and $y_{\mathcal D}$, respectively; cf.~item \textit{i)} above. In order to ensure the separation property \eqref{B1}, we now slightly reduce the size of boxes comprising $\gamma$ to obtain a path $\gamma'$, as follows. Let $n_i$ be such that $\gamma(i)	\in \mathbf{B}_{n_i}$, by which $\gamma(i)$ has side length $L_{n_i}$. For any 
 		$1< j < \ell= \ell(\gamma)$ we simply let ${\gamma}'(j) = B(x, r - 4r^{\xi})$, where $B(x, r) = \overline{\gamma(j)}$ (so $x$ refers to the center of $\gamma(j)$ and $r= {L_{n_j}}/{2}$). 
		
		For $\gamma'(1)$,~resp.~$\gamma'(\ell)$, in addition to ensuring a small gap to other boxes in $\gamma'$, we also aim to preserve adjacency to $B_{{\rm L}}$, resp.~$B_{{\rm R}}$. To this end we proceed as follows. Let $r=\frac{L_{n_1}}{2}$ denote the radius of $\gamma(1)$. If $d({\gamma}(1),  \partial_{{\, \rm L}}T) \ge r$, we let ${\gamma}'(1) \subset {\gamma}(1)$ denote the (unique) box with radius $r - 4 r^{\xi}$ that is 
at Euclidean distance $4 r^{\xi}$ from each face of ${\gamma}(1)$ except for two, one which is shared with (i.e.~a subset of) $B_{{\rm L}}$ and the other which is at distance $8r^{\xi}$. Otherwise, we let ${\gamma}'(1) = B(x, r - 4r^{\xi})$ where $B(x, r) = \overline{\gamma(1)}$ and, with a slight abuse of notation, reset $B_{{\rm L}}$ to be any box of radius $r$ obtained by  shifting $B_{{\rm L}}$ in such a way that it still contains $y_{\mathcal C}$ as well as one face of ${\gamma}'(1)$. This is possible because $r + 4r^{\xi} < 2r $ for all $r \ge C(\xi)$. We proceed analogously with ${\gamma}(\ell)$, $y_{\mathcal D}$ and $B_{{\rm R}}$. 

Overall, it follows by construction of $\gamma'$, using the properties \textit{i)} and \textit{ii)} of $\gamma$, that
 	\begin{equation}\label{eq:bridge_gamma'-diest}
 		d\Big({\gamma}'(j), \bigcup_{1 \, \le \, j' \ne j \, \le \, \ell} {\gamma}'(j') \Big) \wedge d({\gamma}'(j), {\mathcal C} \cup {\mathcal D} \color{black})  \ge 4 \, {\rm rad}({\gamma}'(j))^{\xi},
 	\end{equation}
 and ${\gamma}'(1)$ (resp.~${\gamma}'(\ell)$) is adjacent to $B_{{\rm L}}$ (resp.~$B_{{\rm R}}$), which contains $y_{\mathcal C}$ (resp.~$y_{\mathcal D}$). We then set 
\begin{equation}\label{eq:bridge-fifinal-1}
\mathbb B_1 = \{{\gamma}'(j) : 1 \le j \le \ell\}. 
\end{equation}
We will describe the marked points for the boxes in $\mathbb B_1$ shortly (following \eqref{eq:bridge-fifinal-3}). As item \textit{ii)} in Lemma~\ref{lem:coarse_exist} ensures that the length of any box in ${\gamma}$ is at least $\frac{s}{32}$ and the ratio of the side lengths of ${\gamma}(j)$ and ${\gamma}(j+1)$ lies in $\{1/2, 1, 2\}$, we can fit for each $1 \le j < \ell$ a tube $T_j$ of the form \eqref{eq:def-tube} (with $L= {\rm rad}({\gamma}(j))^{\xi}$) connecting the two faces of ${\gamma}'(j)$ and ${\gamma}'(j+1)$ facing one another. For definiteness, we pick $T_j$ aligned around the axis emanating from the center of the smaller of the two faces. With a view towards applying Lemma~\ref{lem:bridging_super_specialized}, the length of $T_j$ is chosen so that the opposite faces exactly contain $\partial_{\rm L} T_j$ and $\partial_{\rm R} T_j$ as subsets.
It follows that the length (corresponding to $N+2L$ in \eqref{eq:def-tube}) of $T_j$ is at most $8\,{\rm rad}({\gamma}(j))^{\xi}$ provided $s \ge C(\xi)$. Moreover, it can be ensured that
\begin{equation}\label{eq:bridge-fifinal-2}
\text{$B(T_j, 8\,{\rm rad}({\gamma}(j))^{\xi})$, $1 \le j < \ell$, are disjoint and do not 
 	intersect $B( B_{{\rm L}} \cup B_{{\rm R}}, \textstyle \frac1{100}(\frac{s}{64})^{\xi/2} )$.}
\end{equation}
(for the latter one uses that $\{B_{{\rm L}}, B_{{\rm R}}\} $ each neighbor a box in $\mathbb{B}_1$).
Applying Lemma~\ref{lem:bridging_super_specialized} to each $T_j$ yields a bridge $\mathbb B^j = \bigcup_{1 \le i \le J_j} \mathbb 
 	B_{i}^j$ associated to 
\begin{equation}\label{eq:bridge-T-j} \textstyle (\partial_{\rm L} T_{j}, \partial_{\rm R} T_{j}, (\frac s{64})^{\xi}, \frac1{100}(\frac{s}{64})^{\xi/2}, m, \xi, T_{j})
\end{equation} for 
 	each $1 \le j < \ell$, provided $ s \geq C(\xi)$ and with $m=m(\xi)$ as supplied by Lemma~\ref{lem:bridging_super_specialized}. For each $1 \le i < \max_{j} J_j \stackrel{\text{def.}}{=} J - 1$, we then set
\begin{equation}\label{eq:bridge-fifinal-3}
\mathbb B_{i + 1} = \bigcup_{j:\, J_j > i} \mathbb B_{i}^j, \quad \mathbb B_{J} = \{B_{{\rm L}}, B_{{\rm R}}\} \cup \,  \bigcup_{j} \mathbb B_{J_j}^j.
\end{equation}
Boxes in $\mathbb B_{i+1}$ inherit the marked points on their faces from the $\mathbb B_{i}^j$'s. The centers of $\partial_{\rm L} T_{j}, \partial_{\rm R} T_{j}$, $1 \leq j \leq \ell$, along with two arbitrary points at the boundary of $\gamma'(1)$ and $\gamma'(\ell)$ on the face shared with $B_{{\rm L}}$ and $B_{{\rm R}}$, respectively, define the marked points for all boxes comprised in $\mathbb{B}_1$; indeed, cf.~\eqref{eq:bridge-fifinal-1}, exactly two such points lie on the boundary of each box $\gamma'(j)$, $1 \le j \le \ell$.

We now verify that $\mathbb B = \bigcup_{1 \le j \le J} \mathbb B_j$ defined by \eqref{eq:bridge-fifinal-1} and \eqref{eq:bridge-fifinal-3} satisfies the properties \eqref{B1}-\eqref{B4} for $(\mathcal C, 
\mathcal D, s, s', m, \xi, T)$ with $s'= \frac{s^{1/4} }{100}$. We proceed in reverse order and start with \eqref{B4}.
By construction, $|\mathbb{B}| \leq |\mathbb{B}^1|+ \sum_{1\leq j < \ell} |\mathbb{B}^j| +2$ and the desired bound on $|\mathbb B|$ follows by combining item \textit{iv)} in Lemma~\ref{lem:coarse_exist} (recall that $\ell=\ell(\gamma)$), with the fact that each $\mathbb{B}^j$ satisfies \eqref{B4} with aspect ratio $\frac{N}L \leq C$ and $L =  {\rm rad}({\gamma}(j))^{\xi}$. From this \eqref{B4} for $\mathbb{B}$ readily follows upon possibly increasing the value of $m=m(\xi)$, henceforth fixed. The required bound on $J$, defined above \eqref{eq:bridge-fifinal-3}, is inherited from the uniform bound on the $J_j$'s.

We turn to \eqref{B3}. First note that $J \geq 2$ so the boxes in $\mathbb{B}_1$ never form part of $\mathbb{H}$.
 The required box sizes in \eqref{B3} thus follow on the one hand from item \textit{ii)} in Lemma~\ref{lem:coarse_exist}, which takes care of boxes in $\mathbb{B}_1$. For boxes in $\mathbb{B}^j$ the required size follows from the corresponding property \eqref{B3} in force for $\mathbb{B}^j$ by the choice of parameters in \eqref{eq:bridge-T-j} and the obvious monotonicity of condition \eqref{B3} in $s$ and $s'$ (note in particular that $s'= s^{1/4}/100 \le (s / 64)^{\xi/2}/100$ holds for $ s\geq C(\xi)$ since $\xi \in (1/2, 1)$, whence boxes in $\mathbb{H}$ are indeed large enough). Finally items \textit{i)} and \textit{iii)} in Lemma~\ref{lem:coarse_exist} imply together that $B_{{\rm L}}$ and $B_{{\rm R}}$ have radius at least $s'$.

The connectivity requirement \eqref{B2} follows immediately from item \textit{i)} in Lemma~\ref{lem:coarse_exist}, the fact that $\gamma$ is a coarse path, and the observation that the gaps created when passing from $\gamma$ to $\gamma'$ are bridged with the help of the tubes $T_j$, $1\leq j < \ell$ introduced above \eqref{eq:bridge-fifinal-2}. From this, the above definition of the marked points and the connectivity property \eqref{B2} satisfied by each individual bridge $\mathbb B^j$, the property  \eqref{B2} for $\mathbb{B}$ readily follows.

 We finally come to \eqref{B1}.  For boxes in $\mathbb{B}_1$, the required separation property, which only concerns other boxes in $\mathbb{B}_1$ along with $\overline{\mathcal C}$
and $\overline{\mathcal D}$, follows directly from \eqref{eq:bridge_gamma'-diest}. For a box $B \in \mathbb B_i$, $1<i < J$ with $ B \in \mathbb{B}^j_i$, the required separation of $\widetilde{B}$ from other boxes $B' \in \mathbb B^j$ follows from the property \eqref{B1} for the bridge $\mathbb B^j$. If on the other hand $B' \notin \mathbb B^j$, then either i) $B' \in \mathbb{B}_1$, in which case $d(B,B') \geq d (B, \partial_{\rm L} T_{j} \cup \partial_{\rm R} T_{j})$ and separation follows again from \eqref{B1} for $\mathbb{B}^j$; or ii) $B' \in \mathbb B^{j'}$ for some $j' \neq j$ in which case one uses the first property in \eqref{eq:bridge-fifinal-2},
recalling from Lemma~\ref{lem:bridging_super_specialized} that all boxes in $\mathbb B^{j'} $ belong to $ T_{j'}$. This is more than enough to conclude since $\text{rad}(B) \leq {\rm rad}({\gamma}(j))^{\xi}$. For the separation requirement on the holes, one proceeds similarly, using the first item in \eqref{eq:bridge-fifinal-2} when the two holes belong to different bridges $\mathbb B^{j'}$, applying \eqref{B1} when the two boxes belong to the same tube $T^j$ (the same observation with regards to monotonicity in $s'$ as in \eqref{B3} applies). Finally if either of the boxes (but not both) in $\mathbb{H}$ belong to $\{B_{{\rm L}}, B_{{\rm R}}\} $ one uses the second item in \eqref{eq:bridge-fifinal-2}. Lastly to ensure that the $s'$-neighborhoods of $B_{{\rm L}}$ and $B_{{\rm R}}$ are disjoint one uses that $\ell \geq 2$, see item \textit{iv)} in Lemma~\ref{lem:coarse_exist}, which together with item \textit{i)} therein yields that $d(B_{{\rm L}}, B_{{\rm R}}) \geq cs $. 
\end{proof}
 
 We now return to the
 \begin{proof}[Proof of Lemma~\ref{lem:bridging_super_specialized}]
 	The proof is essentially one-dimensional. More precisely, it 
 	suffices to produce a partition of $T^0 \stackrel{\text{def.}}{=} [-L, N + L] \times \{0\}^{d-1}$ into 
 	contiguous segments, each having length at most $2L$, grouped into disjoint families $\mathbb I_1, \ldots, \mathbb I_J$ that satisfy the (continuous analogues of) properties \eqref{B1}-\eqref{B4} of a bridge associated to $(\mathcal C, \mathcal D, s, s', m, \xi, T^0)$ where 
 	$\mathcal C = \{-L\} \times \{0\}^{d-1}$ and $\mathcal D = \{N+L\} \times \{0\}^{d-1}$, and with the ambient space $\mathbb{R}^d$ replaced by $\mathbb{R} \times \{ 0\}^{d-1}$ (note that all of \eqref{B1}-\eqref{B4} are naturally defined under this one-dimensional projection). For then, one simply defines $\mathbb{B}$ as the union of the $d$-dimensional
  boxes having the same centers and side-lengths as the line segments constituting $\mathbb{I}= \bigcup_{1\leq j \leq J} \mathbb{I}_j$, identifying those in $\mathbb{H}$ with boxes associated to line segments in $\mathbb{I}_J$, and define the marked points being the endpoints of the segment $I \in \mathbb{I}$ corresponding to a box $B \in \mathbb{B}$. One readily sees that $\mathbb{B}$ thereby inherits properties \eqref{B1}-\eqref{B4}.
  
  \medskip
We now produce the desired partition of $T^0$ in a recursive fashion. The driving force is the following very simple
 	
 \begin{claim}[$\xi \in (\frac12, 1)$] \label{Claim-1d-bridge} For all $R \ge 4 $ and $ C(\xi) \leq  r \leq \frac R4$, there exists an (ordered) sequence of disjoint sub-intervals $(I_1, I_2, \ldots, I_{k})$ of $[0, R]$ such that, with $|I|= b-a$ when $I=[a,b]$, $a<b$,
 	\begin{align}
 		&\text{$\textstyle k \le 1+ \frac Rr$,}\label{eq:bridge-1d1}\\
 		&\text{$d(I_1, \{0\}), d(I_k, \{R\}) \in [2 r^{\xi}, 7r^{\xi}]$, $d(I_i, I_{i+1}) = 2 r^{\xi}$, $1\leq i < k$}\label{eq:bridge-1d2}\\
 		&\text{$|I_i| \in [r^{\xi}, r]$, for all $1\leq i \leq k$.}\label{eq:bridge-1d3}
 	\end{align}
 \end{claim}	
 
 \begin{proof}[Proof of Claim~\ref{Claim-1d-bridge}]
Let $\tilde{k} = \lfloor \tfrac{R}{r + 2 r^{\xi}} \rfloor \ge 1$ so that 
 	$\tilde k(r + 2 r^{\xi}) \le R < (\tilde k+1)(r + 2 r^{\xi})$ and let 
	$$J_i = [(i-1)r + 2i r^{\xi}, i(r + 
 	2 r^{\xi})], \text{ for $i  \geq1$}.$$
	 If $R - \tilde k(r + 2 r^{\xi}) \le 5r^{\xi}$, we set $k = \tilde{k}$, $I_i = J_i$ for $1 \leq i < k$ and $I_k=I_{\tilde{k}}= [(\tilde k -1)r + 2\tilde k r^{\xi}, \tilde k(r + 
 	2 r^{\xi}) -2r^{\xi}]$. Otherwise $R - \tilde k(r + 
 	2 r^{\xi}) \in (5r^{\xi}, r + 2r^{\xi})$ and we set $k = \tilde{k} + 1$ with $I_i =J_i$ for $1\le i \leq k$ and  $I_{\tilde k + 1} = [\tilde kr + 2(\tilde k+1) r^{\xi}, R - 
 	2r^{\xi}]$. In either case, \eqref{eq:bridge-1d1} plainly holds and so does the second item in  \eqref{eq:bridge-1d2}. Moreover, $d(I_{1}, \{0\})=2r^{\xi}$ in either case, and $ 2r^{\xi} \leq d(I_{k}, \{R\})  \leq 5 r^{\xi}$ in the former whereas $d(I_{k+1}, \{R\}) = 2 r^{\xi}$  in the latter case, yielding \eqref{eq:bridge-1d2}. Moreover, $|I_i|=r$ except possibly when $i=k$, and $|I_k| \in [(r-2r^{\xi}) \wedge (5r^{\xi} - 4r^{\xi}), r] = [r^{\xi}, r]$, whence \eqref{eq:bridge-1d3}.  \end{proof}

We now conclude the proof of the lemma. We start by setting $\mathbb{I}_1 = \{I_{1, 1}, \ldots, I_{1, k}\}$ obtained by applying Claim~\ref{Claim-1d-bridge} to the interval $[-L, N + L]$ for $r = 
 	L/2$ (after a suitable translation) when $L \geq C(\xi)$.  Now suppose that after a round $J' \ge 1$, we have obtained the collections $\mathbb{I}_1, \ldots, \mathbb{I}_{J'}$ of (closed) sub-intervals of $[-L, N + L]$ satisfying
\begin{equation}\label{eq:separation_I}
\begin{split}
d\Big(I, \bigcup_{j \le J'} \bigcup_{I' \in \mathbb{I}_j} I'  \setminus I\Big),\, d (I, \{-L, N + L\}) \ge 2|I|^{\xi} \vee 
\frac{s^{\xi}}2, \quad |I| \ge\frac{ \sqrt{s} }{100},
\end{split}
\end{equation}
for all $I \in \mathbb{I}_{j'}$, 
$j'  \le J'$. Also suppose that $|\mathbb{I}_{j+1}| \le 8|\mathbb{I}_{j}|$ for all $j' < J'$. Note that 
 	condition \eqref{eq:separation_I} is satisfied by $\mathbb{I}_1$ due to the \eqref{eq:bridge-1d2}, \eqref{eq:bridge-1d3} and since
 	$L \ge s$ and $\xi \in (\frac12,1)$ by hypothesis.
 	Now let $\mathbb{I}_{J'}'$ denote the collection consisting of the closures of the components of $[-L, N+L] \setminus \bigcup_{j \le J'} \bigcup_{I \in \mathbb{I}_j} I,$ 
 	which are closed sub-intervals of $[-L, N+L]$. If no interval in $\mathbb{I}_{J'}'$ has length larger than $s$, 
 	we stop, assign $J = J' + 1$ and $\mathbb{I}_J = \mathbb{I}_{J'}'$.  Otherwise we let $\mathbb{I}_{J' + 1}$ consist of 
 	the sub-intervals obtained by applying Claim~\ref{Claim-1d-bridge} individually to each interval $I$ in $\mathbb{I}_{J'}'$ with $R=|I| \ge s (\geq C(\xi))$ and $r = \frac{ |I|}4$. By similar arguments as used for $\mathbb{I}_1$, it follows that $\mathbb{I}_{J' + 1}$ satisfies 
 	\eqref{eq:separation_I}. On the other hand, by \eqref{eq:bridge-1d1} we deduce tha $|\mathbb{I}_{J' + 1}| \le 
 	8|\mathbb{I}_{J'}|$. 
	
	Overall, this defines $\mathbb I_1, \ldots, \mathbb I_J$. Properties~\eqref{B1}-\eqref{B3} for $\mathbb{I}= \bigcup_{1\leq j \leq J} \mathbb{I}_j$ are direct consequences of the above construction and~\eqref{eq:separation_I}. As to \eqref{B4}, it is a consequence of the definition of $J$ as well as 
 	second and the third item in the Claim that $J \le (\log \xi^{-1})^{-1} \log \log 10 L$ for all $C'$ large enough depending only on $\xi$. Also, from our induction hypothesis we have $|\mathbb{I}_{j+1}| \le 8 |\mathbb{I}_{j}|$ for all $j < J-1$ and $|\mathbb{I}_J| \le 2|\mathbb{I}_1| \sum_{j < J} 8^j$. These observations together imply 
 	\eqref{B4}.
 \end{proof}
 
 \begin{remark}
 \label{R:bridge-other-paper} The bridge exhibited in Proposition~\ref{prop:general_bridge} is also employed in our companion article \cite{RI-III} within a different framework, but to a similar effect. Namely, it enables us there to reconstruct a connection within a `coarse pivotal' configuration at a not too high energetic cost. In the specific context of \cite[Section 8.2]{RI-III} where it is used,  the bridge is actually constructed within a tube $T$ having `thin width' and to a more elaborate model $\mathcal{V}_T$ rather than $\mathcal{V}^u$, for which we first need to prove a conditional decoupling property akin to Proposition~\ref{prop:cond_decoup} above; see \cite[Lemma 7.1]{RI-III}. Once this is done, the actual surgery performed in the course of proving \cite[Lemma 8.8]{RI-III} unfolds in a similar way as in Section~\ref{Sec:surgery} below.
 \end{remark}

 \section{Strong percolation from gluing property}\label{sec:supercritical}

In this short section, which can be read independently of the rest of this article we reduce the proof Theorem~\ref{T:MAIN}, to that of proving a single `gluing property,' stated in Lemma~\ref{lem:uniqreduct}, which roughly asserts that any (partial) sprinkling will merge a significant proportion of ambient large clusters. Most of the difficulty hides in the proof of this lemma, which is postponed to later sections. In particular, this is where the delicate surgery argument to perform the gluing is carried out, which draws on the tools derived in the previous two sections.

The present section, which is independent of Sections~\ref{sec:sfe} and~\ref{sec:bridging}, contains the derivation of Theorem~\ref{T:MAIN} from Lemma~\ref{lem:uniqreduct}, which is divided into two parts. First, we show using Lemma~\ref{lem:uniqreduct} as an input that the disconnection assumption \eqref{eq:disc-cond} implies a weak form of \eqref{eq:barh1}, which yields the existence and uniqueness (up to sprinkling) of 
large local clusters with probability tending to one at large scales. This is the content of Proposition~\ref{prop:uniq}. This result is a loose analogue of \cite[Prop.~4.1]{DCGRS20}, and it draws inspiration from a result of \cite{MR3634283}. While the proofs of these two propositions are superficially similar (in fact the layout is borrowed from~\cite{MR3634283}), the crucial gluing device used as an input is a different matter entirely; this is because, as discussed at length in \S\ref{subsec-pf-outline} (see above \eqref{eq:no-FE}), the arguments of \cite{DCGRS20} do not apply at all. We note that the route we take provides a more robust way of 
proving `gluing results' of this type (such as \cite[Prop.~4.1]{DCGRS20} in particular).

In a second step, the result of Proposition~\ref{prop:uniq}, which is not quantitative, is then bootstrapped to a stretched-exponential bound in order to produce the desired controls in \eqref{eq:barh1}, thus concluding the proof of Theorem~\ref{T:MAIN}. Since this second step follows a by now relatively standard renormalization procedure, we present it first immediately after stating Proposition~\ref{prop:uniq}, which plays the role of a `seed' estimate for the renormalization scheme. The remainder of this section is then devoted to the proof of Proposition~\ref{prop:uniq}. 

\bigskip

We will work with the following `unique crossing' event. For any $u > v>0$ and $M \geq 1$, let
\begin{equation}
\label{eq:defcross}
\begin{split}
\text{\textbf{UC}}(M,u,v)\stackrel{\textnormal{def.}}{=} &\{\lr{}{\mathcal V^{u}}{B_M}{\partial 
B_{6M}}\}\,\\& \qquad \cap \Big\{\begin{array}{c}\text{all clusters in $\mathcal V^{u}\cap B_{4M}$ crossing 
$B_{4M}\setminus B_{2M}$} \\ \text{are connected to each other in $\mathcal V^{v}\cap B_{4M}$} 
\end{array}\Big\}.
\end{split}
\end{equation}
Recall the definition of $M(\cdot)$ from \eqref{eq:def_M}, which depends on the parameter $\gamma_M$. 
\begin{prop}
\label{prop:uniq}
For all $\gamma_M \geq \Cr{c:gamma_M}$, if $v>0$ is such that \eqref{eq:disc-cond} holds for some $\alpha > 0$, then for all $\delta >0$,
\begin{equation}\label{eq:propuniq2}
\limsup_{N\to \infty} \inf_{  u \in [2\delta, v-2\delta] } \P\left[\textnormal{\textbf{UC}}(M(N),u, u - \delta)\right] \geq 1 -C \alpha.
 	\end{equation}
 \end{prop}
 
 We first explain how to deduce Theorem~\ref{T:MAIN} from Proposition~\ref{prop:uniq}.
 
 \begin{proof}[Proof of Theorem~\ref{T:MAIN} (assuming Proposition~\ref{prop:uniq})] 
 We first argue that $ii)$ implies $i)$. Let $u>0$ and $v \in (0,u)$ be given. Choose any $v' \in (v,u)$, for concreteness say $v' =\frac{v+u}{2}$.
 By requiring $\text{Exist}(r,v')$ and $\text{Unique}(r,v',v)$ to occur simultaneously for all scales $r=r_0 2^k$ for $r_0 \geq 1$ and $k \geq0$, one readily infers from \eqref{eq:barh1}, using a union bound and a straightforward gluing argument involving \eqref{e:strong-perco} that the disconnection event in \eqref{eq:disc-cond} has stretched exponential decay in $r$, i.e.
$$ \lim_r e^{r^c}\displaystyle{\P[\nlr{}{ \mathcal{V}^v}{B_{r}}{B_{M(r)}}]} = 0.$$
 From this \eqref{eq:disc-cond} immediately follows for any $\gamma_M >1$ in view of \eqref{eq:def_M}.
 
 We now show that $i)$ implies $ii)$, which is the heart of the matter, and brings into play Proposition~\ref{prop:uniq}. To this end we assume from here onwards that $\gamma_M \geq \Cr{c:gamma_M}$. Let $u>0$ and $v,v'$ with $0< v< v'<u$ be given. 
 As we now explain, it is enough to prove a stretched exponential upper bound of the form
 \begin{equation}\label{e:supercrit-renorm10}
  \P\left[\text{\textbf{UC}}(r,v', v)^c\right] \leq C{ e}^{-r^{\Cr{c:decay}}},
 \end{equation}
valid for all $r \geq 1$, with constant $C$ possibly depending on $v$ and $v'$. Indeed, in view of \eqref{eq:defcross} and \eqref{e:strong-perco}, the event $\text{\textbf{UC}}(r/2,v', v)$ implies $\text{Exist}(r,v')$, which inherits the bound \eqref{e:supercrit-renorm10}, as required by \eqref{eq:barh1}. To obtain the analogue estimate for $\P[\textnormal{Unique}(r,v',v)]$, one observes that the latter event is implied by the intersection of all translates of $\text{\textbf{UC}}(\frac{r}{100},v', v)$ mapping the origin to a point in $B_r$; for, any connected component of $\mathcal V^{v'} \cap B_r$ having diameter at least $\frac r5$ will cross the annulus with inner radius $\frac r{50}$ and outer radius $\frac{r}{25}$ around at least one such box, and the joint occurrence of all unique crossing events implies that these components are all connected in $\mathcal{V}^v \cap B_{2r}$, as required for $\textnormal{Unique}(r,v',v)$ to occur. The required upper bound for $\P[\textnormal{Unique}(r,v',v)^c]$ thus follows from \eqref{e:supercrit-renorm10} and a union bound.

  The proof of \eqref{e:supercrit-renorm10} under the hypothesis that \eqref{eq:propuniq2} holds for all $v< u$, as implied by Proposition~\ref{prop:uniq} and item $i)$ in Theorem~\ref{T:MAIN} (our standing assumption), follows a by now relatively standard renormalization scheme, so we only sketch the argument. Let $\varepsilon \stackrel{{\rm def.}}{=} \frac{v' - v}{3} \wedge \frac{u-v'}{3}$. We will apply the results 
of \cite{drewitz2018geometry} to the graph $\Z^d$ (with unit weights); see also the proof of 
\cite[Lemma~5.16]{gosrodsev2021radius} for a similar argument in the context of the Gaussian free field. Let $L_0 = M(N_0)$ for 
a large positive integer parameter $N_0$ to be determined. For any $x \in 
\Z^d$, let $\text{\textbf{UC}}_x(L_0, u , v)$ denote the event in \eqref{eq:defcross} `shifted by $x$' 
(now the center of all the boxes involved).

Choosing $\lambda$ sufficiently large (to be specific, one can pick $\lambda = 20 \geq 20c_{18}C_{10}$ in the context of~\cite[(8.3)]{drewitz2018geometry} where $C_{10}=1$ and $c_{18} < 1$ is determined by the isoperimetric constant on $\Z^d$), one sets for $x \in \Z^d$, with $\bar{\lambda}=1.1\lambda$, $\ell_0=3^d \vee 12 \bar{\lambda}$ and 
$L_n=\ell_0^n L_0$, 
\begin{equation}
	\label{eq:RG4}
	\begin{split}
		&\widehat{\text{\textbf{UC}}}_{x,0} = \bigcap_{y\in {L_0\Z^d \cap B_{\lambda L_0}(x)}}\text{\textbf{UC}}_{y}(L_0, v' - \varepsilon , v + \varepsilon),\\
		&\widehat{\text{\textbf{UC}}}_{x,n} =  \bigcap_{y,z\in (L_{n-1}\Z^d\cap B_{\bar{\lambda} L_n}(x)): \, d(y,z)\geq L_n} \big(\widehat{\text{\textbf{UC}}}_{y,n-1}^k \cup \widehat{\text{\textbf{UC}}}_{z,n-1}^k\big),  \quad \text{for $n \geq 1$}.
	\end{split}
\end{equation}
By choice of $\varepsilon$, one obtains from Proposition~\ref{prop:uniq} applied with $\delta=\frac{\varepsilon}{2}$ (Proposition~\ref{prop:uniq} is in force under assumption $i)$ in Theorem~\ref{T:MAIN} whenever $v'-v \leq c$), followed by a union bound over $L_0\Z^d \cap B_{\lambda 
L_0}(x)$ that $\limsup_{L_0}\P[\widehat{\text{\textbf{UC}}}_{x,0}]  \geq 1 - C \alpha$. Note that the condition $v'-v \leq c$ (which ensures that the requirement `$u\geq 2\delta$' in \eqref{eq:propuniq2} holds
) is harmless; indeed if $v$ is decreased in \eqref{e:supercrit-renorm10} then the bound continues to hold by monotonicity. In particular, by choosing first $\alpha=\alpha(d)$ small enough and then $L_0=L_0(v,v', \gamma_M)$ sufficiently large, one can ensure 
that the conditions (7.5) and (7.6) in \cite{drewitz2018geometry} are satisfied, whence Proposition 7.1 
therein applies and yields that
\begin{equation}
	\label{eq:RG5}
	\P[(\widehat{\text{\textbf{UC}}}_{0,n}^{\varepsilon})^c] \leq 2^{-2^n}, \text{ for all $n \geq 0$},
\end{equation} 
where the events $\widehat{\text{\textbf{UC}}}_{x,n}^{\varepsilon}$ refers to those defined in \eqref{eq:RG4}, but with 
parameters 
$(v', v)$ in place of $(v' - \varepsilon , v + \varepsilon)$. The presence of two opposite directions of monotonicity requires a small extension over the decoupling inequality stated in \cite[Theorem 2.4]{drewitz2018geometry}; alternatively one can resort to the stronger  result \eqref{eq:mul_excursion_sandwich} (extending Proposition~\ref{prop:cond_decoup}) applied with $k=2$. It is worth emphasizing however that the decoupling only needs to operate between boxes of a given size, $L_n$, separated by a large multiple $c \ell_0 L_n$ of their diameter, as originally in \cite{MR2891880}.

In order to conclude, one applies 
\cite[Lemma~8.6]{drewitz2018geometry}, 
which implies that whenever $ \widehat{\text{\textbf{UC}}}_{0,n}^{\varepsilon}$ occurs, any two 
connected sets in $B_{\lambda L_n}$, each of diameter at least $(\lambda/20) L_n$, are connected by a 
path $\gamma \subset B_{2\lambda L_n}$ such that $\widehat{\text{\textbf{UC}}}_{x,0}^{\varepsilon}$ occurs for all 
$x \in \gamma$. This is readily 
seen to imply $\text{\textbf{UC}}(r,v', v)$ for $r =r_n = \lfloor (\lambda/10) L_n \rfloor$, hence \eqref{eq:RG5} yields \eqref{e:supercrit-renorm10} for such $r$. The remaining values of $r$ are taken care of using a union bound involving translates of the event $\text{\textbf{UC}}(r_n,v', v)$ at scale $r_n \geq cr$. 
 \end{proof}

 We now turn to Proposition~\ref{prop:uniq}. Below we give the proof in full assuming the validity of Lemma~\ref{lem:uniqreduct}, our gluing device.
 This initial reduction step is similar to the start of the proof of \cite[Proposition~4.1]{DCGRS20}, to which we will frequently refer throughout the rest of this section. The proof of Lemma~\ref{lem:uniqreduct} will occupy the remaining sections.

 \begin{proof}[Proof of Proposition~\ref{prop:uniq}] 
 Let $v, \delta >0$ and consider any $u $ with $2\delta \le u \leq v - 2\delta$. In the sequel we often abbreviate $M=M(N)$, for arbitrary $N \geq1$.
Define the collection
\begin{equation}
\label{eq:definiq1}
\mathcal{C} = \left\{\text{$C\subset B_{4M}$: $C$ is a cluster in $\mathcal V^u \cap B_{4M}$ intersecting $\partial B_{4M}$}\right\},
\end{equation}
and let $\eta\in\{ 0,1\}^{\Z^d}$ be any percolation configuration such that $\eta(x)=1$ whenever $x\in \mathcal V^u$ for some $x \in \Z^d$, i.e.~all vertices in $\mathcal V^u$ are open in $\eta$. We abbreviate this by $\{\eta=1\} \supset \mathcal V^u$ in the sequel. For such $\eta$, we write
\begin{equation}
	\label{eq:equiv}
	C  \sim_{\eta} C' \text{ if } \lr{}{\eta}{C}{C'}, \text{ for } C,C' \in \mathcal{C}. 
\end{equation}
For any sub-collection $\tilde{\mathcal{C}} \subset \mathcal{C}$, $\sim_{\eta}$ defines an equivalence relation on $\tilde{\mathcal{C}}$ and hence, in particular, the elements of $\tilde{\mathcal{C}} / \sim_{\eta} $ form a partition of $\tilde{\mathcal{C}}$. 
For every $0\leq i\leq \floor{2\sqrt{M}}$, let $V_i=B_{4M-i\sqrt{M}}$ and 
consider the sets
\begin{equation}
	\label{eq:Uij}
	\mathcal{U}_{i}(\eta)\stackrel{\textnormal{def.}}{=} \big\{C \in \mathcal{C} : C\cap V_{2i} \neq \emptyset  \, \big\} \big/ \sim_{\eta}
\end{equation}
for $0\leq i\leq \floor{ \sqrt{M}}$ and $\{\eta=1\} \supset \mathcal V^u$. Denoting $|\mathcal{U}_{i}(\eta)|={U}_{i}(\eta)$, notice that ${U}_{i}(\eta)$ is decreasing in both $\eta$ and 
$i$ in view of \eqref{eq:definiq1} and \eqref{eq:Uij}. This fact will be used repeatedly in the sequel.
Now, consider the percolation configurations $\eta_0 \leq \eta_1 \leq \dots \leq \eta_{\floor{\sqrt{M}}}$ under $\P$, whereby
\begin{equation}
\label{omega_i}
\eta_i \stackrel{\textnormal{def.}}{=} \begin{cases}
		1_{\{x \in \mathcal V^u\}}, & x \in V_{2i},\\
		1_{\{x \in \mathcal V^{u - \delta}\}}, & x \notin V_{2i}.
	\end{cases}
\end{equation}
In plain words, 
$\eta_i$ corresponds to a partial sprinkling outside of $V_{2i}$. For $0\leq i\leq \floor{\sqrt{M}}$, set
\begin{equation}
\label{eq:Ui}
\mathcal{U}_{i}\stackrel{\textnormal{def.}}{=} \mathcal{U}_{i}(\eta_{i}) \text{ and } U_i 
\stackrel{\textnormal{def.}}{=}|\mathcal{U}_i|. 
\end{equation}
Clearly $\eta_i$ is increasing in $i$ so by the previous observation, $U_i$ is decreasing in $i$. Let 	
 \begin{equation}	
 \begin{split}
A&\stackrel{\textnormal{def.}}{=}\{\lr{}{\mathcal V^u}{B_{10N}(x)}{\partial B_{6M}} ~\text{ for all } x \in N\Z^d \text{ 
s.t. } B_{10N}(x) \subset B_{4M}\} \label{eq:supercritA}. 
\end{split} 		
\end{equation}
In view of \eqref{eq:defcross}, \eqref{eq:supercritA} and 
\eqref{eq:Uij}--\eqref{eq:Ui}, for all $N \geq C$ the event $\text{\textbf{UC}}\big(M,u,u-\delta\big)$ occurs as soon as $A$ does 
and ${U}_{\lfloor\sqrt{M}\rfloor} =1$. Consequently, for all $N \geq C$, with $M=M(N)$, we obtain that
\begin{equation}\label{eq:uniqueproof}
\begin{split}	
\P[\text{\textbf{UC}}(M, u, u - \delta)^c] &\leq \P[A^c] + \P\big[A \cap\{ {U}_{\lfloor\sqrt{M}\rfloor}  >1\}\big] \\ 
&\leq C (M/N)^d\,\P\big[\nlr{}{\mathcal V^u}{B_{10N}}{\partial B_{M(10N)}}\big] + \P\big[A\cap\{{U}_{\lfloor\sqrt{M}\rfloor} >1\}],
\end{split}
\end{equation} 	
where, in deducing the second line, we applied a union bound over all $x \in B_{4M}$, used that for all such $x$, one has the inclusions $B_{6M} = B_{6M(N)} \subset B_{10M(N)}(x) \subset B_{M(10N)}$ since $M(10N) \ge 10M(N)$ by \eqref{eq:def_M} (recall that $\gamma_M>1$), and  translation invariance.

Since $v>0$ is such that \eqref{eq:disc-cond} holds and by monotonocity of the disconnection event with respect to $u$ in \eqref{eq:uniqueproof}, there is a subsequence $N_k \to \infty$ depending on $v$ along which the first term on the right-hand side of 
\eqref{eq:uniqueproof} is bounded by $C\alpha$ uniformly in $u \leq v$. 
Thus, \eqref{eq:propuniq2} follows once we show that the second term 
converges to 0 uniformly in $u \in  [2\delta,  v-2\delta]$ as $N \to \infty$. 

We start the proof of the latter by borrowing the set-up of the proof of \cite[Lemma~4.2]{DCGRS20}, which we now adapt to the present framework. Matters will however soon diverge. Recall that $\mathcal{U}_{i}(\eta)$ in \eqref{eq:Uij} is a set comprising equivalence classes of clusters in $\mathcal{C}$, which we view as elements of $2^{\mathcal{C}}$. The following notation will be convenient. We write $\mathscr C$ for a generic element of $2^{\mathcal{C}}$ (which could for instance belong to $\mathcal{U}_{i}(\eta)$) and abbreviate $\bigcup_{C \in 
\mathscr{C}} C\, (\subset \Z^d)$ as $\text{supp}(\mathscr{C})$.  For a percolation configuration $\eta \in \{0,1\}^{
\Z^d}$ such that $\{\eta=1\}\supset \mathcal V^u$, $0\leq i < \floor{ \sqrt{M}}$ and $k\in \{0,\frac12\}$, set
\begin{equation}
\label{eq:U_ijbis}
\mathcal{U}_{i+k,i+1}(\eta) \stackrel{\textnormal{def.}}{=} \{ \mathscr{C} \in \mathcal{U}_i(\eta): 
\text{supp}(\mathscr{C}) \cap V_{2(i+k)} \neq \emptyset,\, \text{supp}(\mathscr{C}) \cap V_{2(i+1)} 
=\emptyset\},
\end{equation}
and $U_{i,i+1}(\eta) \stackrel{\textnormal{def.}}{=} |\mathcal{U}_{i,i+1}(\eta)|$. Clearly, $U_{i,i+1}(\eta)$ is 
decreasing in $\eta$ and $U_{i} (\eta) = U_{i+1} (\eta) + {U}_{i,i+1}(\eta)$. We explain the benefit of introducing $ {U}_{i,i+1}(\eta)$ at the end of this section.
Recall from \eqref{eq:Ui} that $U_i=|\mathcal{U}_i(\eta_i)|$.
The key input is the following result. 
\begin{lemma}[Gluing]\label{lem:uniqreduct} 
For all $\gamma_M  \geq \Cr{c:gamma_M}$, there exist 
$C = C(\gamma_M)$ 
and $c' = c'(\gamma_M)$ such that the following holds. 
Letting
\begin{equation}\label{e:def-E}
E= A \cap \big\{ U_{j+1}> 1\vee ( U_{j}/4 + {U}_{j, j+1}(\eta_j))\big\},
\end{equation}
one has, for all $v >0$ as in the statement of Proposition~\ref{prop:uniq}, all $\delta >0$, $u \in [2\delta,  v-2\delta]$, $N \geq 1$, $0\leq j< \floor{\sqrt{M}} $ (with $M=M(N)$)
, for some $c=c(\delta,v)>0$,  
\begin{equation}
\label{eq:uniqreduct_new1}
\begin{split}
&\P[ E ] \leq C\exp \big\{- c \, e^{c (\log M)^{c'}}\big\}.
\end{split}
\end{equation}
\end{lemma}

We admit Lemma~\ref{lem:uniqreduct} for now. Its proof will span Sections~\ref{Sec:surgery}-\ref{Sec:2lemmas}. With Lemma~\ref{lem:uniqreduct} at hand, we now conclude the proof of Proposition~\ref{prop:uniq}, which follows in a straightforward manner. 
 First, we claim that for $\gamma_M \geq \Cr{c:gamma_M}$, $v, \delta, u$ as above \eqref{eq:uniqreduct_new1}, $N \geq 1$ and $0\leq i\leq \floor{\sqrt{M}} - 4$, one has
\begin{equation}
\label{eq:uniqreduct}
\P[ A\cap  \{ U_{i + 4} > 1\vee \textstyle U_i / 2 \} ] \leq C\exp\big\{- c \, e^{c (\log M)^{c'}}\big\},
\end{equation}
with identical dependence of constants on parameters as in the statement of Lemma~\ref{lem:uniqreduct}. To see this, observe that, using the formula $U_{i} (\eta) = U_{i+1} (\eta) + {U}_{i,i+1}(\eta)$ noted below \eqref{eq:U_ijbis}, choosing $\eta=\eta_i$ as in \eqref{omega_i} and iterating the last identity, it follows that
${U}_{i+4}(\eta_i) + \sum_{j: \, i \leq j < i + 4} {U}_{j,j+1}(\eta_i)  = {U}_{i }(\eta_i)=U_i,$  which implies that ${U}_{j,j+1}(\eta_i)  \leq U_i / 4$ for some $j$ with $ i \leq j < i+4$  since all quantities are non-negative. From this, \eqref{eq:uniqreduct} follows by a union bound over $j$. Indeed just notice that for fixed $i$ and all $j$ with $i \leq j < i+4$, one has  $U_{i+4} = 
U_{i+4}(\eta_{i+4}) \leq  U_{j+1}(\eta_{j+1})=U_{j+1} $ by monotonicity since $i+4 \geq j+1$, and 
that $U_i \geq U_j $ and ${U}_{j,j+1}(\eta_i) \geq {U}_{j,j+1}(\eta_j)$, for all $i \leq j$.

Next, observe that, if the event $\bigcap_{0\leq k < K} \{U_{4(k+1)} \leq 1\vee \frac{U_{4k}}{2}\}$ occurs for 
some $K \geq 1$ 
such that $4K \leq \lfloor\sqrt{M}\rfloor$, then either 
\begin{equation}
\label{eq:Ubound1}
{U}_{\lfloor\sqrt{M}\rfloor}  \leq  U_{4K} \leq 2^{-1} U_{4(K-1)} \leq \cdots \leq 2^{-K} 
U_0 = 2^{-K} |\mathcal{C}| \stackrel{\eqref{eq:definiq1}}{\leq} C  2^{-K} 
M^{d-1},
\end{equation}
or $$U_{4(k+1)} \leq 1 \:\: \text{for some }0\leq k < K.$$ In the latter case ${U}_{\lfloor\sqrt{M}\rfloor}  \leq 
U_{4(k+1)} \leq 1$ by monotonicity. Hence, letting $K = \lfloor(C' \log M / \log 4)\rfloor$ with $C'$ 
large enough in a manner depending on $d$ only so the right-hand side of \eqref{eq:Ubound1} is bounded by $1$, it follows  that the occurrence of
$A \cap \{{U}_{\lfloor\sqrt{M}\rfloor} >1\}$ implies the event $\bigcup_{0\leq k < K} A \cap  \{ U_{4(k+1)} 
\geq 1\vee U_{4k} /2 \}.$ Thus, applying a union bound over $k$ and using \eqref{eq:uniqreduct} 
yields that 
\begin{equation}
\label{eq:reduct2}
\P\big[A\cap \big\{{U}_{\lfloor\sqrt{M}\rfloor} >1\big\}\big] \leq  C K \exp\big\{- c \, e^{c (\log M)^{c'}}\}.
\end{equation}
By choice of $K$, the right-hand side of \eqref{eq:reduct2} tends to $0$ as $N \to \infty$. Feeding this into \eqref{eq:uniqueproof} and recalling the discussion above Lemma~\ref{lem:uniqreduct}, the claim \eqref{prop:uniq} immediately follows.
 \end{proof}
 
The reader may wonder with good reason why the conclusion \eqref{eq:uniqreduct_new1} of Lemma~\ref{lem:uniqreduct} concerns the more verbose event $E$ in \eqref{e:def-E} (which includes the random variable $U_{j,j+1}$) rather than just \eqref{eq:uniqreduct} for instance. The benefit of introducing $U_{j,j+1}$ is to facilitate the presence of `large interfaces', as we now explain.

For 
$0\leq j< \floor{\sqrt{M}}$, 
let 
(recall \eqref{eq:Ui}, \eqref{eq:U_ijbis} for notation)
\begin{equation}
\label{Utilde}
\widetilde{\mathcal{U}}(\eta_j) \stackrel{\textnormal{def.}}{=}
\begin{cases}
( \mathcal{U}_{j}(\eta_j) \setminus \mathcal{U}_{j,j+1}(\eta_j)) \cup\{\widetilde{\mathscr{C}}\}, & \text{ if } \mathcal{U}_{j+\frac12,j+1}(\eta_j) \ne \emptyset, \\
\mathcal{U}_{j}(\eta_j) \setminus \mathcal{U}_{j,j+1}(\eta_j), & \text{ 
otherwise, } 
\end{cases}
\end{equation}
where $\widetilde{\mathscr{C}} \stackrel{\textnormal{def.}}{=} \{ C : C\in \mathscr{C} \text{ for some } 
\mathscr{C} \in \mathcal{U}_{j, j+1}(\eta_j) \}$. 
We 
will no longer carry the argument $\eta_j$ in our subsequent discussion. By the same reasoning that was 
used to verify \cite[(4.19)--(4.20)]{DCGRS20}, the following 
properties of $\widetilde{\mathcal{U}}$ hold on the event $E$ in \eqref{e:def-E}:
 	\begin{align}
 		&\begin{array}{l}
 			\text{for $A_j\in \{V_{2j}\setminus V_{2j+1}, V_{2j+1}\setminus V_{2j+2}\}$, with $A_j$ a deterministic}\\\text{function of $\widetilde{\mathcal{U}}$, each of the sets $\text{supp}(\mathscr{C})$ for $\mathscr{C} \in \widetilde{\mathcal{U}}$ crosses $A_j$} \\
 			\text{and their union intersects all the boxes of radius 
 				$10N$ contained in $A_j$;}\\
 		\end{array} \label{Uconnect}\\
 		&\begin{array}{l}
 			\text{there exists a non-trivial partition $\widetilde{\mathcal{U}}= \widetilde{\mathcal{U}}_1 \sqcup 
 			\widetilde{\mathcal{U}}_2$ such that}\\\text{$\displaystyle 
 			\{\nlr{}{\eta_{j+1}}{{\mathcal S}_1}{{\mathcal{S}}_2}\}$ and $|\{ \mathscr{C} : \mathscr{C} \in 
 			\widetilde{\mathcal{U}}_1  \}| \leq 4$, where $ \textstyle {\mathcal{S}}_i = \bigcup_{\mathscr{C}\in 
 		\widetilde{\mathcal{U}}_i}\text{supp}(\mathscr{C})$. 	} 
 		\end{array} \label{Upartition}
 	\end{align}
These last two facts will be used at the start of the next section, in order to exhibit many different `contact areas,' which are well-separated boxes in which the gluing will be attempted.

\section{Gluing clusters}
\label{Sec:surgery}

The remainder of this article, i.e.~the present section and the next, is devoted to proving Lemma~\ref{lem:uniqreduct}. The proof will bring into play the various ingredients derived  in Sections~\ref{sec:sfe} and \ref{sec:bridging}. The structure of the argument is best explained after a straightforward initial step; see below \eqref{eq:repeated_encounter}. The current section contains the full proof of Lemma~\ref{lem:uniqreduct}, save for two results, Lemmas~\ref{L:bound_bad} and~\ref{cor:ample_connection}, which control the number of certain counts of good ($\mathcal K_g$) and bad ($ \widetilde{\mathcal K}_b$) boxes. These two lemmas are proved separately in Section~\ref{Sec:surgery}.

\begin{proof}[Proof of Lemma~\ref{lem:uniqreduct}]

We start by decomposing the event $E$ in \eqref{e:def-E} by conditioning on all possible realizations $\{\mathscr{C}\}$ of 
$\widetilde{\mathcal{U}}$ defined in \eqref{Utilde}. Applying a union bound over the partition $\{ \mathscr{C}\} = \{ \mathscr{C}\}_1 \sqcup \{ \mathscr{C}\}_2$ 
where $\{ \mathscr{C}\}_i$ corresponds to the realization of $\widetilde{\mathcal{U}}_i$ from \eqref{Upartition} on the event $\{ 
\widetilde{\mathcal{U}} =\{ \mathscr{C}\}\}$, we get, for $S_i 
= \bigcup_{\mathscr C \in \{\mathscr C\}_i} {\rm supp}(\mathscr C)$ the induced realization of $\mathcal{S}_i$, writing $\P_{\mathscr 
C}[\cdot]=\P[ \, \cdot \, | \, \widetilde{\mathcal U} = \{\mathscr C\}]$ whenever $\P[ \, \widetilde{\mathcal U} = \{\mathscr C\}]>0$,
\begin{equation}
\label{eq:Ai1}
\begin{split}
\P[E] \leq \sum_{\{\mathscr C\}} \P\big[\,\widetilde{\mathcal U} = \{\mathscr C\}\big] \times 
\sum_{\substack{\{\mathscr C\} = \{\mathscr C\}_1 \sqcup \{\mathscr C\}_2\\ |\{\mathscr C\}_1| \le 4}}\, 
\P_{\mathscr C}\Big[ \nlr{A_j}{\mathcal{V}^{u-\delta}}{S_1}{S_2} \Big];
\end{split}
\end{equation}
here the first sum is over realizations $\{\mathscr C\}$ such that $\{\widetilde {\mathcal U} = \{\mathscr C\}\} \cap E \ne \emptyset.$

We will eventually control the connection probabilities appearing on the right-hand side of \eqref{eq:Ai1} individually, but this will take a while; cf.~\eqref{e:PE-final} for the final bound obtained. Accordingly we implicitly assume from here on that the event $\{\widetilde{\mathcal U} =  \{\mathscr C\}\}$ occurs, with $\widetilde{\mathcal U}$ satisfying \eqref{Uconnect}-\eqref{Upartition}, and we consider a given non-trivial partition $\{\mathscr C\}_{1, 2} \stackrel{{\rm def}.}{=} (\{\mathscr C\}_1, 
\{\mathscr C\}_2)$ of $\{\mathscr C\}$ corresponding to the realization of $(\widetilde{\mathcal{U}}_1,\widetilde{\mathcal{U}}_2)$ in \eqref{Upartition}.

Towards bounding the relevant connection probability in \eqref{eq:Ai1}, we first apply
Lemma~\ref{L:goodpath_supercrit} to extract for each $0\leq j< \floor{\sqrt{M}} $ a coarse-grained  $*$-crossing of $ 
A_j$ that lies close to both $S_1$ and $S_2$ (with $S_i 
= \bigcup_{\mathscr C \in \{\mathscr C\}_i} {\rm supp}(\mathscr C)$), as follows.  
 Consider the (rescaled) coarse-grainings of $S_1$ and $S_2$ given by
 \begin{equation*}
\begin{split}
&U \stackrel{{\rm def.}}{=}  \frac{1}{10N} \{ x \in 10N\,\Z^d : B_{10N}(x) \cap S_1 \neq \emptyset \} \mbox{ and}\\
&V \stackrel{{\rm def.}}{=} \frac{1}{10N} \{ x \in 10N\,\Z^d : B_{10N}(x) \cap S_2 \neq \emptyset \}.
\end{split}
\end{equation*}
Owing to \eqref{Uconnect}, Lemma~\ref{L:goodpath_supercrit} applies and 
yields a $*$-path $\pi$ satisfying \eqref{eq:cond_path} and connecting 
\begin{equation*}
\begin{split}
V_{{\rm out}, 2j + 2} &\stackrel{{\rm def.}}{=} \frac1{10N}\{x \in 10N \Z^d : B_{10N}(x) \cap B_{4M - (2j 
+ 2)\sqrt{M} + 100N} \ne \emptyset\}	\text{ and}\\
V_{{\rm in}, 2j + 1} &\stackrel{{\rm def.}}{=} \frac1{10N}\{x \in 10N \Z^d : B_{10N}(x) \cap B_{4 M - (2j + 1)\sqrt{M} - 100N} =\emptyset\}
\end{split}	
\end{equation*}
if $A_j = V_{2j + 1} \setminus V_{2j + 2}$, and connecting $V_{{\rm out}, 2j+ 1}$ and $V_{{\rm 
in}, 2j}$ otherwise. Together, 
\eqref{eq:cond_path} and the $*$-connectivity imply that one can 
 extract from $\pi$ a family of boxes $\Lambda=\{ \Lambda_k : 1\leq k \leq K\}$, where $K = \lceil \frac{ \sqrt{M} }{400 N}\rceil$ and $\Lambda_k=
B_{20N}(x_{k})$ with $x_k \in 10N \Z^d$, for all $1 \le k \le K$, 
satisfying the following properties whenever $N \geq C$ and $\gamma_M \ge 2$: 
\begin{align}
&\label{prop:a} \text{$\L_k$ intersects both $S_1$ and $S_2$ and 
$\widetilde{\Lambda}_k \stackrel{{\rm def.}}{=} B_{25N}(x_{k}) \subset A_j$ 
for each $k$, and}\\
&\label{prop:c} \text{$|x_{k}- x_{k + 1}|_\infty \leq 200N \,\, \text{whereas } \, |x_{k}- x_{k'}|_\infty \geq 100N \text{ for all }  1 \leq k \neq k'  \leq K.$}
\end{align}	
Then clearly, by \eqref{prop:a} and monotonicity,
\begin{equation}\label{eq:repeated_encounter}
\P_{\mathscr C}\Big[ \nlr{A_j}{\mathcal{V}^{u-\delta}}{S_1}{S_2} \Big] \leq  
\P_{\mathscr C}\Big[  \bigcap_{k = 1}^K \{\nlr{\widetilde{\Lambda}_k}{\mathcal{V}^{u-\delta}}{S_1}{S_2}\} \Big].
\end{equation}
We stress that the collection $\Lambda = \Lambda(\{\mathscr C\}_{1, 2})$ is a deterministic function of $\{\mathscr C\}_{1, 2}$, realization of a partition $(\widetilde{\mathcal{U}}_1,\widetilde{\mathcal{U}}_2)$ of $\widetilde{\mathcal U}$ satisfying \eqref{Upartition}. The remainder of the proof applies uniformly in $j$ over $0\leq j< \floor{\sqrt{M}} $, which will be assumed implicitly to keep notations reasonable.

\bigskip

Let us briefly pause to describe how the argument unfolds. In order to bound the right-hand side of \eqref{eq:repeated_encounter}, we will first use the deterministic construction from Section~\ref{sec:bridging} to build a {\em bridge} between $S_1$ and $S_2$ inside {\em each} box $\widetilde \L_k$. Each bridge
consists of a chain of boxes at different scales, organized in a hierarchical fashion and satisfying suitable separation properties; see \eqref{B1}-\eqref{B4}. These bridges should be thought of as `efficient' pathways 
linking $S_1$ and $S_2$ inside each $\widetilde \L_k$. Their existence is guaranteed by Proposition~\ref{prop:general_bridge}. Clearly, the right hand side in \eqref{eq:repeated_encounter} is bounded from above by requiring that none of the bridges contains a connection between $S_1$ and $S_2$ in $\mathcal V^{u - \delta}$.

The method for bounding (from below) the probability that the bridge inside some 
$\widetilde \L_k$ connects $S_1$ and $S_2$ in 
$\mathcal V^{u - \delta} \cap \widetilde \L_k$ will make 
great use of so-called {\em good} 
boxes (in the bridge), which come in two types. They either ensure a small but reasonable probability for the existence of crossings inside them either via 
a conditional decoupling property (cf.~Proposition~\ref{prop:cond_decoup}) or via the `sprinkled finite-energy' property exhibited in Section~\ref{sec:sfe}. 

\bigskip

We now introduce the relevant notions of `decoupling-goodness' and `finite-energy goodness' that will play a central role here. 
Recall the event $\mathcal G_{B}^{\cdot, \cdot }$ from the statement of Proposition~\ref{prop:cond_decoup}, see above \eqref{eq:corcond_decouple}
A box $B = B(x,r)$ is said to be {\em decoupling-good (with parameters $(u, \delta)$)} 
if 
\begin{equation}\label{e:-dec-good}
G_B = \mathcal G_B^{u, \delta u^{-1}}
\end{equation}
occurs. The relevant notion of finite-energy goodness is provided by Proposition~\ref{lem:finite_energy}.
The box $B=B(x,r)$ is called \textit{finite-energy good (with parameters $(u,\delta, r_0)$)} if an event $\widetilde{F}_B$ with the properties postulated by Proposition~\ref{lem:finite_energy} occurs (for instance, \eqref{def:mathscrE-concrete} to be concrete).

We are now ready to introduce the bridge $\mathbb B^k$ corresponding to $\Lambda_k$ and the associated notion of good box. Let $\xi = 1 - \Cr{c:box_gap}$ (see Proposition~\ref{prop:cond_decoup}), $\Cr{C:bridge_sup} = m(\xi)$ as in Proposition~\ref{prop:general_bridge} and for a parameter $\gamma > 10$  to be fixed later, set
\begin{equation}\label{eq:gamma}{r_0}= 
\lfloor (\log N)^{\gamma}\rfloor.
\end{equation}
The parameter $r_0$ is henceforth fixed to have the value chosen in \eqref{eq:gamma}.
Applying Proposition~\ref{prop:general_bridge} with $(N,L)$ in \eqref{eq:def-tube} chosen as $(0, 20N)$, for each $1 \le k \le K$ we define $\mathbb B^k 
= \bigcup_{1 \le j \le J_k} \mathbb B_{j}^k$ to be a bridge associated to the septuple 
$(\L_k \cap S_1, \L_k \cap S_2$, $(6400)^4(\log N)^{4 \gamma}$, $32(\log N)^{\gamma}$, $\Cl{C:bridge_sup}$, $\xi, \L_k)$; for definiteness, if several bridges exist we choose the smallest one according to some deterministic ordering. We then introduce
\begin{multline}
\label{e:supercrit-good}
\{  \text{$\L_k$ is good}\} =\\ \big\{ \text{the 
boxes in $ \textstyle \mathbb B_{-}^k \stackrel{{\rm def}.}{=} \bigcup_{1 \le j < J_k} \mathbb B_j^k$ are decoupling-good with parameters~$(u,\delta)$}\big\},
\end{multline}
as well as
\begin{equation}\label{def:mathscrE}
\widetilde{F}_k= \big\{\text{the boxes in $\mathbb B_{J_k}^k$ are finite-energy good with parameters~$u, \delta$ and ${r_0}$}\big\},
\end{equation}
(with $r_0$ as in \eqref{eq:gamma}), and with bad referring to the complement of good in \eqref{e:supercrit-good}, define the random sets 
\begin{align}
&\label{e:K_b}
\mathcal K_b = \mathcal K_b(\{\mathscr C\}_{1,2}) = \big\{ k \in [K] \stackrel{{\rm def}.}{=} \{1, \ldots, K\} : \ \Lambda_k \text{ is bad}\big\},\\
&\label{e:K_b-tilde}
\widetilde{\mathcal K}_b = \widetilde{\mathcal K}_b(\{\mathscr C\}_{1,2}) = \big\{ k \in [K] : \ \Lambda_k \text{ is bad or } \widetilde{F}_k^c \text{ occurs}\big\}.
\end{align}
We reiterate at this 
point, as the notation in \eqref{e:K_b}-\eqref{e:K_b-tilde} suggests, that the collection of boxes $\Lambda = \Lambda(\{\mathscr C\}_{1,2})$, cf.~below \eqref{eq:repeated_encounter}, and with it the bridges $\{\mathbb B^k : 1 \le k \le K\}$ as well as the sets $\mathcal K_b$, $\widetilde{\mathcal K}_b$, 
are all implicitly functions of $\{\mathscr C\}_{1, 2} = (\{\mathscr C\}_1, \{\mathscr 
C\}_2)$, realization of a partition of $\widetilde {\mathcal U}$ satisfying 
\eqref{Upartition} on the event that $\{ \widetilde {\mathcal U}= \{\mathscr C\} \}$ with $\{\mathscr C\}$ such that $\{\widetilde {\mathcal U} = \{\mathscr C\}\} \, \cap \, E \ne \emptyset$, cf.~below \eqref{eq:Ai1}. Thus, with a slight abuse of notation, $\mathcal K_b$ (and similarly, $\widetilde{\mathcal K}_b$) is declared under both $\P_{\mathscr{C}}$ and $\P[ \cdot \cap E]$ and $\mathcal K_b$ is in the latter case
a measurable function of a (non-trivial) partition of $\widetilde{\mathcal U}$ satisfying \eqref{Upartition}, which counts the number of bad boxes associated to that partition. 
The following result controls $|\widetilde{\mathcal K}_b|$. For a real number $b \geq 1$, it will be convenient to abbreviate $\beta = 2^{-\lfloor (\log N)^{b}\rfloor}$ in the sequel.

\begin{lem}[$|\widetilde{\mathcal K}_b|$ is small]
\label{L:bound_bad}
There exists 
$\Cl{C:gamma_low_bnd1} \geq 10$ such that for any $b \ge 1$ and 
$\gamma \ge \Cr{C:gamma_low_bnd1} b$, the following holds. 
For all $\gamma_M \ge 200 b$, $v, \delta >0$, $u \in [2\delta,  v-2\delta]$ and $N \geq 1$,
\begin{equation}
	\label{eq:sumstep_bad}
		\P\big[ \big\{ |\widetilde{\mathcal K}_b| \ge 4K \beta \text{ for a (non-trivial) partition of } \widetilde{\mathcal U}\big\} \cap E\big] \leq  
		C \exp\big\{- ce^{c(\log N)^{8b}}\big\}, 
\end{equation}
for some $c=c(v,\delta)$ and $C = C(\gamma_M)$.
\end{lem}

Observe that $|\mathcal K_b| \leq |\widetilde{\mathcal K}_b|$, which is thus also controlled by \eqref{eq:sumstep_bad}. The proof of Lemma~\ref{L:bound_bad} would take us too far on a tangent and 
is postponed to \S\ref{subsec:proof_bound_bad}. It is worth pointing out that a naive union bound will not produce any useful estimate, see Remark~\ref{R:supercrit-badpts} for more on this. 

The next lemma supplies a complementary result to \eqref{eq:sumstep_bad} for a certain count of good indices, subset of $[K] \setminus \mathcal K_b$ in \eqref{e:K_b}, which we now introduce. We return to the `quenched' framework under $\P_{\mathscr{C}}$, with a given partition $\{\mathscr C\}_{1, 2}$ of $\mathscr C$ and the latter such that $\{\widetilde {\mathcal U} = \{\mathscr C\}\} \cap E \ne \emptyset.$
For any $k \in [K]$, $1 \le j < J_k$, and $B \in \mathbb 
B_{-}^k$ (see \eqref{e:supercrit-good} for notation), consider the connection events  
\begin{equation}\label{def:mthcalEB}
\begin{split}
&E_B = \big\{\lr{}{\mathcal V^u \cap B}{z_{1, B}}{z_{2, B}}\big\},\\
&E_{j, k} = \bigcap_{B \in \mathbb B^k_{j}} E_B, \   E_k =	\bigcap_{1 \le j < J_k} E_{j, k},
\end{split}
\end{equation}
where $z_{1, B}$ and $z_{2, B}$ in the definition of $E_B$ are the 
two marked vertices of $B$ declared by property~\eqref{B2} of the bridge $\mathbb{B}^k$. We write $\mathcal K_g = \mathcal K_g(\{\mathscr C\}_{1,2}) \subset [K] 
\setminus \mathcal K_b$ for the subset of indices corresponding to good boxes $\L_k$, i.e.~satisfying \eqref{e:supercrit-good}, and such that in addition, the event $E_k$ occurs. That is, 
\begin{equation}
\label{e:K_g}
\mathcal K_g = \mathcal K_g(\{\mathscr C\}_{1,2}) = \{ k \in [K]  : \ \Lambda_k \text{ is good and $E_k$ occurs}\};
\end{equation}
cf.~\eqref{e:K_b}. In plain words, and on account of \eqref{B2}, indices in $\mathcal K_g$ correspond to boxes $\L_k$ where the sets $S_1 \cap 
\L_k$ and $S_2 \cap \L_k$ are connected to each other through $\L_k \cap \mathcal V^u$ {\em except} for 
the boxes in $\mathbb B_{J_k}^k$ (referred to as {\em holes} in Section~\ref{sec:bridging}, see below \eqref{eq:def-tube}). 


\begin{lem}[$\mathcal K_g$ is large]\label{cor:ample_connection}
For all $\gamma > 10$, $b \ge 2 \Cr{explore_bad}$, $v>0$ as in the statement of Proposition~\ref{prop:uniq}, $\delta >0$, $u \in [2\delta,  v-2\delta]$ and $N \ge C(\delta,v)$,
	\begin{align}
		\label{eq:sumstep}
		\P_{\mathscr C}\big[ |\mathcal{K}_g| \leq 8K \beta\big] \leq 2^{-cK \beta} + 8K \, \P_{\mathscr C}\big[ |\mathcal K_b|  \ge 4K \beta\big].
	\end{align}
\end{lem}

Lemma~\ref{cor:ample_connection} is proved in \S\ref{sec:K_g}. With Lemmas~\ref{L:bound_bad} and~\ref{cor:ample_connection} at our disposal, we are in a position to conclude the proof of Lemma~\ref{lem:uniqreduct}. We first choose $b = 2(\Cr{explore_bad} \vee 10)$, 
$\gamma = \Cr{C:gamma_low_bnd1}b$ and assume that $\gamma_M \ge 200b$ so that the conclusions of 
Lemmas~\ref{L:bound_bad} and \ref{cor:ample_connection} hold for all $v>0$ as in the statement of Proposition~\ref{prop:uniq}, $\delta >0$, $u \in [2\delta,  v-2\delta]$ and $N \ge C= C(\delta,v)$; the dependence of constants on $\delta$ and $v$ will be kept implicit for the remainder of the proof. The previous choices also completely fix the parameters of the bridge, see above \eqref{e:supercrit-good} (only $\gamma$ remained to be chosen).

With $E_k$ as in \eqref{def:mthcalEB} and $\widetilde{F}_k$ as in \eqref{def:mathscrE}, let $\widetilde{\mathcal{K}}_g$ denote the subset of indices $k \in [K]$ 
corresponding to the boxes $\L_k$ such that the event $E_k \cap \widetilde{F}_k$ 
occurs. As will soon become clear, the indices in $\widetilde{\mathcal{K}}_g$ correspond to boxes in which, by virtue of Proposition~\ref{lem:finite_energy}, a full path in $\mathcal{V}^{u-\delta}$ joining $S_1$ and $S_2$ can be re-constructed at a not-too-degenerate cost. Lemmas~\ref{L:bound_bad} and \ref{cor:ample_connection} will then jointly be used to establish
that typically, the set $\widetilde{\mathcal{K}}_g$ has large cardinality. 

From the definition of the events $E_k$ and 
\eqref{eq:mathscrEBmeasurable}, it follows that for any $\widetilde K_g \subset [K]$,
\begin{multline}\label{eq:tildeKgmble}
\text{the event $\{\widetilde{\mathcal U }= \{\mathscr C\}, \widetilde {\mathcal K}_g = \widetilde K_g\}$ is} \\ \text{measurable relative to $\big(\mathcal I^u 
\cap B_{4M}, \, \mathcal I^{u-\delta} 
\cap (B_{4M} \setminus V_{2j}), \, \{\widetilde{\mathcal C}^{\widehat B}_u : B \in \mathbb B_{J_k}^k, k \in \widetilde K_g\}\big)$}.
\end{multline}
The important observation is that, if $E_k $ occurs and $S_1$ and $S_2$ are not connected in $\mathcal{V}^{u-\delta} \cap \widetilde{\Lambda}_k$ (cf.~\eqref{prop:a} regarding $\widetilde{\Lambda}_k$), then for $N \geq C$ not all of the boxes in $\mathbb B_{J_k}^k$ can be completely vacant in $\mathcal{V}^{u-\delta}$. For, otherwise, by definition of the events $E_B$ in \eqref{def:mthcalEB} and on account of property~\eqref{B2}, the set $\mathbb B^k \cap \mathcal{V}^{u-\delta}$ comprises a connection between $S_1$ and $S_2$. Moreover, by construction, see Section~\ref{sec:bridging}, any of the boxes in $\mathbb B^{k}$ is a subset of $\L_k + B_{s}$ with $s = 
(6400)^4(\log N)^{4\gamma}$ by our choice of parameters above \eqref{e:supercrit-good}, hence this connection is indeed contained in $\widetilde{\Lambda}_k$ whenever $N \geq C$. By this observation, if $ \widetilde K_g = \{k_1 
< \ldots < k_{\tilde \ell}\} \subset [K]$ such that $\tilde \ell = |\widetilde K_g| 
\geq 1$, $1 \le \ell \le \tilde \ell$ and $N \geq C$, abbreviating 
$$
\text{Disc}_{\ell}=\big\{\nlr{\widetilde\Lambda_{k_{i}}}{\mathcal{V}^{u-\delta}}{S_1}{S_2}, \text{ for all } 1 \le i < \ell \big\}
$$
and $\mathbb H^{k_\ell} = \mathbb B_{J_{k_{\ell}}}^{k_\ell}$ ($\mathbb H$ stands for `holes'), one obtains that
\begin{equation}\label{eq:mm-decomp}
\P_{\mathscr C}\big[\widetilde{\mathcal{K}}_g=\widetilde{K}_g,\, \text{Disc}_{\ell+1} \big]  
\leq \P_{\mathscr C}\bigg[\widetilde{\mathcal{K}}_g=\widetilde{K}_g,\, 
\text{Disc}_{\ell}, \, \Big(\bigcap_{B \in \mathbb H^{k_\ell} }\{B \subset \mathcal V^{u - \delta}\}\Big)^c\bigg].
\end{equation}
We now prepare the ground to decouple the events indexed by $B \in \mathbb H^{k_\ell}$ in \eqref{eq:mm-decomp}, for which Proposition~\ref{lem:finite_energy} will come into play. First, let us recollect some facts about the separation between boxes. Recall from \eqref{eq:gamma} that ${r_0}= \lfloor (\log N)^{\gamma}\rfloor$. By property~\eqref{B1} for the bridge $\mathbb B^{k_{\ell}}$, for any two distinct boxes $B, B' \in \mathbb H^{k_\ell}  $, their $s'$-thickenings $B^{s'}$ and $(B')^{s'}$, where $s' = 32(\log N)^{\gamma} \ge 32 {r_0}$ by our choices, are disjoint sets. Furthermore, any box in $ B \in \mathbb H^{k_\ell}$ is a subset of $\L_{k_{\ell}} + B_{s}$ with $s = (6400)^4(\log N)^{4\gamma}$ by definition of a bridge. Moreover,
$\widetilde \L_k \subset A_j$ by~\eqref{prop:a} and the $\ell^\infty$-distance between $\L_{k}$ and $\L_{k'}$ is at least $50N 1\{k \ne k'\}$ by~\eqref{prop:c}. Hence, the box $ B \in \mathbb H^{k_\ell}$ is a subset of $A_j$ and is disjoint 
from $\widetilde \L_{k}$ for any $k \ne k_{\ell}$ as soon as $N \ge C$. 

For a box $B=B(x,r)$, let $\widehat{B}= B(x, r+7r_0)$ (cf.~the statement of Proposition~\ref{lem:finite_energy}). It will also be convenient to assume an arbitrary ordering of the boxes $ \mathbb H^{k_\ell}$, and $B' < B$ refers to the boxes $B' \in \mathbb H^{k_\ell}$ appearing before $B \in \mathbb H^{k_\ell}$ in this ordering. Combining the above observations on the location of the holes, \eqref{eq:tildeKgmble} and recalling that $\widetilde{\mathcal C}^{\widehat{B}}_u$ is measurable relative to $\omega_{\widehat{B}}^-$ (see below \eqref{def:omega-BPhiBu}), it follows that for any $1 \le \ell \le \tilde{\ell}$ and $B \in \mathbb H^{k_\ell}$,  whenever $N \ge C$, 
\begin{equation}\label{eq:big_mblity}
\Big\{\widetilde{\mathcal U }= \{\mathscr C\}, \, \widetilde  {\mathcal K}_g= \widetilde K_{g}, \text{Disc}_{\ell}, \bigcap_{B' < B}\{B' \subset \mathcal V^{u - \delta}\}\Big\}
 \text{ is measurable~relative to  }(\omega_{{\widehat B}}^-,\, \mathcal I^u \cap \widehat B). 
 \end{equation}	
Let us now re-write, for any  $1 \le \ell \le \tilde{\ell}$ (with $k_\ell \in \widetilde{K}_g$)
\begin{multline*} 
\P_{\mathscr C}\Big[\widetilde{\mathcal{K}}_g=\widetilde{K}_g,\, 
 \text{Disc}_{\ell}, \, \bigcap_{B \in \mathbb H^{k_\ell} }\{B \subset \mathcal V^{u - \delta}\} \Big]\\
= \P_{\mathscr C}\big[ \widetilde{\mathcal{K}}_g=\widetilde{K}_g,\, 
\text{Disc}_{\ell}\big] \prod_{B \in \mathbb H^{k_\ell} }\P_{\mathscr C}\Big[B \subset \mathcal V^{u - \delta} \, \Big | \, \widetilde{\mathcal{K}}_g=\widetilde{K}_g,\,  \text{Disc}_{\ell}, \, \bigcap_{B' < B}\{B' \subset \mathcal V^{u - \delta}\}\Big].
\end{multline*}
Since $k_\ell \in \widetilde {\mathcal K}_g$ on the event $\{\widetilde {\mathcal K}_g = \widetilde K_g\}$, we obtain the following lower bound by applying \eqref{eq:finite_energy}, thus using \eqref{eq:big_mblity} and keeping in mind that 
the event $\widetilde{F}_B$ occurs for any box $B \in \mathbb H^{k_\ell}$ by definition of ${\mathcal K}_g$ and on account of \eqref{def:mathscrE}: for all $N \ge C $,
\begin{equation*}
\begin{split}
\P_{\mathscr C}\Big[\widetilde{\mathcal{K}}_g=\widetilde{K}_g,\, 
 \text{Disc}_{\ell}, \bigcap_{B \in \mathbb H^{k_\ell}}\{B \subset \mathcal V^{u - \delta}\}\Big] \ge \P_{\mathscr C}\big[ \widetilde{\mathcal{K}}_g=\widetilde{K}_g,\, 
\text{Disc}_{\ell}\big]  \times \exp\big\{-(\log N)^{C} |\mathbb H^{k_\ell}|\big\} ;
\end{split}	
\end{equation*}
in applying \eqref{eq:finite_energy} to obtain the right-hand side, we also used that
${r_0}= \lfloor (\log N)^{\gamma}\rfloor$ and the fact that any box in $B \in \mathbb H^{k_\ell}$ has radius $r \leq s = 
(6400)^4(\log N)^{4\gamma}$ by~\eqref{B3} and our choice of parameters. Moreover, by property~\eqref{B4} applied to $\mathbb H^{k_\ell} (= \mathbb B_{J_{k_{\ell}}}^{k_\ell})$, we have that $ |\mathbb H^{k_\ell}| \le  (\log N)^C$ and hence we deduce by combining the previous display and \eqref{eq:mm-decomp} that for $ \widetilde K_g = \{k_1 
< \ldots < k_{\tilde \ell}\} \subset [K]$ such that $\tilde \ell = |\widetilde K_g| 
\geq 1$, all $1 \le \ell \le \tilde \ell$ and $N \geq C$,
\begin{equation*}
	\begin{split}
		\P_{\mathscr C}&\big[\widetilde{\mathcal{K}}_g=\widetilde{K}_g,\, 
		  \text{Disc}_{\ell+1} \big]  \le \big(1 - e^{-(\log N)^{C'}}\big) \P_{\mathscr C}\big[\widetilde{\mathcal{K}}_g=\widetilde{K}_g,\,  \text{Disc}_{\ell}\big].
	\end{split}	
\end{equation*}
Iterating this inequality over all $1 \le \ell \le \tilde{\ell}$, we thus obtain, for all $\widetilde K_g \subset [K]$ and $N \geq C$,
\begin{equation}\label{eq:decay_disconnect}
\begin{split}
\P_{\mathscr C}\Big[\widetilde{\mathcal{K}}_g=\widetilde{K}_g,\, \bigcap_{k \in \widetilde{K}_g }\{\nlr{\widetilde\Lambda_{k}}{\mathcal{V}^{u-\delta}}{S_1}{S_2}\}\Big]	
&\le \exp\big\{- |\widetilde{K}_g| e^{-(\log N)^{C}} \big\} \P_{\mathscr C}\big[\widetilde{\mathcal{K}}_g=\widetilde{K}_g\big].
\end{split}	
\end{equation}
 
We now return to bounding the disconnection probability on the right-hand side of \eqref{eq:repeated_encounter} from above (which will later feed into \eqref{eq:Ai1} and yield the desired upper bound on $\P[E]$). We will only take advantage of the cost occasioned by the absence of connections between $S_1$ and $S_2$ in boxes $\Lambda_k$ with indices $k \in \widetilde{\mathcal{K}}_g$, which is quantified by \eqref{eq:decay_disconnect}. Its utility relies on $|\widetilde{\mathcal{K}}_g|$ being sufficiently large. Lemmas~\ref{L:bound_bad} and~\ref{cor:ample_connection} now come into effect. We thus write
\begin{equation}
\label{eq:final-S-i-bound-trichotomy}
\P_{\mathscr C}\Big[ \bigcap_{k \in [K] } \nlr{\widetilde{\Lambda}_k}{\mathcal{V}^{u-\delta}}{S_1}{S_2}\Big] \leq a_1 + a_2 + a_3,
\end{equation}
where we set, with $\mathcal{K}_g$ as introduced in \eqref{e:K_g} and $b$ as chosen at the start of the proof,
\begin{align*}
&a_1 = \P_{\mathscr C}\Big[ \bigcap_{k \in \widetilde{\mathcal{K}}_g }\{\nlr{\widetilde\Lambda_{k}}{\mathcal{V}^{u-\delta}}{S_1}{S_2}\} , \,  |\widetilde{\mathcal{K}}_g| \geq 4K \beta\Big], \\
&a_2 = \P_{\mathscr C}\big[ |\widetilde{\mathcal{K}}_g| < 4K \beta, \,  |\mathcal{K}_g| > 8K \beta \big],\quad a_3 = \P_{\mathscr C}\big[ |\mathcal{K}_g| \leq 8K \beta\big].
\end{align*}
We proceed to bound each term individually. Since $K > M^{\frac12}/ 400 N$, see above \eqref{prop:a}, where $M=M(N)$ is given by \eqref{eq:def_M}, it immediately follows by decomposing $a_1$ over realizations of $ \widetilde{\mathcal{K}}_g$ and using \eqref{eq:decay_disconnect} that for all $ \gamma_M \geq C$ and $N \geq C'(\gamma_M)$, one has $a_1 \leq \exp\{-  \, e^{ (\log M)^{c}}\big\}$ with $c=c(\gamma_M)$. By virtue of Lemma~\ref{cor:ample_connection}, which is in force, $a_3$ is bounded by the sum of a similar term and $C K \, \P_{\mathscr C}[ |\widetilde{\mathcal K}_b|  \ge 4K \beta]$ whenever $N \geq C(\delta,v)$. Note that for the latter, which corresponds to the second term on the right-hand side of \eqref{eq:sumstep}, we have replaced $|{\mathcal K}_b| $ by $|\widetilde{\mathcal K}_b| \geq |{\mathcal K}_b| $, cf.~\eqref{e:K_b}-\eqref{e:K_b-tilde}. Finally, the event defining $a_2$ implies that $|\mathcal{K}_g \setminus \widetilde{\mathcal{K}}_g| \geq 4K \beta$. But on account of \eqref{e:K_g} and since $ k \in \widetilde{\mathcal{K}}_g$ requires both $E_k$ and $\widetilde{F}_k$ to occur, for any $k \in (\mathcal{K}_g \setminus \widetilde{\mathcal{K}}_g)$ one knows that $\widetilde{F}_k^c$ occurs. In view of \eqref{e:K_b-tilde}, this means that $(\mathcal{K}_g \setminus \widetilde{\mathcal{K}}_g ) \subset  \widetilde{\mathcal{K}}_b$, hence $a_2  \leq \P_{\mathscr C}[ |\widetilde{\mathcal K}_b|  \ge 4K \beta] $. 

Inserting the resulting bounds for $a_i$, $1 \leq i \leq 3$, into \eqref{eq:final-S-i-bound-trichotomy}, and substituting back into \eqref{eq:repeated_encounter} and subsequently into \eqref{eq:Ai1}, observing while doing so that $\mathscr C$ is a partition of a subset of $\mathcal C$ where $|\mathcal C| \le  C M^{d-1}$ by \eqref{eq:definiq1}, whence the second sum in \eqref{eq:Ai1} over the number of partitions $\{\mathscr C\} 
= \{\mathscr C\}_1 \sqcup \{\mathscr C\}_2$ with $|\{\mathscr C\}_1| \le 4$ is over at most 
$C'M^{4(d-1)}$ terms, it follows that 
\begin{multline}\label{e:PE-final}
\P[E] \leq M^C \exp\big\{-  e^{ (\log M)^{c(\gamma_M)}}\big\} \  \\ + \ CK \sum_{\{\mathscr C\}} \P\big[\,\widetilde{\mathcal U} = \{\mathscr C\}\big] \times 
\sum_{\substack{\{\mathscr C\} = \{\mathscr C\}_1 \sqcup \{\mathscr C\}_2\\ |\{\mathscr C\}_1| \le 4}}\,  \P_{\mathscr C}\big[ |\widetilde{\mathcal K}_b|  \ge 4K \beta\big], 
\end{multline}
for all $\gamma_M \geq C$ and $N \geq C'(\gamma_M, \delta,v)$. Recalling that the sum over $\{\mathscr C\}$ is over realizations of $\widetilde {\mathcal U}$ such that $\{\widetilde {\mathcal U} = \{\mathscr C\}\} \cap E \ne \emptyset$, the two summations in the second line of \eqref{e:PE-final} re-combine to give left-hand side of \eqref{eq:sumstep_bad}, and \eqref{eq:uniqreduct_new1} follows directly from \eqref{e:PE-final} and Lemma~\ref{L:bound_bad}.
\end{proof}

\section{Controlling $|{\mathcal K}_g|$ and $ |\widetilde{\mathcal K}_b|$}
\label{Sec:2lemmas}

In this section, we supply the outstanding proofs of Lemmas~\ref{L:bound_bad} and~\ref{cor:ample_connection}, in this order, which where admitted in the course of proving Lemma~\ref{lem:uniqreduct} in Section~\ref{Sec:surgery}. We then reflect on the arguments of Sections \ref{Sec:surgery}-\ref{Sec:2lemmas} by making a few concluding observations, see Remark~\ref{R:final}.

\subsection{Proof of Lemma~\ref{L:bound_bad}}\label{subsec:proof_bound_bad}

In this section we supply the missing proof of Lemma~\ref{L:bound_bad}. We start with a combinatorial result. Recall that $N, b \geq 1$ and $K = \lceil \frac{ \sqrt{M} }{400 N}\rceil$ with $M=M(N)$, cf.~above \eqref{prop:a}, which depends implicitly on the choice of $\gamma_M$, see \eqref{eq:def_M}. Recall that $\beta = 2^{-\lfloor (\log N)^{b}\rfloor}$ and that $[K]=\{1,\dots, K\}$.

\begin{claim}[$b \geq 1$] \label{Cl1-supercrit} Let $\Gamma = 
	2^{\lfloor (\log N)^{10b}\rfloor}$. For all $N \geq C$ and $\gamma_M \ge 30b$ and ${\mathcal{K}} \subset [K]$ with $|{\mathcal K}| \ge  \beta K$, there exists $ \mathcal{K}'  \subset \mathcal K$ with $| \mathcal{K}'  | = \beta \Gamma$ such that, with $\mathcal{K}' = \{ k_1', \dots, k_{\beta \Gamma}' \}$, one has $(0 <) k_{i+1}' - k_i' \le \Gamma$ for all $1 \le i < \beta \Gamma$.
\end{claim}

\begin{proof}Let $k_1 < \ldots < k_{|{\mathcal K}|}$ denote the (ordered) elements in ${\mathcal K} \subset [K]$. Let ${I} \subset [\, |{\mathcal K}| \,]$ denote the subset of indices $i$ such that $k_{i+1} - k_i > \Gamma$. 
	Notice $|{I}|$ is at most $\frac{K}{ \Gamma}$ for otherwise we would have $k_{|{\mathcal K}|} > K$. 
	Consequently, the number of {\em maximal} subintervals of $[\, |{\mathcal K}| \,] \setminus {I} $ is {\em at most} $1 + \frac{K}{\Gamma}$, whereas the sum of their lengths is {\em at least} $|{\mathcal K}| - 
	\frac{K}{\Gamma}$. Let ${I'}$ denote a subinterval with maximum length among these. It follows that	\begin{equation*}
		|{I'}| \ge  \frac{|{\mathcal K}| - \frac{K}{\Gamma}}{1 + \frac K{\Gamma}} \ge  1.5 \beta \Gamma- 1 \ge \beta \Gamma
	\end{equation*}
	for all $N \ge C$ and $\gamma_M \ge 30b$, where we used that $|{\mathcal K}| \ge \beta K$ 
	(recall that $K \ge \tfrac{\sqrt{M}}{40 N} = \tfrac{1}{40 N}e^{(\log N)^{\gamma_M/2}}$ and $\beta = 2^{-\lfloor (\log 
		N)^b \rfloor}$). Writing the elements of ${I'}$ as $j, \ldots, j + |{I'}| - 1$, it is a consequence of the 
	definition of ${I'}$ that $k_{n + 1} - k_n \le \Gamma$ for all $1 \le n < j + |{I'} - 1|$. Since $|{I'}| \ge 
	\beta \Gamma$, we can simply take $\mathcal{K}' = \{k_j, \ldots, k_{j + \beta \Gamma - 1}\}$. 	
\end{proof}

Claim~\ref{Cl1-supercrit} will soon be applied to extract from a given large collection of boxes (indexed by $[K]$, cf.~above \eqref{prop:a}) a large sub-collection of not too distant boxes, cf.~\eqref{prop:a}.

\bigskip

We now proceed with the proof Lemma~\ref{L:bound_bad}, which occupies the remainder of this paragraph.

\begin{proof}[Proof of Lemma~\ref{L:bound_bad}] We start by replacing $\widetilde{\mathcal K}_b $ by a larger set $ \widehat{\mathcal K}_b$, which is more convenient to work with.
Consider the event (cf.~\eqref{def:mathscrE})
\begin{equation}\label{def:widehatF}
\widehat{F}_k \stackrel{{\rm def.}}{=} \bigcap_{B \in \mathbb B_{J_k}^k} \widehat{F}_B ,
\end{equation}
where $\widehat{F}_B= \widehat{F}_B(u, u, u, \delta/2, \delta, r_0)$ refers to the event introduced in \eqref{eq:fin_en_inclusion}, and $r_0 = \lfloor (\log N)^{\gamma}\rfloor$ as per our choice in \eqref{eq:gamma}, which will remain implicit in all subsequent notation. Define the corresponding random set $\widehat{\mathcal K}_b$ (cf.~the definition of $\widetilde{\mathcal{K}}_b$ in \eqref{e:K_b-tilde})
\begin{equation}\label{e:K_b-hat}
\widehat{\mathcal K}_b = \widehat{\mathcal K}_b(\{\mathscr C\}_{1,2}) = \big\{ k \in [K] : \ \Lambda_k 
\text{ is bad or } \widehat{F}_k^c \text{ occurs}\big\}.
\end{equation}
As noted in 
Remark~\ref{R:FE-implication-event} above, $\widetilde{F}_B^c \subset \widehat{F}_{B}^c$ and hence $\widetilde{\mathcal K}_b \subset 
\widehat{\mathcal K}_b$. Thus \eqref{eq:sumstep_bad} follows from the stronger statement
\begin{equation}
\label{eq:sumstep_badhat}
	\P\big[ \big\{ |\widehat{\mathcal K}_b| \ge 4 \beta K  \text{ for a (non-trivial) partition of } \widetilde{\mathcal U}\big\} \cap E\big] \leq  
	C \exp\big\{- e^{c(\log N)^{8b}}\big\},
\end{equation}
which we proceed to show. To start with, we extract from $\widehat{\mathcal K}_b$ a large sub-family $\widehat{\mathcal K}_b'$ of not-too-distant boxes using Claim~\ref{Cl1-supercrit}, as follows. Let $\mathbb{L}_N=10N\Z^d$. Recall that the centers $x_k$ of all boxes in the collection $\Lambda=\{ \Lambda_k : 1\leq k \leq K\}$, part of which have an index $k$ that belongs to $\widehat{\mathcal K}_b$, all lie on the lattice $\mathbb{L}_N$; see above \eqref{prop:a}. In the sequel, for $z \in \mathbb{L}_N$ and integer $\ell \ge 0$, we write $D_{z, \ell} = z + 10N\cdot[0, 2^{\ell})^d \subset \mathbb{L}_N$. Now suppose the event in 
	\eqref{eq:sumstep_badhat} occurs. Then by Claim~\ref{Cl1-supercrit} applied to $\mathcal{K}= \widehat{\mathcal{K}}_b$, there exists $\widehat{\mathcal{K}}_b' \subset \widehat{\mathcal{K}}_b$ and a corresponding
	family of boxes $(\L_k)_{k \in \widehat{\mathcal K}_b'}$ satisfying the defining condition in \eqref{e:K_b-hat}, such that $|\widehat{\mathcal K}_b'| \geq \beta \Gamma$ 
	and any two consecutive indices in $\widehat{\mathcal K}_b'$ are at most $\Gamma$ apart. Consequently, it follows from 
	property \eqref{prop:c} of the family $\Lambda$ that 
	\begin{equation}\label{eq:dyadic_inclusion}
		\{ x_k : k \in \widehat{\mathcal K}_b'\} \subset D_{z, \ell_0}\, \text{ with $\ell_0 = \log_2 2\Gamma^2$ and $z \in V_{2j}$.}
	\end{equation}
For all values of $(u \geq ) u_1 \ge u_2 \ge u_3 > \delta_2 > \delta_1 > 0$ and to each box $\L = B(x,20N)$, $x \in \mathbb{L}_N$, with $\widetilde \L = B(x, 25N)$, cf.~\eqref{prop:a}, we now attach the (likely) events $\widehat{H}_{\L}$ with
\begin{equation}\label{def:H_B}
		\widehat{H}_{\L}^c = \widehat{H}_{\L}^c(u_1, u_2, u_3, \delta_1, \delta_2) \stackrel{{\rm def}.}{=} \bigcap_{\substack{B \subset \widetilde\L,\\ {\rm rad}(B) \ge 16{r_0}}}  \widehat{G}_B(u_1,\delta_2) \,  \cap \,  \bigcap_{B \subset \widetilde\L} \widehat{F}_{B}(u_1, u_2, u_3, \delta_1, \delta_2),
	\end{equation}
with $\widehat{F}_B$ as above \eqref{eq:fin_energy_good} and where $\widehat{G}_B(u_1,\delta_2) = \{ \mathcal{C}_{u_1} \in \Xi_{B}^{u, \delta_2 u^{-1}}\}$, see \eqref{def:GBudelta} for notation, and ${r_0}= \lfloor (\log N)^\gamma \rfloor$ as before. The threshold $u$ implicit in $\widehat{G}_B$ will also remain fixed throughout and not appear in our notation; it refers to the the same value $u$ entering the definition of $\widetilde{\mathcal K}_b$ (whose cardinality we are trying to bound) through \eqref{e:supercrit-good}-\eqref{def:mathscrE}. 
	
	The event $\widehat{H}_{\L}$ in \eqref{def:H_B} is designed as follows. The `decoupling-good' event $G_B$ introduced in \eqref{e:-dec-good}, which enters the definition of a good box $\Lambda_k$ in \eqref{e:supercrit-good}, satisfies $G_B= \widehat{G}_B(u,\delta)$. Thus, recalling the specifics of our bridge above \eqref{e:supercrit-good}, which entail among other things that any box constituting a bridge in $\Lambda_k$ has radius exceeding $16r_0$ and is contained in $\widetilde \Lambda_k$, it follows in view of \eqref{e:supercrit-good} and \eqref{def:widehatF} that $\widehat{H}_{\L_k}^c(u, u, u, \frac\delta2, \delta) \subset \{\text{$\L_k$ is good and $\widehat{F}_k$ occurs}\}$. Thus, if $k \in \widehat{\mathcal K}_b$ as in \eqref{e:K_b-hat}, then the event $\widehat{H}_{\L_k}(u, u, u, \frac\delta2, \delta)$ occurs. 
	
 Now, for any $\ell \in \N$, $q \in \Z$, $Q = 2^{q}$ and $\Lambda_x= B(x,20N)$, let $n_{z, \ell}= \sum_{x \in D_{z, \ell}} 1_{\widehat{H}_{\Lambda_{x}} }$, which depends on $\xi=(u_1, u_2, u_3, \delta_1, \delta_2)$, and define
	\begin{equation}\label{eq:def-h-supercrit}
		\begin{split}
			&H_{\ell, q; z} = H_{\ell, q; z}(\xi)  \stackrel{{\rm def}.}{=}  \{n_{z, \ell} (\xi) \ge Q \}, \\  
			&h_{\ell, q} = h_{\ell, q} (\xi) \stackrel{{\rm def}.}{=} \P\left[ H_{\ell, q;0}\right]
 \ \big( =\P\left[ H_{\ell, q;z}\right] \big),
		\end{split}
	\end{equation}
where the last equality is due to translation invariance). It 
	then follows by the previous observation regarding the implications of the event $\widehat{H}_{\L}^c$ and \eqref{eq:dyadic_inclusion}, recalling that $\widehat{\mathcal{K}}_b' \subset \widehat{\mathcal{K}}_b$ has cardinality exceeding $\beta \Gamma= 2^{q_0}$ with
$q_0 = \lfloor (\log N)^{10b}\rfloor - \lfloor (\log N)^{b}\rfloor$, that the left-hand side of \eqref{eq:sumstep_badhat} is bounded by $M^d h_{\ell_0, q_0}(u, u, u, \delta/2, \delta)$ with $\ell_0$ as in \eqref{eq:dyadic_inclusion}. Hence, to obtain \eqref{eq:sumstep_badhat}, it is sufficient to prove that
	\begin{equation}\label{eq:pellq_bnd}
		h_{\ell_0, q_0}(u, u, u, \delta/2, \delta) \le C \exp\big\{- ce^{c(\log N)^{8b}}\big\} \quad(\text{with } q_0 = \lfloor (\log N)^{10b}\rfloor - \lfloor (\log N)^{b}\rfloor),	\end{equation}
	for $c=c(v,\delta)$ and suitable values of $\gamma$ and $\gamma_M$, as in the statement of Lemma~\ref{L:bound_bad}. 
	
	In order to bound $h_{\ell_0, q_0}$, we set up a recursive inequality for the functions $h_{\ell,q}$ in \eqref{eq:def-h-supercrit}. To this end fix an integer $\lambda \in [0, \ell]$. As we now explain, for any choice of $u_1 \ge u_2 \ge u_3 > \delta_2 > 
\delta_1 > 0$, implicit in the sequel, and $z \in \mathbb{L}_N$, the inclusion
\begin{multline}\label{eq:Erqinclusion}
		H_{\ell, q; z} \subset \Big(\bigcup_{\substack{x, y \in D_{z, \ell} \, \cap \, (z + 2^{\ell-\lambda}\mathbb{L}_N), \\ 2^{-(\ell-\lambda)}(10N)^{-1}|x - y|_{\infty} \, \ge \, 2}} \big(H_{\ell-\lambda, q-d\lambda -1; x} \cap 
		H_{\ell-\lambda, q-d\lambda -1; y}\big)\Big) \\
		\, \cup \, \Big(\bigcup_{x \in D_{z, \ell} \, \cap \, (z+ 2^{\ell-\lambda}\mathbb{L}_N)} H_{\ell-\lambda, q - 2d - 1; x}\Big)
\end{multline}
holds. Indeed, recall that $D_{z, \ell} = z + 10N\cdot[0, 2^{\ell})^d$ (viewed as a subset of $\mathbb{L}_N = 10N \Z^d$), and let $D_{z, \ell}^{\lambda} = D_{z, \ell} \, \cap \, (z + 2^{\ell-\lambda}\mathbb{L}_N)$. Since the boxes in $\{D_{x, \ell-\lambda}: x \in D_{z, \ell}^\lambda\}$ form a partition of $D_{z, \ell}$, we have on 
the event $H_{\ell, q; z}$ that $$n_{z, \ell}= \sum_{x \in D_{z, \ell}^{\lambda}} n_{x, \ell-\lambda} \geq Q .  $$
 Clearly then, for some $x \in D_{z, \ell}^\lambda$, the bound $n_{x, \ell-\lambda} \ge Q 2^{-d\lambda}=2^{q-d\lambda}$ holds. In particular, in view of \eqref{eq:def-h-supercrit} this means that~$H_{\ell-\lambda, q - d\lambda - 1; x}$ occurs (the seemingly unnecessary $-1$ will become clear momentarily). Now suppose that the first event on the right 
hand side of the inclusion in \eqref{eq:Erqinclusion} does {\em not} happen. Then by definition of $H_{\ell-\lambda, q-d\lambda -1; \cdot}^c$ in \eqref{eq:def-h-supercrit} it must be the case 
that for $x$ as above and any $y \in D_{z, \ell}^\lambda$ such that $2^{-(\ell-\lambda)}(10N)^{-1}|x - y|_{\infty} \ge 2$,
one has $ n_{y, \ell-\lambda} < Q 2^{-d\lambda - 1}$.
However, this yields that
\begin{equation*}
	\sum_{\substack{y \in D_{z, \ell}^\lambda : \, 2^{-(\ell-\lambda)}(10N)^{-1}|x - y|_{\infty} \le 1}} \, n_{y, \ell-\lambda} \ge   Q - Q 2^{-d\lambda - 1}2^{d\lambda} = Q/2.
\end{equation*}
Since the number of $y$ satisfying the constraint in the first summation is at most $4^d = 2^{2d}$, it 
follows that for one such $y$ we must have $n_{y, \ell-\lambda} \ge Q 2^{-2d - 1} $, i.e.~$H_{\ell-\lambda, q - 2d - 1; y}$ occurs. Overall this yields \eqref{eq:Erqinclusion}.

We now study the implication of 
\eqref{eq:Erqinclusion} in terms of the corresponding probabilities $h_{\cdot, \cdot}$ in \eqref{eq:def-h-supercrit}, with the aim of showing~\eqref{eq:pellq_bnd}. By a union bound, one immediately writes
\begin{equation}\label{eq:first_recursion}
	\begin{split}
		h_{\ell, q} \le \, 4^{d\lambda}\sup_{x \in \tilde D_{0, \ell}^\lambda}\P\left[H_{\ell-\lambda, q - d\lambda - 1; 0}  \cap H_{\ell-\lambda, q - d\lambda - 1; x}\right]  + 2^{d\lambda}\, h_{\ell - 
			\lambda, q - 2d - 1},
	\end{split}
\end{equation}
where $\tilde D_{0, \ell}^\lambda = \{x \in D_{0, \ell}^\lambda : 2^{-(\ell-\lambda)} (10N)^{-1}|x|_{\infty} \ge 2\}$. We now decouple the two events $H_{\ell-\lambda, q - d\lambda - 1; 0}$ and $H_{\ell-\lambda, q - d\lambda - 1; x}$ in \eqref{eq:first_recursion} by means of Proposition~\ref{prop:cond_decoup}, as follows. In the notation of \S\ref{subsec:RI}, we pick $r=30 N$ and choose sets $A$ and $U$ with $ B_r \subset A \subset U \subset B_{2r}$ such that the conclusions of Proposition~\ref{prop:cond_decoup} hold. Recall that these choices induce a corresponding process $(\mathcal{D}_u)_{u> 0}$ of excursions from $A$ to $\partial U$ between pairs of points in the clothesline $\mathcal{C}_u$, cf.~the discussion leading to \eqref{eq:IuDu}.

By inspection of \eqref{def:H_B}, referring to \eqref{eq:fin_energy_good} regarding the events $\widehat{F}_B=  \bigcap_{1 \le i \le 3} \widehat{F}^i_B$ and to \eqref{def:GBudelta}
 regarding $\widehat{G}_B(u_1,\delta_2) = \{ \mathcal{C}_{u_1} \in \Xi_{B}^{u, \delta_2 u^{-1}}\}$, 
 one infers that the event $H_{\ell-\lambda, q - d\lambda - 1; 0}$ is measurable relative to $(\mathcal D_{u_3 - \delta_2}, \mathcal D_{u_3 - \delta_1}, \mathcal D_{u_2 - \delta_1}, \mathcal D_{u_2}, \mathcal D_{u_1})$. Note that the parameters involved are listed in increasing order.
 Moreover, the event $H_{\ell-\lambda, q - d\lambda - 1; 0}^c$ is monotone in each of the five configurations involved. We illustrate this in the case of $\mathcal D_{u_1}$ and leave the others for the reader to check. The dependence on $\mathcal D_{u_1}$ comes through $\widehat{F}_{B'}^3$ in \eqref{eq:fin_energy_good}, where $B' \subset \Lambda$, which is clearly decreasing in $\mathcal D_{u_1}$, and through $\widehat{G}_{B'}(u_1,\delta_2) $. The latter only depends on $\mathcal{C}_{u_1}$ and is decreasing in $\mathcal{C}_{u_1}$ (and thus also in $\mathcal{D}_{u_1}$) on account of Remark~\ref{R:Caio}, items~\ref{rem:caio_monot} and~\ref{eq:monotonicity_Du}.
 
Returning to \eqref{eq:first_recursion}, assume now that $\ell-\lambda \ge 4$, which implies by definition of $\tilde D_{0, \ell}^\lambda $ that the supremum ranges over $x$ satisfying $|x|_{\infty} > 200 N$. It then follows by the afore measurability properties of $H_{\ell-\lambda, q - d\lambda - 1; 0}$ and translation invariance that the events $H_{\ell-\lambda, q - d\lambda - 1; 0}$ and $H_{\ell-\lambda, q - d\lambda - 1; x}$ are conditionally independent under $\mathbb{P}[\, \cdot \,  \vert \, \overline{\mathcal{C}}\,]$, where $\overline{\mathcal{C}}= (\mathcal C_{u_3 - \delta_2}, \mathcal C_{u_3 - \delta_1}, \mathcal C_{u_2 - \delta_1}, \mathcal C_{u_2}, \mathcal C_{u_1})$.
 
The sprinkling inherent to the decoupling will be parametrized by $\delta' > 0$, soon to be chosen proportional to $\delta$, and we assume that $\delta_2 - \delta_1 > 2\delta', \delta_1 > 2\delta'$ and $u_3 > 3\delta'$. Let $\mathcal{A}$ refer to the event in \eqref{eq:good_event_mul_excursion_bnd} with $k=5$ and parameters $(u_1,\dots u_5)$ appearing there chosen as $(u_3 - \delta_2, u_3 - \delta_1, u_2 - \delta_1, u_2, u_1)$, and $(\delta_1,\dots, \delta_5)$ as $\big(\tfrac{\delta'}{u_3 - 
	\delta_2}, \tfrac{\delta'}{u_3 - \delta_1}, \tfrac{\delta'}{u_2 - \delta_1}, \tfrac{\delta'}{u_2}, \frac{\delta'}{u_1}\big)$. Conditioning on $\overline{\mathcal{C}}$, it follows that for all $\ell-\lambda \ge 4$ and $x \in \tilde D_{0, \ell}^\lambda$, with $H_{\cdot}= H_{\ell-\lambda, q - d\lambda - 1 ;  \cdot}$ and $\xi=(u_1, u_2, u_3, \delta_1, \delta_2)$, 
\begin{multline*}
		\P\left[H_{ 0}(\xi) \cap H_{ x}(\xi)  \cap \{\overline{\mathcal{C}} \in \mathcal{A}\} \right] = \E\left[\mathbb{P} [ H_{0}(\xi) \, \vert \, \overline{\mathcal{C}} \, ] \,   \mathbb{P} [ H_{x}(\xi) \, \vert \, \overline{\mathcal{C}}\,]  \, 1_{\overline{\mathcal{C}} \in \mathcal{A}}  \right] \\ 
		\stackrel{\eqref{eq:mul_excursion_sandwich}}{\leq } 5\Cr{C:box_gap}e^{-c u_0 (\delta')^2 (2^{\ell-\lambda}N)^{\Cr{c:box_gap}}} +
		 \mathbb{P} [H_{0} (u_1 + \delta', u_2 - \delta', u_3 - 3\delta', \delta_1 - 2\delta', \delta_2 - 4\delta')] \cdot \mathbb{P} [H_{x}(\xi)],
\end{multline*}
where $u_0 = \min \{ u_3 - \delta_2, u_3 - \delta_1, u_2 - \delta_1, u_2, u_1\}$ and the inequality follows upon recalling the monotonicity features of $H_{0}$. To assist the careful reader in demystifying the parameters for the event $H_{0}$ in the second line, notice for instance that $H_{ 0}(\xi)$ is function of $\mathcal{D}_{u_3-\delta_1}$ through $(\widehat{F}_{B}^2)^c$ alone, cf.~\eqref{eq:fin_energy_good}, which is decreasing in $\mathcal{D}_{u_3-\delta_1}$, and this gets replaced by  the smaller configuration $\mathcal{D}_{u_3 - 3\delta'-(\delta_1 - 2\delta')}= \mathcal{D}_{u_3 - \delta_1 - \delta'} (\leq \mathcal{D}_{u_3-\delta_1})$, thus leading to an upper bound; a similar observation with opposite monotonicity applies to $\mathcal{D}_{u_3-\delta_2}$. Using \eqref{eq:good_event_mul_excursion_bnd} to bound $\P[\overline{\mathcal{C}} \notin \mathcal{A}]$ and feeding the resulting estimate into \eqref{eq:first_recursion}, it follows that for all $\ell-\lambda \ge 4$, any $ u_1 \ge u_2 \ge u_3 > \delta_2 > \delta_1 > 0$ and $\delta'> 0$ such that $\delta_2 - \delta_1 > 2\delta', \delta_1 > 2\delta'$ and $u_3 > 3\delta'$,
abbreviating $\xi = (u_1, u_2, u_3, \delta_1, \delta_2)$ and $\xi'= (u_1 + 
		\delta', u_2 - \delta', u_3 - 3\delta', \delta_1 - 2\delta', \delta_2 - 4\delta' )$,
\begin{equation}\label{eq:partial_diff1}
	\begin{split}
		h_{\ell, q} (\xi) \le  4^{d\lambda}&\big(h_{\ell-\lambda, q - d\lambda - 1}^2(\xi') + 6\Cr{C:box_gap}e^{-c u_0 (\delta')^2 (2^{\ell-\lambda}N)^{\Cr{c:box_gap}}}\big)  + 2^{d\lambda}\, h_{\ell-\lambda, q - 2d - 1}(\xi).
	\end{split}
\end{equation}	

We will now iterate \eqref{eq:partial_diff1}, starting from $u_1=u_2=u_3=u$, $(\ell, q) = (\ell_0, q_0)$ and $(\delta_1, \delta_2) = (\frac{\delta}{2}, \delta)$, as needed for the purpose of estimating the left-hand side of \eqref{eq:pellq_bnd}.
We set $\xi_{n}=(u_{1,n},u_{2,n}, u_{3,n}, \delta_{1,n}, \delta_{2, n} )$ for $n \geq 0$ with $u_{i,0}= u_i$, $\delta_{i,0}=\delta_i$, and $\xi_{n}=(\xi_{n-1})'$ as above \eqref{eq:partial_diff1} with increment parameter $\delta' = \delta'_n= \frac\delta{64n^2}$ for all $n \geq 1$. With these choices, the parameter values $\xi_n$ after round $n$ are given by $$\xi_n  = \textstyle(u + \Delta_n, u - \Delta_n, u - 3\Delta_n, \frac\delta 2 - 
2\Delta_n, \delta - 4\Delta_n),$$
with $\Delta_n = \tfrac{\delta}{64}\sum_{1 \le i \le n} \frac1 {i^2}.$ In particular, recalling that $u \geq 2 \delta$ (see above \eqref{eq:sumstep_bad}), it is always the case that $u_{1,n} \ge u_{2,n} \geq u_{3,n} > 3\delta_{1}' >3\delta_{n+1}'$, $\delta_{1, n} = \frac\delta 2 - 
2\Delta_n > 2\delta_n'$ and $\delta_{2, n} - \delta_{1, n} = \frac\delta 2 - 
2\Delta_n > 2\delta_n'$. 

Observe that the quantity we wish to bound in \eqref{eq:pellq_bnd} now simply reads $h_{\ell_0, q_0} (\xi_0)$. 
Moreover, setting $\ell_n = \ell_0- n\lambda$ and assuming that $\ell_0 -  n\lambda \geq 4$, \eqref{eq:partial_diff1} yields a bound for $h_{\ell_{n-1}, q} (\xi_{n-1})$ in terms of $h_{\ell_{n}, q - d\lambda - 1} (\xi_{n})$ and $h_{\ell_{n}, q - 2d - 1} (\xi_{n})$. We can naturally map each term involving $h_{\ell_n, \cdot}$ appearing at the outcome of the $n$-th round to an $n$-tuple $\mb m = (m_1, \ldots, m_n) \in \{0, 1\}^n$ encoding whether it was obtained from the {\em quadratic} term (encoded as $1$) or the {\em linear} term (encoded as 
$0$) at each round $ 1\leq k \leq n$. Thus, denoting $m = \sum_{1 \le k \le n} m_k$ for any $\mb m \in \{0, 1\}^n$, setting $q_{m, n} = q_0 - dm(\lambda-2) - (2d+1)n$, one obtains from \eqref{eq:partial_diff1} and an induction argument that, as long as $\ell_n = \ell_0- n\lambda \geq 4$,
\begin{multline}\label{eq:iterated_rel}
		h_{\ell_0, q_0}(\xi_0) 
		\le \sum_{\mb m \in \{0, 1\}^n} 4^{n(d\lambda+C)2^{m}}h_{\ell_n,\, q_{m, n}}^{2^m}(\xi_n)\\
		 + \sum_{0 \le k < n} \,\sum_{\mb m \in \{0, 1\}^k} 4^{(k+1)(d\lambda+C)2^{m}}\exp\big\{-cu (k+1)^{-4}(2^{\ell_k-\lambda}N)^{\Cr{c:box_gap}}2^{m}\big\},
\end{multline}
where one also uses the simple fact that $(a + b + c)^{k} \le 3^{k}(a^{k} + b^{k} + c^{k})$ for all $a, b,c \ge 0$ and positive integers $k(=2^m)$ while expanding the powers arising from the squares in \eqref{eq:partial_diff1}.

We now derive a meaningful lower bound on $m$ for any non-vanishing term contributing to the first sum on the right-hand side of \eqref{eq:iterated_rel}, for suitable choice of $\lambda$. First notice that $h_{\ell, q}= 0$ whenever $q > d\cdot\ell$ by definition, see \eqref{eq:def-h-supercrit} (indeed, $n_{z, \ell} \leq 2^{d\ell}$) and hence the only terms that contribute 
to the first sum in \eqref{eq:iterated_rel} are precisely those $\mb m$ for which
\begin{equation}\label{eq:bndqmn}
	q_{m, n} - d\cdot\ell_n = (q_0 - d\ell_0) + d(n - m) (\lambda - 2) - n \le 0.
\end{equation}
Owing to the definition of $\ell_0$ in \eqref{eq:dyadic_inclusion} (recalling that $\Gamma = 
	2^{\lfloor (\log N)^{10b}\rfloor}$) and $q_0$ in \eqref{eq:pellq_bnd}, it follows that $n \le \frac{\ell_0}{\lambda}$ implies that $n \le \frac{q_0}2$ whenever $N \ge C$ and $\lambda \ge 8$. Therefore any $\mb m$ satisfying 
\eqref{eq:bndqmn} must also satisfy, for any $\lambda \ge 50d$, $N \ge C$ and $n \le 
\frac{\ell_0}{\lambda}$,
\begin{equation*}
	\begin{split}
		n - m \le \frac{\ell_0}{\lambda} \frac{1 - \frac{q_0 - n}{d\ell_0}}{1 - \frac{2}{\lambda}} 
				\,  \stackrel{(\lambda \ge 4)}{\le}\frac{\ell_0}{\lambda} \Big(1 - \frac{q_0}{2d\ell_0} + \frac{4}{\lambda}\Big) \, 
		\stackrel{\eqref{eq:dyadic_inclusion}, \eqref{eq:pellq_bnd}}{\le} \, \frac{\ell_0}{\lambda} \Big(1 - \frac{1}{6d} + \frac{4}{\lambda}\Big) \stackrel{(\lambda \ge 50 d)}{\le} \frac{\ell_0}{\lambda} \Big(1 -\frac{1}{12d}\Big)
	\end{split}
\end{equation*}
where in the third step we also used the fact that $N$ is large and that $\lim_N \frac{q_0}{\ell_0}=\frac12$. Fix
\begin{equation}\label{def:n}
	\textstyle	n = \big \lceil \frac{\ell_0}{\lambda}\big(1 -\frac{1}{12d}\big)^{\frac12}\big\rceil, \quad \lambda = 50 d
\end{equation}
so that in view of our previous discussions there exists $\Cl[c]{c:quad_frac} > 0$ 
such that, when $N \geq C$, one has for all nonzero summands in the first sum on the right-hand side of \eqref{eq:iterated_rel} that	
\begin{equation}\label{e:m-islarge}
m \ge \Cr{c:quad_frac} n.
\end{equation}
Moreover, \eqref{def:n} implies that
\begin{equation}\label{e:l_n large}
\ell_n = \ell_0 - n\lambda	\ge \Cr{c:quad_frac} \ell_0.
\end{equation}
In particular, \eqref{e:l_n large} implies that $\ell_n \geq 4$ when $N \geq C$ since $\ell_0 \to \infty$ as $N \to \infty$. Thus \eqref{eq:iterated_rel} is in force for the choices \eqref{def:n}.

\begin{claim} \label{Cl-2-supercrit}For any $b \geq1$, $\gamma \ge C b$, $v, \delta >0$, $u \in [2\delta,  v-2\delta]$, $N \geq 1$ and $0 \leq m \leq n$ with $n$ as in \eqref{def:n}, one has
\begin{equation}\label{e:supercrit_1-bad}h_{\ell_n,\, q_{m, n}}(\xi_n) \le 
C e^{-c(\delta,v) (\log N)^{12b}}.\end{equation}
\end{claim}
	
Let us first conclude the proof of \eqref{eq:pellq_bnd} assuming Claim~\ref{Cl-2-supercrit} to hold. Plugging the estimate for $h_{\ell_n,\, q_{m, n}}$ into \eqref{eq:iterated_rel} 
	and subsequently using the fact that $\ell_0 = 1 + 2\lfloor(\log N)^{10b}\rfloor$, \eqref{def:n} and \eqref{e:m-islarge} to deal with the first sum on the right-hand side of~\eqref{eq:iterated_rel}, as well as \eqref{e:l_n large} and \eqref{def:n} to deal with second one, one obtains that, for suitable $c=c(\delta,v)$,
	\begin{multline*}
			h_{\ell_0, q_0}(\xi_0) \le 
			2^n\exp\big\{-(\log N)^{12b}(c - C(\log N)^{-2b})2^{c \ell_0}\big\} \\
			+ 4^n  \exp\big\{-cu n^{-4}2^{c\ell_n}N^c + C(n+1)\big\}
			\le 
			C \exp\big\{-e^{c(\log N)^{8b}}\big\},
	\end{multline*}
	whence \eqref{eq:pellq_bnd}. We now return to the 
	
	\begin{proof}[Proof of Claim~\ref{Cl-2-supercrit}] Recalling that $\xi_n= (u_{1,n},u_{2,n}, u_{3,n}, \delta_{1,n}, \delta_{2, n} )$, it follows in view of \eqref{eq:def-h-supercrit}, \eqref{def:H_B} and \eqref{eq:fin_energy_good}, since $2^{q_{m,n}}> 0$ (note that $q_{m,n}$ could be negative), that
	\begin{multline*}
			h_{\ell_n,\, q_{m, n}}(\xi_n) 
			\le C N^{d+1} \bigg(\sup_{\substack{B 
					\subset \tilde \L,\\ {\rm rad}(B) \ge 16{r_0}}} \P\Big[\big(\widehat{ G}_B(u_{1,n}, \delta_{2,n})\big)^c\Big] + \sup_{B \subset \tilde \L} \P\Big[\big(\widehat{F}_{B}^1(u_{2,n}, \delta_{1, n})\big)^c\Big] \\
			+ \sup_{B \subset \tilde \L} \P\Big[\big(\widehat{{F}}_{B}^2(u_{3,n}, \delta_{1, n}, \delta_{2, n})\big)^c\Big] + \sup_{B \subset \tilde \L} \P\Big[\big(\widehat{{F}}_{B}^3(u_{1,n})\big)^c\Big]\bigg).
	\end{multline*}
	Hence, it suffices to bound each of the four probabilities probabilities individually by the right-hand side of \eqref{e:supercrit_1-bad}. Recall $\widehat{{F}}_{B}^3(u_{1,n})$ from \eqref{eq:fin_energy_good}. Since $u_{1,n} = u + \Delta_n$, $2\delta \leq u  \leq v- 2\delta$ by assumption and one has the bound $\Delta_n \le \tfrac{\delta}{32}$ valid for all $n$ (see below \eqref{eq:partial_diff1}), it follows that $\tfrac{\delta}{2} \leq u_{1,n}  \le v$. Thus, using standard bounds on the tail probabilities for Poisson and geometric random variables and recalling that ${r_0}= \lfloor (\log N)^{\gamma} \rfloor$ it follows that the probability of $\widehat{{F}}_{B}^3(u_{1,n})^c$ is bounded by $Ce^{-c(\delta,v)(\log N)^{12b}}$ whenever $\gamma \geq Cb$. 
	
	A similar bound can be derived for $\P\big[(\widehat{{F}}_{B}^2(u_{3,n}, \delta_{1, n}, \delta_{2, n}))^c\big]$ by means of \eqref{eq:I_u} since $\mathcal I^{u_{3,n} - \delta_{1,n}} \setminus\mathcal I^{u_{3,n} - \delta_{2,n}}$ has the same as $\mathcal{I}^{\delta_{2, n} - \delta_{1, n}}$, upon observing that $\delta_{2, n} - \delta_{1, n} = \frac\delta2 - 6 \Delta_n \ge \tfrac{7}{16}\delta$. As for the event $\big(\widehat{F}_{B}^1(u_{2,n}, \delta_{1, n})\big)^c$ (see~\eqref{def:ABg} and \eqref{eq:fin_energy_good}), one notices that
		\begin{equation*}
		\begin{split}
\bigcap_{\substack{x, y \in \mathcal I^{u - \delta} 
					\, \cap \, D,\\ D \subset A(B, {r_0}),\, {\rm rad}(D) = {r_0}}}\big\{  \lr{}{\mathcal I^u \, \cap\, D_{{r_0}}}{x}{y}  \big\} \cap  \bigcap_{\substack{D' \subset A(B, {r_0}),\\ {\rm rad}(D') = \lfloor\tfrac{{r_0}}{2}\rfloor}}\left\{  \mathcal I^u \cap D' \neq \emptyset \right\} \subset {\widehat F}_{B}^1(u, \delta).
		\end{split}
	\end{equation*}
(with $D,D'$ ranging over $\ell^{\infty}$-boxes with the prescribed features). The desired bound on the probability of $\widehat{F}_{B}^1(u_{2,n}, \delta_{1, n})$ now follows by combining a union bound, \eqref{eq:I_u} and \cite[Proposition~1]{MR2819660} (see also \cite[(5.4) and (5.20)]{DPR22} for more precise estimates).

Lastly, with regards to $\widehat{ G}_B(u_{1,n}, \delta_{2,n})^c$, one has due to \eqref{def:GBudelta}, writing $\bar{\delta}_{2,n} = (\delta_{2,n}-\Delta_n) u_{1,n}^{-1} (\geq c\delta)$,
$$
\widehat{G}_B(u_{1,n},\delta_{2,n}) = \big\{ \mathcal{C}_{u_{1,n}} \in \Xi_{B}^{u, \delta_{2,n} u^{-1}}\big\} \supset \big\{ \mathcal{C}_{u_{1,n}} \in \Xi_{B}^{u_{1,n}, \bar{\delta}_{2,n} }\big\} = \mathcal G_{B}^{u_{1,n},\bar{\delta}_{2,n} },
$$
and the desired bound for $\P[(\widehat{ G}_B(u_{1,n}, \delta_{2,n})^c]$ follows by applying \eqref{eq:corcond_decouple} with $r=r_0 (= \lfloor (\log N)^{\gamma} \rfloor)$. 
 \end{proof}
With Claim~\ref{Cl-2-supercrit} proved, the proof of Lemma~\ref{L:bound_bad} is complete.
\end{proof}	

\begin{remark}[$\gamma$ vs.~$\gamma_M$]\label{R:supercrit-badpts}
We now briefly take a look back at the proof of Lemma~\ref{L:bound_bad}. The probability for just one (fixed) box $\Lambda_k$ to have index $k$ in the bad set $\widetilde{\mathcal{K}}_b$, see \eqref{e:K_b}, is already quite costly. Indeed this roughly amounts to the bound obtained in the proof of Claim~\ref{Cl-2-supercrit}, and can be seen to be of order $\exp\{-c(\delta)(\log N)^{c\gamma}\}$ with $\gamma$ the exponent defining $r_0$, cf.~\eqref{eq:gamma}. On the other hand, finding one such point inside $\widetilde{\mathcal{K}}_b$ by means of a union bound yields a combinatorial complexity $M(N)^C$, so for this to produce a meaningful upper bound for the left-hand side of \eqref{eq:sumstep_bad}, one would need $\gamma$ to be large compared with $\gamma_M$. Unfortunately, the opposite is really underlying our finite energy mechanism: indeed, the constant $C$ in the exponent on the right-hand side of \eqref{eq:decay_disconnect} can be traced back to be proportional to $\gamma$ (which has been fixed at the beginning of that proof), and it is crucial for the number of `good' indices $|\widetilde{\mathcal{K}}_g|$ to compensate for the resulting cost $e^{-(\log N)^{C}}$ in \eqref{eq:decay_disconnect}. However this number is at most $K$, see above \eqref{prop:a}, which is never going to be larger than linear in the size $M(N)$, thus forcing $\gamma_M \geq C\gamma$.
\end{remark}

\subsection{Proof of Lemma~\ref{cor:ample_connection}}\label{sec:K_g}

Recall that $[K]= \{ 1,\dots,K\}$ and $K = \lceil \frac{ \sqrt{M} }{400 N}\rceil$, see above \eqref{prop:a}, and that the sets $\mathcal{K}_b$, $\mathcal{K}_g$, see \eqref{e:K_b}, \eqref{e:K_g}, implicitly depend on $\gamma$ (through $r_0$ and the bridge, see around \eqref{eq:gamma}), $u$ and $\delta$.
Lemma~\ref{cor:ample_connection} will follow readily from the following result

\begin{lem}
\label{L:explore_bad} For all $\gamma > 10$, $b \ge 1$, $v>0$ as in the statement of Proposition~\ref{prop:uniq}, $\delta >0$, $u \in [2\delta,  v-2\delta]$ and $N \ge C(\delta,v)$, $K_g, K_b \subset [K]$ and $k \in [K]$ with 
$K_g\cap K_b=\emptyset$ and $k \notin K_g\cup K_b$, 
\begin{multline}
\label{eq:1step}
\P_{\mathscr C}[\mathcal{K}_g = K_g \cup \{k \},\, \mathcal{K}_b=K_b] \\
\geq (D_0)^{-1} \, \P_{\mathscr C}[\mathcal{K}_g =K_g,\, \mathcal{K}_b =K_b] - D_1 \, \P_{\mathscr 
	C}[\mathcal{K}_g = K_g,\, \mathcal{K}_b=K_b \cup \{k\}],
\end{multline}
where $D_0 = \exp\{ C(\delta,v)(\log eN)^{\Cl{explore_bad}}\}$ 
and $D_1 = \Cr{explore_bad}\log\log e^2N$.
\end{lem}
We first suppose that Lemma~\ref{L:explore_bad} holds and give the short
\begin{proof}[Proof of Lemma~\ref{cor:ample_connection}] 
Throughout we assume that $\gamma > 10$ and that $\delta,u,v$ fulfill the assumptions of Lemma~\ref{cor:ample_connection}, or, equivalently, Lemma~\ref{L:explore_bad}.
Fix non-negative integers $m$ and $n$ such that $n + m \leq K$.  Summing 
\eqref{eq:1step} over all disjoint subsets $K_g$ and $K_b$ of $[K]$ with cardinalities $m$ and at most $n$, respectively, and letting $k 
\in [K] \setminus (K_g \cup K_b)$ which is a set with cardinality at least $K - n - m$, one obtains that for all $N \geq C(\delta, v)$ and $b \geq 1$,
\begin{multline}\label{eq:sumstep2}
(m+1)\,\P_{\mathscr C}[|\mathcal{K}_g| = m + 1, |\mathcal{K}_b| \leq n]  \\
 \geq  {D_0}^{-1} (K - n - m) \, \P_{\mathscr C}[|\mathcal{K}_g| = m, |\mathcal{K}_b| \leq n] 
- D_1(n+1) \, \P_{\mathscr C}[|\mathcal{K}_g| = m , |\mathcal{K}_b| \leq n+1], 
\end{multline}
where the factor $m + 1$ (and similarly $n+1$) arises due to the number of  ways of 
decomposing $\mathcal{K}_g$ with $|\mathcal{K}_g| = m + 1$ into a set of cardinality $m$ and a singleton. 
In view of \eqref{eq:sumstep}, let
\begin{equation}
		\label{eq:nk_range}
		n = \lfloor 4K \beta \rfloor \text{ and }\, m \leq 2n.
	\end{equation}
Since $K \ge \sqrt{M} / 400 N$, see above \eqref{prop:a}, and by definition of $M=M(N)$, see \eqref{eq:def_M}, it follows that for all $n$ and $m$ as in \eqref{eq:nk_range} and all $ b \geq 1$, one has $K - n - m  \ge \frac K2$ whenever $N \ge C$. If in addition, $b \ge 2 \Cr{explore_bad}$, this implies that
\begin{equation}\label{eq:D0D1bnd}
\frac{D_0 D_1(n+1)}{(K - n - m)} \leq \frac12\, \text{ and } \, \frac{D_0\,(m+1)}{(K - n - m)} \leq \frac14\,  \text{ for $N \geq C(\delta,v)$}.
\end{equation}
Thus under \eqref{eq:nk_range} and for all $N \ge C(\delta,v)$, $b \ge 2 \Cr{explore_bad}$, \eqref{eq:sumstep2} and \eqref{eq:D0D1bnd} yield that
\begin{equation*}
\begin{split}
 \P_{\mathscr C}[|\mathcal{K}_g| = m] &\le \P_{\mathscr C} [ |\mathcal K_g| = m, |\mathcal K_b| \le n ] + 
 \P_{\mathscr C} [ |\mathcal K_b| > n ] \\
 &\le \frac14\P_{\mathscr C} [ |\mathcal K_g| = m + 1] + \frac12\P_{\mathscr C} [ |\mathcal K_g| = m] + \P_{\mathscr C} [ |\mathcal K_b| > n ]
\end{split}
	\end{equation*}
and consequently, 
\begin{equation}
\label{eq:sumstep3}
 \P_{\mathscr C}[|\mathcal{K}_g| = m] \leq 2^{-1}  \P_{\mathscr C}[|\mathcal{K}_g| = m + 1] + 2 
 \P_{\mathscr C}[ |\mathcal{K}_b| \geq n].
\end{equation}
Now iterating \eqref{eq:sumstep3} $n$ times for each choice of $m$ with $0 \le m < 2n$ and 
then summing the resulting inequalities over $m$ in the range $0 \le m < 2n$ readily implies the bound \eqref{eq:sumstep}.
\end{proof}

We now give the
\begin{proof}[Proof of Lemma~\ref{L:explore_bad}] For $K_g, K_b \subset [K]$ and $k \in [K]$ with 
$K_g\cap K_b=\emptyset$, $k \notin K_g\cup K_b$, let 
\begin{equation*}
H \stackrel{{\rm def}.}{=} \big\{\mathcal K_g  \cap ([K] \setminus \{k\}) = K_g,\  \mathcal K_b  \cap ([K] \setminus \{k\}) = K_b\big\}.
\end{equation*}
It then follows from the definitions of $\mathcal K_b$ and $\mathcal K_g$ in \eqref{e:K_b} and \eqref{e:K_g}, that
\begin{multline}\label{eq:explore_bad_1}
\P_{\mathscr C}[\mathcal{K}_g = K_g \cup \{k\},\, \mathcal{K}_b = K_b]  \\ = \P_{\mathscr C}\left[ E_k, \text{$\L_k$ is good}, H\right] = 
\P_{\mathscr C}\left[ E_k, H\right] - \P_{\mathscr C}\left[ E_k, \text{$\L_k$ is bad}, H\right]\\
\geq \P_{\mathscr C}\big[ E_{j, k} ,  \, 1 \le j < J_k ,  \, H \big] - \P_{\mathscr 
C}\left[\mathcal{K}_g = K_g,\, \mathcal{K}_b=K_b \cup \{k\}\right],
\end{multline}	
where the last line follows upon recalling $E_k$ from \eqref{def:mthcalEB} and observing that the last probability is equal to $ \P_{\mathscr C}\left[\text{$\L_k$ is bad}, H\right]$.
We are going to show that under the assumptions on $\gamma,b,v,\delta, u$ of Lemma~\ref{L:explore_bad} (henceforth tacitly assumed), for each $1 \le J < J_k$ and $N \ge C(\delta,v)$, one has
\begin{multline}\label{eq:explore_bad_single_rnd_bnd}
\P_{\mathscr C}\big[ E_{j, k}, \, 1 \le j \le J, \, H \big]\\ 
\ge 
e^{-C(\delta,v) (\log N)^{2\Cr{C:bridge_sup}}} \,  \P_{\mathscr C}\big[ E_{j, k},  \, 1 \le j < J, \, 
H \big] - 
C' \, \P_{\mathscr C}[\mathcal{K}_g = K_g,\, \mathcal{K}_b=K_b \cup \{k\}]
\end{multline}
(see the paragraph preceding \eqref{eq:gamma} regarding $\Cr{C:bridge_sup}$). Indeed, \eqref{eq:1step} 
follows from \eqref{eq:explore_bad_1} and \eqref{eq:explore_bad_single_rnd_bnd} after the latter is 
iterated over all $1 \le J < J_k$ as $\{\mathcal{K}_g = K_g,\, \mathcal{K}_b = K_b\} \subset H$
and $J_k \le C \log \log N$ by property~\eqref{B4} and our choice of parameters for the bridge (see above \eqref{e:supercrit-good}).

It thus remains to deduce \eqref{eq:explore_bad_single_rnd_bnd}. Recall from \eqref{def:mthcalEB} that $E_{J,k}$ is an intersection over various connection events $E_B$ for $B$ ranging over boxes in $\mathbb B_{J}^k$. We aim to decouple the events $E_B$'s successively by conditioning on `all but the immediate vicinity of $B$'. We proceed to make this precise. We assume to this effect that $n = |\mathbb B_{J}^k|$ and fix an arbitrary ordering $B_1, B_2, \ldots, B_n$ of the elements of $\mathbb B_{J}$. Consider an arbitrary $m$ with $1\leq m \leq n$ and assume that $B_m= B(x,r)$ for some $x \in \Z^d$ and $r \geq 1$. With $\xi=1 - \Cr{c:box_gap}$ as defined above \eqref{e:supercrit-good}, let $\mathcal{F}_m = \mathcal{F}_{B_{4 M} \setminus B(x, r + 
\lceil r^{\xi}\rceil)}$, where $\mathcal{F}_K =\sigma(\eta_j(x), \,\eta_0(x), \, x \in K)$ with $\eta_{\cdot}$ as in \eqref{omega_i} and the index $j$ referring to the (partially sprinkled) configuration used to define $\widetilde{\mathcal{U}}$, see \eqref{Utilde}. We first claim that
\begin{equation}
\label{e:U-tilde_meas}
\{\widetilde{\mathcal U} = \{ \mathscr C\}\} \in \mathcal{F}_m.
\end{equation}
This is because, due to \eqref{eq:Uij}, \eqref{eq:Ui}, \eqref{eq:U_ijbis} and \eqref{Utilde}, $\widetilde{\mathcal U}$ is obtained by grouping (i.e.~forming equivalence classes of) clusters in $\mathcal{C}$, see \eqref{eq:definiq1}, which are clusters of the boundary
$\partial B_{4 M}$ for the configuration $(\eta_0(x): x \in B_{4M})$, according to a grouping rule that depends on $\eta_j$, which in view of \eqref{omega_i}, is only affecting $\eta_0$ outside $V_{2j}=B_{4M-2j \sqrt{M}}$. But since $\Lambda_k$ is well inside $V_{2j}$ by \eqref{Uconnect} and \eqref{prop:a}, so is the box $B(x, r + 
\lceil r^{\xi}\rceil)$ by construction of $\mathbb{B}^k$, and \eqref{e:U-tilde_meas} follows using~\eqref{B1}.

Now, recall from above \eqref{def:ABg} that $G_B = \mathcal G_B^{u, \delta u^{-1}}$ and abbreviate
$$
E_{B_m}^- =  \big\{ E_{j, k} , \, 1 \le j< J, \, E_{B_l}, \, 1\leq l  < m \big\}.
$$
We are now going to argue that whenever $N \geq C(\delta,v)$,
\begin{equation}
\label{eq:explore_bad_conditional_connect}
\begin{split}
	\P_{\mathscr C}&\big[\, E_{B_m}^-, \, E_{B_m}, \, G_{B_m}, \, H \, \big| \, \sigma(\mathcal{F}_m, 1_{G_{B_m}}, 1_{H} )\big] \ge  e^{-C(\delta,v)(\log N)^{2}} \,
	1_{\{ E_{B_m}^-, \, G_{B_m}, \, H\}}.
\end{split}
\end{equation}
To see this, first note that due to \eqref{e:U-tilde_meas}, one can replace $\P_{\mathscr C}$ by $\P$ on the left-hand side of \eqref{eq:explore_bad_conditional_connect}.
Next, one observes that the event $\{ E_{B_m}^-, \, G_{B_m}, \, H\}$  is measurable relative to $\sigma(\mathcal{F}_m, 1_{G_{B_m}}, 1_{H})$ in view of \eqref{def:mthcalEB},~\eqref{B1} and property~\eqref{prop:c} for the family $\Lambda$ of boxes. Then, \eqref{eq:explore_bad_conditional_connect} simply follows from Proposition~\ref{prop:cond_decoup}, item ii), which, as we now explain, yields that 
$$\P\big[ E_{B_m}| \sigma(\mathcal{F}_m, 1_{G_{B_m}}, 1_{H} )] 1_{G_{B_m}}
\stackrel{\eqref{def:GBudelta}, \eqref{eq:twopointsbound}}{\geq}
e^{-C(\delta,v)(\log N)^{2}}1_{G_{B_m}}.$$
Here, in applying \eqref{def:GBudelta}, we used that $G_B = \mathcal G_B^{u, \delta u^{-1}}$, see \eqref{e:-dec-good}, and that the additional conditioning on $(\mathcal{F}_m, 1_{H})$ does not affect the conditional probability as 
$\mathcal I^u \cap B_m$ is independent of $\sigma(\mathcal{F}_m, 1_{H})$ conditionally on 
$\mathcal C_u$ (see also the discussion preceding 
\cite[(7.40)]{RI-I} for a similar argument). As for \eqref{eq:twopointsbound}, we used the fact that $u(1+ \delta u^{-1})= u+ \delta \leq v- \delta$ by assumption on $u$ and picked
\begin{equation}\label{eq:gamma_conn}
\kappa= \kappa(v,\delta)= \inf_r \P\big[\lr{}{\mathcal V^{v-\delta}}{B_r}{\partial B_{2r}}\big]
\end{equation} in the context of Lemma~\ref{lem:twopointsbound}. With theses choices, the lower bound $e^{-C(\delta,v)(\log N)^{2}}$ uniform in $u \leq v-2\delta$ follows immediately by application of \eqref{eq:twopointsbound} since $B_m$ has radius at most $CN$, provided we argue that the hypothesis $\kappa >0$ in \eqref{eq:cond-u-twopointsbound} holds. This follows in turn from our assumption on $v$, by which \eqref{eq:disc-cond} holds for some $\alpha > 0$: indeed if $\kappa$ in 
\eqref{eq:gamma_conn} vanishes then by \cite{MR2744881} one knows that $\P[\lr{}{\mathcal V^{v}}{0}{\partial B_{r}}] \leq C(v)\exp\{-r^c\}$, which precludes \eqref{eq:disc-cond} via a straightforward union bound (the relevant disconnection probability tends to $1$, rather than $0$, let alone with some speed, as required in \eqref{eq:disc-cond}).
Overall, \eqref{eq:explore_bad_conditional_connect} thus follows.

Using \eqref{eq:explore_bad_conditional_connect}, we now complete the proof of 
\eqref{eq:explore_bad_single_rnd_bnd}. To this end, we first notice that
\begin{multline}
	\label{eq:explore_bad_lower_bnd_caio2}
		\P_{\mathscr C}\big[ E_{B_m}^-, \,  G_{B_m}, \, H\big] 		\ge \P_{\mathscr C}\big[ E_j, \, 1 \le j< J , \,  
		E_{B_l}, \, 1 \leq l < m, \, H\big] - \P_{\mathscr C}[G_{B_m}^c, H]\\ 
		\stackrel{\eqref{e:supercrit-good}, \eqref{e:K_b}}{\ge} \P_{\mathscr C}\big[ E_j, \, 1 \le j< J , \,  
		E_{B_l}, \, 1 \leq l < m, \, H\big] - \P_{\mathscr 
			C}\left[\mathcal{K}_g = K_g,\, \mathcal{K}_b=K_b \cup \{k\}\right].
\end{multline}
Now, we integrate \eqref{eq:explore_bad_conditional_connect}, combine it with 
\eqref{eq:explore_bad_lower_bnd_caio2} and then iterate the resulting inequality over all $ m \le n = 
|\mathbb B^k_{J}|$ to deduce \eqref{eq:explore_bad_single_rnd_bnd}. Using the bound
$|\mathbb B_J^k| \le C( \log N)^{\Cr{C:bridge_sup} + 1}$ implied by property~\eqref{B4} and our choice of parameters (see the paragraph above \eqref{e:supercrit-good}), the bound~\eqref{eq:explore_bad_single_rnd_bnd} follows.
\end{proof}

\begin{remark}[Enhanced gluing property] \label{R:final}
 We conclude by collecting the following result, which, as explained below, follows by inspecting our arguments in Sections \ref{Sec:surgery}-\ref{Sec:2lemmas} and is interesting in its own right. For $r \geq 1$, $u,v >0$ and with $M=M(r)$, let
 \begin{equation}\label{eq:GLUE-1}
\widetilde{E} = \widetilde{E}(r, u,v) =\widetilde{E}_1(r,u) \cap \widetilde{E}_2(r,u,v)
\end{equation}
where
\begin{align*}
&\widetilde{E}_1= \left\{\begin{array}{c} \text{each translate $B$ of $B_r$ such that $B \subset B_{4M}$} \\
\text{intersects a cluster $C \subset \mathcal{V}^u$  with $\text{diam}(C) \geq 20M$} \end{array} \right\} \\ 
&\widetilde{E}_2 =\left\{\begin{array}{c}\text{there exist two clusters in $\mathcal V^{u}\cap B_{4M}$ crossing 
$B_{4M}\setminus B_{2M}$} \\ \text{that are not connected to each other in $\mathcal V^{v}\cap B_{4M}$} 
\end{array}\right\}.
 \end{align*}
 The event $\widetilde{E}$ thus expresses an `existence without (weak) uniqueness' property, cf.~\eqref{e:strong-perco}.
 \begin{prop} \label{P:glue-boost}
 For all $\gamma_M \geq \Cr{c:gamma_M}$, if $v>0$ is such that 
\begin{equation} \label{eq:glue-cond-v}
\inf_r  \P\big[\lr{}{\mathcal V^{v}}{B_r}{\partial B_{2r}}\big] >0,
\end{equation} 
then for all $ r \geq 1$, $\delta > 0$, and $u$ with $2\delta \leq u \leq v-2\delta$,  
\begin{equation}
\label{eq:uniqreduct_new1bis}
\begin{split}
&\P[ \widetilde{E}(r,u,u-\delta) ] \leq C\exp \big\{- c \, e^{c (\log M(r))^{c'}}\big\}.
\end{split}
\end{equation}
with constants $C,c$ possibly depending on $\gamma_M$, $v$ and $\delta$.
 \end{prop}
In particular, the condition \eqref{eq:glue-cond-v} on $v$ relaxes the assumption on $v$ inherent to Sections~\ref{sec:supercritical} onwards, which is tailored to the needs of Theorem~\ref{T:MAIN}. Proposition~\ref{P:glue-boost} can be obtained by following essentially the line of argument in the proof of Proposition~\ref{prop:uniq} upon making the following observations. The event $\widetilde E$ in \eqref{eq:GLUE-1} is roughly equivalent (in fact it implies) the event $A\cap\{{U}_{\lfloor\sqrt{M}\rfloor} >1\}$ appearing in \eqref{eq:uniqueproof}. The latter is then bounded by \eqref{eq:reduct2} in the course of proving Proposition~\ref{prop:uniq} with the help of Lemma~\ref{lem:uniqreduct}, the key gluing lemma. Crucially, the latter continues to hold under the modified assumption \eqref{eq:glue-cond-v} on $v$. Indeed, upon inspecting Sections~\ref{Sec:surgery}-\ref{Sec:2lemmas}, one sees that the only place where this is used is in the proof of Lemma~\ref{L:explore_bad} to exhibit the cost \eqref{eq:explore_bad_conditional_connect} of re-constructing a piece of path inside a box of the bridge at a not too degenerate cost, uniform in $u (\leq v-2\delta)$. As explained around \eqref{eq:gamma_conn}, this follows by combining conditional decoupling (which always holds) and a connectivity lower bound from Lemma~\ref{lem:twopointsbound}, which only requires \eqref{eq:glue-cond-v}. 
\end{remark}

\bigskip

\noindent \textbf{Acknowledgements.} This work was initiated at IH\'ES. It has received funding from the European Research Council 
(ERC) under the European Union's Horizon 2020 research and innovation programme, grant 
agreement No.~757296. AT acknowledges support by grants
``Projeto Universal'' (406250/2016-2) and ``Produtividade em Pesquisa'' (304437/2018-2) from
CNPq and ``Jovem Cientista do Nosso Estado'', (202.716/2018) from FAPERJ. The research of FS is currently supported by the ERC under the European Union's Horizon 2020 research and innovation programme, grant agreement No.~851565. PFR thanks the IMO in Orsay (including support from ERC grant agreement No.~740943) for its hospitality during the final stages of this project. SG’s research is supported by the SERB grant 
SRG/2021/000032, a grant from the Department of Atomic Energy, Government of India, under project 
12–R\&D–TFR–5.01–0500 and in part by a grant from the Infosys Foundation as a member of the 
Infosys-Chandrasekharan virtual center for Random Geometry.  HDC acknowledges funding from the NCCR SwissMap, the Swiss FNS, and 
the Simons collaboration on localization of waves. SG, PFR, FS and AT all thank the University of Geneva for its hospitality on several occasions.

\bibliography{biblicomplete}

\begin{thebibliography}{10}

\bibitem{AizKesNew87}
M.~Aizenman, H.~Kesten, and C.~M. Newman.
\newblock Uniqueness of the infinite cluster and continuity of connectivity
  functions for short and long range percolation.
\newblock {\em Comm. Math. Phys.}, 111(4):505--531, 1987.

\bibitem{CaioSerguei2018}
C.~Alves and S.~Popov.
\newblock Conditional decoupling of random interlacements.
\newblock {\em ALEA Lat. Am. J. Probab. Math. Stat.}, 15(2):1027--1063, 2018.

\bibitem{AntPis96}
P.~Antal and A.~Pisztora.
\newblock On the chemical distance for supercritical {B}ernoulli percolation.
\newblock {\em Ann. Probab.}, 24(2):1036--1048, 1996.

\bibitem{MR3634283}
I.~Benjamini and V.~Tassion.
\newblock Homogenization via sprinkling.
\newblock {\em Ann. Inst. Henri Poincar\'{e} Probab. Stat.}, 53(2):997--1005,
  2017.

\bibitem{zbMATH02164755}
T.~Bodineau.
\newblock Slab percolation for the {Ising} model.
\newblock {\em Probab. Theory Relat. Fields}, 132(1):83--118, 2005.

\bibitem{cerf2015}
R.~Cerf.
\newblock A lower bound on the two-arms exponent for critical percolation on
  the lattice.
\newblock {\em Ann. Probab.}, 43(5):2458--2480, 09 2015.

\bibitem{zbMATH07226362}
A.~Chiarini and M.~Nitzschner.
\newblock Entropic repulsion for the occupation-time field of random
  interlacements conditioned on disconnection.
\newblock {\em Ann. Probab.}, 48(3):1317--1351, 2020.

\bibitem{https://doi.org/10.48550/arxiv.2107.06326}
D.~Contreras, S.~Martineau, and V.~Tassion.
\newblock Supercritical percolation on graphs of polynomial growth.
\newblock {\em Preprint}, arXiv:2107.06326, 2021.

\bibitem{MR1384041}
J.-D. Deuschel and A.~Pisztora.
\newblock Surface order large deviations for high-density percolation.
\newblock {\em Probab. Theory Related Fields}, 104(4):467--482, 1996.

\bibitem{drewitz2018geometry}
A.~Drewitz, A.~Pr{\'e}vost, and P.-F. Rodriguez.
\newblock Geometry of {G}aussian free field sign clusters and random
  interlacements.
\newblock {\em Preprint}, arXiv:1811.05970, 2018.

\bibitem{DPR22}
A.~Drewitz, A.~Pr{\'e}vost, and P.-F. Rodriguez.
\newblock Critical exponents for a percolation model on transient graphs.
\newblock {\em Invent. Math.}, 232(1):229--299, 2023.

\bibitem{MR3269990}
A.~Drewitz, B.~R\'ath, and A.~Sapozhnikov.
\newblock Local percolative properties of the vacant set of random
  interlacements with small intensity.
\newblock {\em Ann. Inst. Henri Poincar\'e Probab. Stat.}, 50(4):1165--1197,
  2014.

\bibitem{MR3390739}
A.~Drewitz, B.~R\'ath, and A.~Sapozhnikov.
\newblock On chemical distances and shape theorems in percolation models with
  long-range correlations.
\newblock {\em J. Math. Phys.}, 55(8):083307, 30, 2014.

\bibitem{DCGR20}
H.~Duminil-Copin, S.~Goswami, and A.~Raoufi.
\newblock Exponential decay of truncated correlations for the {I}sing model in
  any dimension for all but the critical temperature.
\newblock {\em Comm. Math. Phys.}, 374(2):891--921, 2020.

\bibitem{DCGRS20}
H.~Duminil-Copin, S.~Goswami, P.-F. Rodriguez, and F.~Severo.
\newblock Equality of critical parameters for percolation of {Gaussian} free
  field level sets.
\newblock {\em Duke Math. J.}, 172(5):839--913, 2023.

\bibitem{RI-III}
H.~Duminil-Copin, S.~Goswami, P.-F. Rodriguez, F.~Severo, and A.~Teixeira.
\newblock Finite-range interlacements and couplings.
\newblock {\em Preprint}, arXiv:2308.07303, 2023.

\bibitem{RI-I}
H.~Duminil-Copin, S.~Goswami, P.-F. Rodriguez, F.~Severo, and A.~Teixeira.
\newblock Phase transition for the vacant set of random walk and random
  interlacements.
\newblock {\em Preprint}, 2023.

\bibitem{DumTas15}
H.~{Duminil-Copin} and V.~Tassion.
\newblock A new proof of the sharpness of the phase transition for {B}ernoulli
  percolation and the {I}sing model.
\newblock {\em Commun.~Math.~Phys.}, 343(2):725--745, 2016.

\bibitem{GanGriRus88}
A.~Gandolfi, G.~Grimmett, and L.~Russo.
\newblock On the uniqueness of the infinite cluster in the percolation model.
\newblock {\em Comm. Math. Phys.}, 114(4):549--552, 1988.

\bibitem{gosrodsev2021radius}
S.~Goswami, P.-F. Rodriguez, and F.~Severo.
\newblock {On the radius of Gaussian free field excursion clusters}.
\newblock {\em Ann.~Probab.}, 50(5):1675 -- 1724, 2022.

\bibitem{GRS23+}
S.~Goswami, P.-F. Rodriguez, and Y.~Shulzhenko.
\newblock {\em In preparation}, 2023.

\bibitem{GriMar90}
G.~R. Grimmett and J.~M. Marstrand.
\newblock The supercritical phase of percolation is well behaved.
\newblock {\em Proc. Roy. Soc. London Ser. A}, 430(1879):439--457, 1990.

\bibitem{zbMATH05728593}
O.~H{\"a}ggstr{\"o}m and J.~Jonasson.
\newblock Uniqueness and non-uniqueness in percolation theory.
\newblock {\em Probab. Surv.}, 3:289--344, 2006.

\bibitem{zbMATH06257634}
X.~Li and A.-S. Sznitman.
\newblock A lower bound for disconnection by random interlacements.
\newblock {\em Electron. J. Probab.}, 19:26, 2014.
\newblock Id/No 17.

\bibitem{zbMATH01496108}
R.~Lyons and O.~Schramm.
\newblock Indistinguishability of percolation clusters.
\newblock {\em Ann. Probab.}, 27(4):1809--1836, 1999.

\bibitem{zbMATH07227743}
M.~Nitzschner and A.-S. Sznitman.
\newblock Solidification of porous interfaces and disconnection.
\newblock {\em J. Eur. Math. Soc. (JEMS)}, 22(8):2629--2672, 2020.

\bibitem{PopTeix}
S.~Popov and A.~Teixeira.
\newblock Soft local times and decoupling of random interlacements.
\newblock {\em J. Eur. Math. Soc. (JEMS)}, 17(10):2545--2593, 2015.

\bibitem{MR3568036}
E.~B. Procaccia, R.~Rosenthal, and A.~Sapozhnikov.
\newblock Quenched invariance principle for simple random walk on clusters in
  correlated percolation models.
\newblock {\em Probab. Theory Related Fields}, 166(3-4):619--657, 2016.

\bibitem{10.1214/ECP.v20-3734}
B.~R{\'a}th.
\newblock {A short proof of the phase transition for the vacant set of random
  interlacements}.
\newblock {\em Electronic Communications in Probability}, 20(none):1 -- 11,
  2015.

\bibitem{MR2819660}
B.~R\'ath and A.~Sapozhnikov.
\newblock On the transience of random interlacements.
\newblock {\em Electron. Commun. Probab.}, 16:379--391, 2011.

\bibitem{MR3650417}
A.~Sapozhnikov.
\newblock Random walks on infinite percolation clusters in models with
  long-range correlations.
\newblock {\em Ann. Probab.}, 45(3):1842--1898, 2017.

\bibitem{AHL_2022__5__987_0}
F.~Severo.
\newblock Sharp phase transition for {Gaussian} percolation in all dimensions.
\newblock {\em Annales Henri Lebesgue}, 5:987--1008, 2022.

\bibitem{severo2022uniqueness}
F.~Severo.
\newblock Uniqueness of unbounded component for level sets of smooth {G}aussian
  fields.
\newblock {\em Preprint}, arXiv:2208.04340, 2022.

\bibitem{MR2512613}
V.~Sidoravicius and A.-S. Sznitman.
\newblock Percolation for the vacant set of random interlacements.
\newblock {\em Comm. Pure Appl. Math.}, 62(6):831--858, 2009.

\bibitem{MR2744881}
V.~Sidoravicius and A.-S. Sznitman.
\newblock Connectivity bounds for the vacant set of random interlacements.
\newblock {\em Ann. Inst. Henri Poincar\'e Probab. Stat.}, 46(4):976--990,
  2010.

\bibitem{MR1232192}
E.~M. Stein.
\newblock {\em Harmonic analysis: real-variable methods, orthogonality, and
  oscillatory integrals}, volume~43 of {\em Princeton Mathematical Series}.
\newblock Princeton University Press, Princeton, NJ, 1993.

\bibitem{MR2680403}
A.-S. Sznitman.
\newblock Vacant set of random interlacements and percolation.
\newblock {\em Ann. of Math. (2)}, 171(3):2039--2087, 2010.

\bibitem{MR2891880}
A.-S. Sznitman.
\newblock Decoupling inequalities and interlacement percolation on
  {$G\times\mathbb{Z}$}.
\newblock {\em Invent. Math.}, 187(3):645--706, 2012.

\bibitem{MR3602841}
A.-S. Sznitman.
\newblock Disconnection, random walks, and random interlacements.
\newblock {\em Probab. Theory Related Fields}, 167(1-2):1--44, 2017.

\bibitem{zbMATH07114721}
A.-S. Sznitman.
\newblock On macroscopic holes in some supercritical strongly dependent
  percolation models.
\newblock {\em Ann. Probab.}, 47(4):2459--2493, 2019.

\bibitem{zbMATH07483480}
A.-S. Sznitman.
\newblock Excess deviations for points disconnected by random interlacements.
\newblock {\em Probab. Math. Phys.}, 2(3):563--611, 2021.

\bibitem{zbMATH07395560}
A.-S. Sznitman.
\newblock On the {{\(C^1\)}}-property of the percolation function of random
  interlacements and a related variational problem.
\newblock In {\em In and out of equilibrium 3: celebrating Vladas
  Sidoravicius}, pages 775--796. Cham: Birkh{\"a}user, 2021.

\bibitem{https://doi.org/10.48550/arxiv.2105.12110}
A.-S. Sznitman.
\newblock On the cost of the bubble set for random interlacements.
\newblock {\em Invent. Math.}, 233(2):903--950, 2023.

\bibitem{MR2525105}
A.~Teixeira.
\newblock Interlacement percolation on transient weighted graphs.
\newblock {\em Electron. J. Probab.}, 14:no. 54, 1604--1628, 2009.

\bibitem{MR2498684}
A.~Teixeira.
\newblock On the uniqueness of the infinite cluster of the vacant set of random
  interlacements.
\newblock {\em Ann. Appl. Probab.}, 19(1):454--466, 2009.

\bibitem{Tei11}
A.~Teixeira.
\newblock On the size of a finite vacant cluster of random interlacements with
  small intensity.
\newblock {\em Probab.~Theor.~Rel.~Fields}, 150(3):529--574, 2011.

\bibitem{MR2838338}
A.~Teixeira and D.~Windisch.
\newblock On the fragmentation of a torus by random walk.
\newblock {\em Comm. Pure Appl. Math.}, 64(12):1599--1646, 2011.

\end{thebibliography}
\bibliographystyle{abbrv}

\end{document}